\documentclass[reqno, 12pt]{amsart}
\pdfoutput=1
\makeatletter
\let\origsection=\section \def\section{\@ifstar{\origsection*}{\mysection}}
\def\mysection{\@startsection{section}{1}\z@{.7\linespacing\@plus\linespacing}{.5\linespacing}{\normalfont\scshape\centering\S}}
\makeatother

\usepackage{amsmath,amssymb,amsthm}
\usepackage{mathrsfs}
\usepackage{mathabx}\changenotsign
\usepackage{dsfont}
\usepackage{bbm}

\usepackage{xcolor}
\usepackage[backref=section]{hyperref}
\usepackage[ocgcolorlinks]{ocgx2}
\hypersetup{
	colorlinks=true,
	linkcolor={red!60!black},
	citecolor={green!60!black},
	urlcolor={blue!60!black},
}

\usepackage[open,openlevel=2,atend]{bookmark}

\usepackage[abbrev,msc-links,backrefs]{amsrefs}
\usepackage{doi}

\renewcommand{\PrintDOI}[1]{\doi{#1}}

\usepackage[T1]{fontenc}
\usepackage{lmodern}
\usepackage[babel]{microtype}
\usepackage[english]{babel}

\usepackage{geometry}
\numberwithin{equation}{section}
\numberwithin{figure}{section}

\usepackage{enumitem}
\def\rmlabel{\upshape({\itshape \roman*\,})}

\def\alabel{\upshape({\itshape \alph*\,})}

\def\nlabel{\upshape({\itshape \arabic*\,})}

\let\polishlcross=\l
\def\l{\ifmmode\ell\else\polishlcross\fi}

\def\tand{\ \text{and}\ }
\def\qand{\quad\text{and}\quad}

\let\emptyset=\varnothing
\let\setminus=\smallsetminus

\makeatletter
\def\moverlay{\mathpalette\mov@rlay}
\def\mov@rlay#1#2{\leavevmode\vtop{   \baselineskip\z@skip \lineskiplimit-\maxdimen
		\ialign{\hfil$\m@th#1##$\hfil\cr#2\crcr}}}
\newcommand{\charfusion}[3][\mathord]{
	#1{\ifx#1\mathop\vphantom{#2}\fi
		\mathpalette\mov@rlay{#2\cr#3}
	}
	\ifx#1\mathop\expandafter\displaylimits\fi}
\makeatother

\newcommand{\dcup}{\charfusion[\mathbin]{\cup}{\cdot}}

\DeclareFontFamily{U}  {MnSymbolC}{}
\DeclareSymbolFont{MnSyC}         {U}  {MnSymbolC}{m}{n}
\DeclareFontShape{U}{MnSymbolC}{m}{n}{
	<-6>  MnSymbolC5
	<6-7>  MnSymbolC6
	<7-8>  MnSymbolC7
	<8-9>  MnSymbolC8
	<9-10> MnSymbolC9
	<10-12> MnSymbolC10
	<12->   MnSymbolC12}{}
\DeclareMathSymbol{\powerset}{\mathord}{MnSyC}{180}

\usepackage{tikz}
\usetikzlibrary{calc,decorations.pathmorphing}
\usetikzlibrary{arrows,decorations.pathreplacing}
\pgfdeclarelayer{background}
\pgfdeclarelayer{foreground}
\pgfdeclarelayer{front}
\pgfsetlayers{background,main,foreground,front}

\usepackage{multicol}
\usepackage{subcaption}
\newcommand{\pedge}[9]{
	
	\ifx\relax#6\relax
	\def\qoffs{0pt}
	\else
	\def\qoffs{#6}
	\fi
		
	\def\phedge{
		($#1+#5!\qoffs!-90:#2-#5$) -- 
		($#2+#1!\qoffs!-90:#3-#1$) -- 
		($#3+#2!\qoffs!-90:#4-#2$) -- 
		($#4+#3!\qoffs!-90:#5-#3$) -- 
		($#5+#4!\qoffs!-90:#1-#4$) -- cycle}

	\coordinate (12) at ($#1!\qoffs!90:#2$);
	\coordinate (15) at ($#1!\qoffs!-90:#5$);
	\coordinate (23) at ($#2!\qoffs!90:#3$);
	\coordinate (21) at ($#2!\qoffs!-90:#1$);
	\coordinate (34) at ($#3!\qoffs!90:#4$);
	\coordinate (32) at ($#3!\qoffs!-90:#2$);
	\coordinate (45) at ($#4!\qoffs!90:#5$);
	\coordinate (43) at ($#4!\qoffs!-90:#3$);
	\coordinate (51) at ($#5!\qoffs!90:#1$);
	\coordinate (54) at ($#5!\qoffs!-90:#4$);

	\def\nphedge{
		(15) let \p1=($(15)-#1$), \p2=($(12)-#1$) in 
		arc[start angle={atan2(\y1,\x1)}, delta angle={atan2(\y2,\x2)-atan2(\y1,\x1)-360*(atan2(\y2,\x2)-atan2(\y1,\x1)>0)}, x radius=\qoffs, y radius=\qoffs] --
		(21) let \p1=($(21)-#2$), \p2=($(23)-#2$) in 
		arc[start angle={atan2(\y1,\x1)}, delta angle={atan2(\y2,\x2)-atan2(\y1,\x1)-360*(atan2(\y2,\x2)-atan2(\y1,\x1)>0)}, x radius=\qoffs, y radius=\qoffs] --
		(32) let \p1=($(32)-#3$), \p2=($(34)-#3$) in 
		arc[start angle={atan2(\y1,\x1)}, delta angle={atan2(\y2,\x2)-atan2(\y1,\x1)-360*(atan2(\y2,\x2)-atan2(\y1,\x1)>0)}, x radius=\qoffs, y radius=\qoffs] --
		(43) let \p1=($(43)-#4$), \p2=($(45)-#4$) in 
		arc[start angle={atan2(\y1,\x1)}, delta angle={atan2(\y2,\x2)-atan2(\y1,\x1)-360*(atan2(\y2,\x2)-atan2(\y1,\x1)>0)}, x radius=\qoffs, y radius=\qoffs] --
		(54) let \p1=($(54)-#5$), \p2=($(51)-#5$) in 
		arc[start angle={atan2(\y1,\x1)}, delta angle={atan2(\y2,\x2)-atan2(\y1,\x1)-360*(atan2(\y2,\x2)-atan2(\y1,\x1)>0)}, x radius=\qoffs, y radius=\qoffs] --
		cycle}

	\ifx\relax#7\relax
	\def\plwidth{1pt}
	\else
	\def\plwidth{#7}
	\fi
	
	\ifx\relax#9\relax
	\fill \nphedge;
	\else
	\fill[#9]\nphedge;
	\fi
	
	\ifx\relax#8\relax
	\draw[line width=\plwidth,rounded corners=\qoffs]\nphedge;
	\else
	\draw[line width=\plwidth,#8]\nphedge;
	\fi
}

\newcommand{\qedge}[7]{
	
	\ifx\relax#4\relax
	\def\qoffs{0pt}
	\else
	\def\qoffs{#4}
	\fi
	
	\def\qhedge{
		($#1+#3!\qoffs!-90:#2-#3$) --
		($#2+#1!\qoffs!-90:#3-#1$) --
		($#3+#2!\qoffs!-90:#1-#2$) -- cycle}

	\coordinate (12) at ($#1!\qoffs!90:#2$);
	\coordinate (13) at ($#1!\qoffs!-90:#3$);
	\coordinate (23) at ($#2!\qoffs!90:#3$);
	\coordinate (21) at ($#2!\qoffs!-90:#1$);
	\coordinate (31) at ($#3!\qoffs!90:#1$);
	\coordinate (32) at ($#3!\qoffs!-90:#2$);
	
	\def\nqhedge{
		(13) let \p1=($(13)-#1$), \p2=($(12)-#1$) in
		arc[start angle={atan2(\y1,\x1)}, delta angle={atan2(\y2,\x2)-atan2(\y1,\x1)-360*(atan2(\y2,\x2)-atan2(\y1,\x1)>0)}, x radius=\qoffs, y radius=\qoffs] --
		(21) let \p1=($(21)-#2$), \p2=($(23)-#2$) in
		arc[start angle={atan2(\y1,\x1)}, delta angle={atan2(\y2,\x2)-atan2(\y1,\x1)-360*(atan2(\y2,\x2)-atan2(\y1,\x1)>0)}, x radius=\qoffs, y radius=\qoffs] --
		(32) let \p1=($(32)-#3$), \p2=($(31)-#3$) in
		arc[start angle={atan2(\y1,\x1)}, delta angle={atan2(\y2,\x2)-atan2(\y1,\x1)-360*(atan2(\y2,\x2)-atan2(\y1,\x1)>0)}, x radius=\qoffs, y radius=\qoffs] --
		cycle}
	
	\ifx\relax#5\relax
	\def\qlwidth{1pt}
	\else
	\def\qlwidth{#5}
	\fi
	
	\ifx\relax#7\relax
	\fill \nqhedge;
	\else
	\fill[#7]\nqhedge;
	\fi
	
	\ifx\relax#6\relax
	\draw[line width=\qlwidth,rounded corners=\qoffs]\nqhedge;
	\else
	\draw[line width=\qlwidth,#6]\nqhedge;
	\fi
}

\newcommand{\redge}[8]{
	
	\ifx\relax#5\relax
	\def\qoffs{0pt}
	\else
	\def\qoffs{#5}
	\fi
	
	\def\rhedge{
		($#1+#4!\qoffs!-90:#2-#4$) -- 
		($#2+#1!\qoffs!-90:#3-#1$) -- 
		($#3+#2!\qoffs!-90:#4-#2$) -- 
		($#4+#3!\qoffs!-90:#1-#3$) -- cycle}

	\coordinate (12) at ($#1!\qoffs!90:#2$);
	\coordinate (14) at ($#1!\qoffs!-90:#4$);
	\coordinate (23) at ($#2!\qoffs!90:#3$);
	\coordinate (21) at ($#2!\qoffs!-90:#1$);
	\coordinate (34) at ($#3!\qoffs!90:#4$);
	\coordinate (32) at ($#3!\qoffs!-90:#2$);
	\coordinate (41) at ($#4!\qoffs!90:#1$);
	\coordinate (43) at ($#4!\qoffs!-90:#3$);
	
	\def\nrhedge{
		(14) let \p1=($(14)-#1$), \p2=($(12)-#1$) in 
		arc[start angle={atan2(\y1,\x1)}, delta angle={atan2(\y2,\x2)-atan2(\y1,\x1)-360*(atan2(\y2,\x2)-atan2(\y1,\x1)>0)}, x radius=\qoffs, y radius=\qoffs] --
		(21) let \p1=($(21)-#2$), \p2=($(23)-#2$) in 
		arc[start angle={atan2(\y1,\x1)}, delta angle={atan2(\y2,\x2)-atan2(\y1,\x1)-360*(atan2(\y2,\x2)-atan2(\y1,\x1)>0)}, x radius=\qoffs, y radius=\qoffs] --
		(32) let \p1=($(32)-#3$), \p2=($(34)-#3$) in 
		arc[start angle={atan2(\y1,\x1)}, delta angle={atan2(\y2,\x2)-atan2(\y1,\x1)-360*(atan2(\y2,\x2)-atan2(\y1,\x1)>0)}, x radius=\qoffs, y radius=\qoffs] --
		(43) let \p1=($(43)-#4$), \p2=($(41)-#4$) in 
		arc[start angle={atan2(\y1,\x1)}, delta angle={atan2(\y2,\x2)-atan2(\y1,\x1)-360*(atan2(\y2,\x2)-atan2(\y1,\x1)>0)}, x radius=\qoffs, y radius=\qoffs] --
		cycle}
	
	\ifx\relax#6\relax
	\def\rlwidth{1pt}
	\else
	\def\rlwidth{#6}
	\fi
	
	\ifx\relax#8\relax
	\fill \nrhedge;
	\else
	\fill[#8]\nrhedge;
	\fi
	
	\ifx\relax#7\relax
	\draw[line width=\rlwidth,rounded corners=\qoffs]\nrhedge;
	\else
	\draw[line width=\rlwidth,#7]\nrhedge;
	\fi
}

\let\epsilon=\varepsilon
\let\eps=\epsilon
\let\rho=\varrho
\let\theta=\vartheta

\def\NN{{\mathds N}}

\def\ZZ{{\mathds Z}}
\def\PP{{\mathds P}}
\def\RR{{\mathds R}}

\newcommand{\cR}{\mathcal{R}}

\newcommand{\ccA}{\mathscr{A}}
\newcommand{\ccB}{\mathscr{B}}
\newcommand{\ccS}{\mathscr{S}}
\newcommand{\ccC}{\mathscr{C}}
\newcommand{\ccP}{\mathscr{P}}

\newcommand{\ccW}{\mathscr{W}}

\newcommand{\ccZ}{\mathscr{Z}}

\newcommand{\gB}{\mathfrak{B}}
\newcommand{\gU}{\mathfrak{U}}
\newcommand{\gA}{\mathfrak{A}}
\newcommand{\gE}{\mathfrak{E}}

\newtheoremstyle{note}  {4pt}  {4pt}  {\sl}  {}  {\bfseries}  {.}  {.5em}          {}
\newtheoremstyle{introthms}  {3pt}  {3pt}  {\itshape}  {}  {\bfseries}  {.}  {.5em}          {\thmnote{#3}}
\newtheoremstyle{remark}  {2pt}  {2pt}  {\rm}  {}  {\bfseries}  {.}  {.3em}          {}

\theoremstyle{plain}
\newtheorem{theorem}{Theorem}[section]
\newtheorem{lemma}[theorem]{Lemma}
\newtheorem{prop}[theorem]{Proposition}

\newtheorem{cor}[theorem]{Corollary}
\newtheorem{fact}[theorem]{Fact}
\newtheorem{claim}[theorem]{Claim}

\theoremstyle{note}
\newtheorem{dfn}[theorem]{Definition}

\theoremstyle{remark}
\newtheorem{remark}[theorem]{Remark}

\newtheorem{exmpl}[theorem]{Example}

\usepackage{lineno}
\newcommand*\patchAmsMathEnvironmentForLineno[1]{
	\expandafter\let\csname old#1\expandafter\endcsname\csname #1\endcsname
	\expandafter\let\csname oldend#1\expandafter\endcsname\csname end#1\endcsname
	\renewenvironment{#1}
	{\linenomath\csname old#1\endcsname}
	{\csname oldend#1\endcsname\endlinenomath}}
\newcommand*\patchBothAmsMathEnvironmentsForLineno[1]{
	\patchAmsMathEnvironmentForLineno{#1}
	\patchAmsMathEnvironmentForLineno{#1*}}
\AtBeginDocument{
	\patchBothAmsMathEnvironmentsForLineno{equation}
	\patchBothAmsMathEnvironmentsForLineno{align}
	\patchBothAmsMathEnvironmentsForLineno{flalign}
	\patchBothAmsMathEnvironmentsForLineno{alignat}
	\patchBothAmsMathEnvironmentsForLineno{gather}
	\patchBothAmsMathEnvironmentsForLineno{multline}
}

\def\Ubad{U_\text{\rm bad}}
\usepackage{scalerel}

\makeatletter
\newcommand{\overrighharpoonup}[1]{\ThisStyle{%
		\vbox {\m@th\ialign{##\crcr
				\rightharpoonupfill \crcr
				\noalign{\kern-\p@\nointerlineskip}
				$\hfil\SavedStyle#1\hfil$\crcr}}}}

\def\rightharpoonupfill{%
	$\SavedStyle\m@th\mkern+0.8mu\cleaders\hbox{$\shortbar\mkern-4mu$}\hfill\rightharpoonuptip\mkern+0.8mu$}

\def\rightharpoonuptip{%
	\raisebox{\z@}[2pt][1pt]{\scalebox{0.55}{$\SavedStyle\rightharpoonup$}}}

\def\shortbar{%
	\smash{\scalebox{0.55}{$\SavedStyle\relbar$}}}
\makeatother

\let\seq=\overrighharpoonup

\makeatletter
\newcommand{\overlefharpoonup}[1]{\ThisStyle{%
		\vbox {\m@th\ialign{##\crcr
				\leftharpoonupfill \crcr
				\noalign{\kern-\p@\nointerlineskip}
				$\hfil\SavedStyle#1\hfil$\crcr}}}}

\def\leftharpoonupfill{%
	$\SavedStyle\m@th\mkern+0.8mu\cleaders\hbox{$\shortbar\mkern-4mu$}\hfill\leftharpoonuptip\mkern+0.8mu$}

\def\leftharpoonuptip{%
	\raisebox{\z@}[2pt][1pt]{\scalebox{0.55}{$\SavedStyle\leftharpoonup$}}}

\let\seqq=\overlefharpoonup

\def\sa{\seq a}
\def\sb{\seq b}
\def\su{\seq u}
\def\sw{\seq w}
\def\sx{\seq x}

\def\zetas{\zeta_{\star}}
\def\thetas{\theta_{\star}}

\def\zetass{\zeta_{\star\star}}
\def\thetass{\theta_{\star\star}}
\def\al{\alpha}
\def\P{\Psi}

\makeatletter
\newsavebox\myboxA
\newsavebox\myboxB
\newlength\mylenA

\newcommand*\xoverline[2][0.75]{%
	\sbox{\myboxA}{$\m@th#2$}%
	\setbox\myboxB\null% Phantom box
	\ht\myboxB=\ht\myboxA%
	\dp\myboxB=\dp\myboxA%
	\wd\myboxB=#1\wd\myboxA% Scale phantom
	\sbox\myboxB{$\m@th\overline{\copy\myboxB}$}%  Overlined phantom
	\setlength\mylenA{\the\wd\myboxA}%   calc width diff
	\addtolength\mylenA{-\the\wd\myboxB}%
	\ifdim\wd\myboxB<\wd\myboxA%
	\rlap{\hskip 0.5\mylenA\usebox\myboxB}{\usebox\myboxA}%
	\else
	\hskip -0.5\mylenA\rlap{\usebox\myboxA}{\hskip 0.5\mylenA\usebox\myboxB}%
	\fi}
\makeatother

\def\copr{\;\middle|\;}
\def\coprn{\mid}

\let\herz=\heartsuit
\let\pik=\spadesuit

\begin{document}
\dedicatory{Dedicated to Endre Szemer\'edi on the occasion of his $80^{\text{th}}$ birthday}

\title[On Hamiltonian cycles in hypergraphs with dense link graphs]
{On Hamiltonian cycles in hypergraphs with dense link graphs}
\author[J.~Polcyn]{Joanna Polcyn}
\address{Adam Mickiewicz University, Faculty of Mathematics and Computer Science, Pozna\'n, Poland}
\email{joaska@amu.edu.pl}

\author[Chr.~Reiher]{Christian Reiher}
\address{Fachbereich Mathematik, Universit\"at Hamburg, Hamburg, Germany}
\email{Christian.Reiher@uni-hamburg.de}
\email{bjarne.schuelke@uni-hamburg.de}

\author[V.~R\"{o}dl]{Vojt\v{e}ch R\"{o}dl}
\address{Department of Mathematics, Emory University, Atlanta, USA}
\email{vrodl@emory.edu}
\thanks{The third author is supported by NSF grant DMS 1764385.}

\author[B. Sch\"ulke]{Bjarne Sch\"ulke}

\subjclass[2010]{Primary: 05C65. Secondary: 05C45}
\keywords{Hamiltonian cycles, Dirac's theorem, hypergraphs}

\begin{abstract}
We show that every $k$-uniform hypergraph on $n$ vertices whose minimum $(k-2)$-degree 
is at least $(5/9+o(1))n^2/2$ contains a Hamiltonian cycle. A construction due to Han 
and Zhao shows that this minimum degree condition is optimal. The same result was proved
independently by Lang and Sahueza-Matamala.
\end{abstract}

\maketitle
	
\section{Introduction}

Hamiltonian cycles are a central theme in graph theory and extremal combinatorics. 
Dirac's classic result~\cite{Dirac} states that every graph on~$n\geq 3$ 
vertices whose minimum degree is at least~$\frac n2$ contains a Hamiltonian cycle.
The present work continues the investigation of hypergraph generalisations of 
Dirac's theorem -- an area of research owing many deep 
insights to Endre Szemer\'edi.   

\subsection{Hypergraphs and Hamiltonian cycles}
For~$k\ge 2$ a~\emph{$k$-uniform hypergraph} is defined to be a pair~$H=(V,E)$ consisting 
of a (finite) \emph{set of vertices}~$V$ and a 
set
\[
	E\subseteq V^{(k)}=\{U\subseteq V\colon \vert U\vert=k\}
\]
of edges. A $k$-uniform hypergraph $H=(V, E)$ with $n$ vertices is said to contain 
a {\it Hamiltonian cycle} if its vertex set admits a cyclic 
enumeration $V=\{x_i\colon i\in\ZZ/n\ZZ\}$ such that $\{x_i, x_{i+1}, \ldots, x_{i+k-1}\}\in E$ 
holds for all $i\in \ZZ/n\ZZ$. Observe that this naturally generalises the familiar notion of 
Hamiltonian cycles in graphs. 

In contrast to the graph case, there are several interesting minimum degree notions  
for hypergraphs. For a $k$-uniform hypergraph $H=(V, E)$ and a set~$S\subseteq V$ the {\it degree 
of~$S$ in~$H$} is defined by 
\[
	d_H(S)=\vert\{e\in E\colon S\subseteq e\}\vert\,.
\]
\vbox{
Moreover, for an integer~$i$ with~$1\leq i < k$ the number
\[
	\delta _{i}(H)=\min\bigl\{d_H(S)\colon S\in V^{(i)}\bigr\} 
\]
is called the {\it minimum~$i$-degree of $H$}. 
}

The research on minimum $i$-degree conditions guaranteeing the existence of Hamiltonian cycles 
in hypergraphs was initiated by Katona and Kierstead~\cite{KaKi}. 
The main problem is to determine, for any two given integers~$k\ge 2$ 
and~$i\in [k-1]$, the optimal minimum $i$-degree condition 
which for $k$-uniform hypergraphs ensures the existence of a Hamiltonian cycle. 
Notice that Dirac's aforementioned theorem solves the case $(k, i)=(2, 1)$. 

In general, if $i<j$, then a minimum $j$-degree condition seems to reveal more structural information 
about a hypergraph than a minimum $i$-degree condition. For this reason, it is reasonable to 
suspect that the difficulty of the problem we are interested in increases with~$k-i$. 
The first case, $i=k-1$, was solved more than a decade ago by R\"odl, Ruci\'nski, 
and Szemer\'edi~\cite{RRS}. 

\begin{theorem}
	For every integer~$k\ge 2$ and every $\alpha>0$ there exists an integer~$n_0$ such that 
	every~$k$-uniform hypergraph~$H$ on~$n\geq n_0$ vertices 
	with~$\delta_{k-1}(H)\geq\left(\frac{1}{2}+\alpha\right)n$ contains a Hamiltonian cycle.
\end{theorem}

Similarly as for Dirac's theorem, slightly unbalanced bipartite hypergraphs show that this result is asymptotically best possible. Our main result addresses the next case, $i=k-2$.

\begin{theorem}\label{t:main}
	For every integer~$k\geq 3$ and every~$\alpha >0$, there exists an integer~$n_0$ such that 
	every $k$-uniform hypergraph~$H$ on~$n\geq n_0$ vertices 
	with~$\delta_{k-2}(H)\geq\left(\frac{5}{9}+\alpha \right)\frac{n^2}{2}$ contains a 
	Hamiltonian cycle.
\end{theorem}

In previous articles written in collaboration with Ruci\'nski, Schacht, and Szemer\'edi~\cites{R, Y}
we solved the cases $k=3$ and $k=4$.
The general case was also obtained by Lang and Sanhueza-Matamala~\cite{Lang} 
in independent research. A construction due to Han and Zhao~\cite{HZ16} reproduced in the 
introduction of~\cite{Y} shows that the number $\frac59$ appearing in Theorem~\ref{t:main} is optimal.

We would like to conclude this subsection by pointing to some problems for future investigations.
First and foremost, it remains an intriguing question whether for $k\ge 4$ the minimum $(k-3)$-degree
condition $\delta_{k-3}(H)\ge \left(\frac{5}{8}+o(1) \right)\frac{n^3}{6}$ enforces the existence
of a Hamiltonian cycle. Here the number $\frac 58$ would again match the construction of 
Han and Zhao~\cite{HZ16}.   

Another possible area of research would be to extend the work of P\'osa~\cite{Posa} 
and Chv\'atal~\cite{Chv}, who in the graph case studied which conditions on the degree
sequence (rather than just on the minimum degree) guarantee the existence of Hamiltonian 
cycles. Such degree sequence versions have recently been obtained for the Hajnal-Szemer\'edi 
theorem~\cite{HS}  
by Treglown~\cite{Tr} and for P\'osa's conjecture (see~\cite{Fiedler}*{Problem 9})	
by Staden and Treglown~\cite{ST}. 
It would be very interesting to find similar theorems for Hamiltonian cycles in hypergraphs. 
For first results in this direction we refer to~\cite{Schuelke}. 

\subsection{Organisation and Overview}
We use the {\it absorption method} developed by R\"odl, Ruci\'nski, and Szemer\'edi
and surveyed by Szemer\'edi himself in~\cite{Sz}. Therefore, the proof decomposes in the usual 
way into a Connecting Lemma, an Absorbing Path Lemma, and a Covering Lemma. 

Very roughly speaking, the Absorbing Path Lemma reduces the task of proving Theorem~\ref{t:main}
to the much easier problem of finding `almost spanning' cycles in $k$-uniform hypergraphs $H$ 
satisfying~$\delta_{k-2}(H)\geq\left(\frac{5}{9}+\alpha \right)\frac{|V(H)|^2}{2}$. 
Such an almost spanning cycle is build in two main steps: First, the covering lemma asserts
that we can cover almost all vertices by means of long paths. Second, the Connecting Lemma
allows us to connect these `pieces' into one long cycle.

In our earlier articles we stored all information about $H$ that became relevant in the course 
of the proof in various `setups' and the complexity of these setups got somewhat out of control. 
To avoid this in the present work, we abandon the setups and replace them by the much more
flexible notion of a {\it constellation} (see Definition~\ref{d:1723} below).

Section~\ref{sec:prelim} lays out a systematic treatment of these constellations and contains 
several auxiliary results that will assist us later.   
The subsequent \hbox{Sections~\ref{sec:conn}\,--\,\ref{sec:cov}} deal with the  
main lemmata enumerated above: connecting, absorbing, and covering.
Lastly, in Section~\ref{sec:main-pf} we derive Theorem~\ref{t:main} from these results. 

\section{Preliminaries}\label{sec:prelim}

\subsection{Graphs}\label{subsec:graphs}
In our earlier articles~\cites{R,Y} dealing with the $3$- and $4$-uniform case of 
Theorem~\ref{t:main} we inductively reduced connectability properties of the hypergraphs under
discussion to connectability properties of their $2$-uniform link graphs.
Here we pursue the same strategy and the present subsection contains the graph theoretic
preliminaries that we require for this purpose. 
The central notion we work with in this context is taken from~\cite{R}*{Definition~2.2}
and reappeared as~\cite{Y}*{Definition~2.1}. 

\begin{dfn} \label{d:robust}
	Given $\beta>0$ and $\ell\in\NN$ a graph~$R$ is said to 
	be {\it $(\beta, \l)$-robust} 
	if for any two distinct vertices~$x$ and~$y$ of~$R$ the number of $x$-$y$-paths of 
	length $\ell$ is at least $\beta |V(R)|^{\ell-1}$.
\end{dfn}

It turns out that every graph whose edge density is larger than $5/9$ possesses 
a robust subgraph containing more than two thirds of its vertices that is 
quite disconnected from the rest of the graph. The following statement to this 
effect was proved in \cite{Y}*{Proposition~2.2} (marginally 
strengthening~\cite{R}*{Proposition~2.3}).

\begin{prop} \label{prop:robust}
	Given $\alpha$, $\mu >0$, there exist $\beta>0$ and an odd integer $\ell\ge 3$
	such that for sufficiently large~$n$, every $n$-vertex graph $G=(V, E)$ 
	with $|E|\ge \big(\frac{5}{9}+\alpha\big)\frac{n^2}{2}$
	contains a $(\beta,\l)$-robust induced subgraph $R\subseteq G$ satisfying
	\begin{enumerate}[label=\rmlabel]
		\item\label{it:rc1} $|V(R)|\ge \bigl(\frac23+\frac\alpha 2\bigr)n$,
		\item\label{it:rc2} and $e_G\big(V(R), V\setminus V(R)\big)\leq \mu n^2$. 			
	\end{enumerate}
\end{prop}

\begin{remark}\label{r:224}
	We shall usually apply Proposition~\ref{prop:robust} with $\mu\le \frac\alpha4$. In this 
	case, clause~\ref{it:rc2} yields 
	\[
		e(R)
		\ge 
		\Big(\frac{5}{9}+\frac{\alpha}{2}\Big)\frac{n^2}{2}-\frac{(n-|V(R)|)^2}{2}
		\overset{\ref{it:rc1}}{\ge}
		\Big(\frac49+\frac23\alpha\Big)\frac{n^2}{2}\,.
	\]
	Originally, this estimate was included as a third clause into~\cite{Y}*{Proposition~2.2},
	but it seems preferable to omit this part.  
\end{remark}
	
In Section~\ref{sec:Abpa} below we need to render our absorbers connectable. 
To this end we shall utilise a consequence of the following graph theoretic lemma. 

\begin{lemma}\label{l:21}
	Let $\alpha>0$ and let $G$ be a graph with $n$ vertices and at 
	least $\big(\frac 59+\alpha\big)\frac{n^2}{2}$ edges. If
	\[
		A = \bigl\{x\in V(G)\colon d(x) < n/3\bigr\}
	\]
	and 
	\[
		B = \bigl\{x\in V(G)\colon |N(x)\setminus A|\le \al n/3\bigr\}\,,
	\]
	then 
	\[
		e(A\cup B) \le \frac{n^2}{18}\,.
	\]
\end{lemma}
	
\begin{proof}
	In the special case that $|A|< (\frac 13 - \frac{\al}{3})n$, every vertex $x\in B$ satisfies 
	\[
		d(x)\le |N(x)\setminus A|+|A| < \frac n3\,, 
	\]
	which yields $B\subseteq  A$ and the desired inequality
	\[
		e(A\cup B) = e(A) \le \frac 12|A|^2\le \frac{1}{18}n^2\,.
	\]

	So henceforth we may suppose that 
	\begin{equation}\label{eq:20}
		|A|\ge \Big(\frac 13 - \frac{\al}{3}\Big)n\,.
	\end{equation}
	Now the definition of $A$ implies
	\[
		\frac 59 n^2 
		\le 
		2e(G)
		=
		\sum_{x\in V(G)}d(x)\le \frac 13|A|n + (n-|A|)n 
		= 
		n^2 - \frac 23|A|n\,,
	\]
	i.e.,
	\begin{equation}\label{eq:23}
		|A| \le \frac 23 n\,,
	\end{equation}
	and
	\begin{equation}\label{eq:24}
		e(G-A)\ge \Big(\frac 59 + \al \Big)\frac{n^2}{2}-\frac 13|A|n\,.
	\end{equation}	

	Setting $X=V(G)\setminus (A\cup B)$ we conclude from the definition of $B$ that 
	\begin{equation}\label{eq:25}
		2e(B\setminus A) + e(B\setminus A, X) 
		= 
		\sum_{x\in B\setminus A}|N(x)\setminus A|
		\le 
		|B\setminus A|\cdot \frac{\al}{3}n\le \frac{\al}{3}n^2\,,
	\end{equation}
	which together with \eqref{eq:24} yields
	\begin{align*}
		|X|^2 
		&\ge 
		2e(X) 
		= 
		2e(G-A) - 2e(B\setminus A)- 2e(B\setminus A,X) \\
		&\ge 
		\Big(\frac 59 + \al\Big)n^2 - \frac 23|A|n-\frac 23\al n^2 
		\ge 
		\frac 59 n^2 - \frac 23 |A|n\,.
	\end{align*}
	In view of \eqref{eq:23} this entails
	\[
		|X|^2 
		\ge 
		\frac 49 n^2 - \frac 23|A|n + \frac 14 |A|^2 
		= 
		\Big(\frac 23 n-\frac 12|A|\Big)^2\,,
	\]
	wherefore
	\begin{equation}\label{eq:26}
		|X| \ge \frac 23n - \frac 12|A|\,.
	\end{equation}

	Next, we claim that
	\begin{equation}\label{eq:27}
		\frac 13|A|n+|B\setminus A||A| + \frac 12|X|^2 \le \Big(\frac 13 + \frac{\al}{6}\Big)n^2\,.
	\end{equation}
	In view of $|A|+|B\setminus A|+|X|=n$ the left side of this estimate rewrites as
	\[
		\frac 13|A|n+(n-|A|-|X|)|A|+\frac 12|X|^2
		=
		\frac 43 |A|n - \frac 32 |A|^2 + \frac 12(|A|-|X|)^2\,.
	\]
	By \eqref{eq:26} and $X\subseteq V(G)\setminus A$ we have 
	\[
		\frac 23 n - \frac 32 |A|
		\le 
		|X| - |A| 
		\le 
		n-2|A|
	\]
	and, hence, 
	\[
		(|A|-|X|)^2
		\le 
		\max\Big\{(n-2|A|)^2, \Big(\frac 23n -\frac 32|A|\Big)^2\Big\}\,.
	\]
	So to conclude the proof of~\eqref{eq:27} it suffices to observe that				
	\[
		\frac 43 |A|n - \frac 32|A|^2 + \frac 12(n-2|A|)^2 = \frac{n^2}3+{\frac 16 (n-|A|)(n-3|A|)}
		\overset{\eqref{eq:20}}{\le} 
		\Big(\frac 13 + \frac{\al}{6}\Big)n^2
	\]
	and, similarly, 
	\[
		\frac 43|A|n - \frac 32 |A|^2 + \frac 12 \Big(\frac 23 n - \frac 32|A|\Big)^2
		=
		\frac{n^2}{3} - \Big(\frac 13n-\frac 12|A|\Big)^2 - \frac 18|A|^2 
		\le 
		\frac{n^2}{3}\,.
	\]

	Having thus established~\eqref{eq:27} we appeal to the definition of $A$ again and observe
	\[
		e(A)+e(G)=\sum_{x\in A}d(x)+e(G-A)
		\le 
		\frac 13 |A|n+e(G-A)\,.
	\]
	Consequently, 
	\[
		e(A\cup B) + e(G)
		\le 
		\frac 13|A|n+e(B\setminus A, A) + e(B\setminus A)+e(G-A)
	\]
	and \eqref{eq:25} leads to 	
	\[
		e(A\cup B)+e(G)
		\le 
		\frac 13 |A|n+|B\setminus A||A|+e(X)+\frac{\al}{3}n^2\,.
	\]
	Owing to \eqref{eq:27} we deduce 
	\[
		e(A\cup B)+ e(G) 
		\le 
		\Big(\frac 13 +\frac \al 6\Big)n^2+\frac \al 3 n^2 
		= 
		\Big(\frac 23 + \al \Big)\frac{n^2}2
		\le 
		\frac{1}{18}n^2+e(G)\,,
	\]
	which implies the desired estimate $e(A\cup B)\le \frac 1{18}n^2$.
\end{proof}
	
\begin{remark}\label{r:22}
	The set $A$ already had an appearance in~\cite{Y} and Lemma~2.3 there is roughly equivalent
	to the weaker estimate $e(A)\le \frac{n^2}{18}$. 
	Concerning the set $B$ one can prove $|B|\le \frac n3$, but this 
	fact is not going to be exploited in the sequel. 
\end{remark}

The following consequence of Lemma~\ref{l:21} will later be generalised to $k$-uniform hypergraphs 
(see Lemma~\ref{l:24}) and constitutes the base case of an induction on $k$. 

\begin{cor}\label{c:23}
	Let $\al > 0$, and let $V$ be a set of $n$ vertices.
	If $G$, $G'$ are two graphs with $V(G), \, V(G') \subseteq V$ and  
	\[
		e(G),\, e(G')\ge \Big(\frac 59 + \al\Big)\frac{n^2}{2}\,,
	\]
	then there are at least $\frac{\al^2}{3}n^3$ triples $(x,y,z)\in V^3$ such that 
	\begin{itemize}
		\item $xyz$ is a walk in $G$,
		\item $xy\in E(G')$,
		\item and $d_G(y),\,d_G(z)\ge \frac n3$.
	\end{itemize}
\end{cor}
\begin{proof}
	By adding some isolated vertices to $G$ and $G'$ if necessary, we may assume that $V(G)=V(G')=V$.
	The sieve formula yields	
	\[
		|E(G)\cap E(G')|
		\ge 
		2\Big(\frac 59 + \al\Big) \frac{n^2}{2} - \frac{n^2}{2}
		=
		\Big(\frac 1{18}+\al\Big)n^2\,.
	\]
	Define the sets $A$ and $B$ with respect to $G$ as in Lemma \ref{l:21}. 
	In view of that lemma itself, there are at least $\al n^2$ unordered 
	pairs $xy\in E(G)\cap E(G')$ for which $x, y\in A\cup B$ fails. 
	Consequently, there are at least $\al n^2$ ordered pairs $(x,y)\in V^2$ such 
	that $xy\in E(G)\cap E(G')$ and $y\notin A\cup B$. For each of them there are, 
	by the definition of $B$, at least $\frac{\al}{3}n$ vertices $z$ with $yz\in E(G)$ 
	and $z\notin A$. Altogether, this yields at least $\frac{\al^2}{3}n^3$ 
	triples $(x,y,z)$ with the desired properties.
\end{proof}

\subsection{Hypergraphs}

In this subsection we introduce our terminology and some preliminary results on hypergraphs. 
When there is no danger of confusion we shall omit several parentheses, braces, 
and commas. For instance, we write $x_1\dots x_k$ for the edge $\{x_1,\dots, x_k\}$ of 
a $k$-uniform hypergraph.

\subsubsection*{Walks, paths, and cycles}
A $k$-uniform \emph{walk~$W$} of length $\l\geq 0$ is a hypergraph 
whose vertices can, possibly with repetitions, be enumerated 
as $(x_1,\dots,x_{\l+k-1})$ in such a way that $e\in E(W)$ if and only if $e=x_i\dots x_{i+k-1}$ for 
some $i\in[\l]$. The ordered $(k-1)$-tuples $(x_1,\dots, x_{k-1})$ and $(x_{\l+1},\dots,x_{\l+k-1})$
are called the \emph{end-tuples} of~$W$ and we say that $W$ is 
a $(x_1\dots x_{k-1})$-$(x_{\l+1}\dots x_{\l+k-1})$-walk.  
This notion of end-tuples is not symmetric and implicitly fixes a direction 
of~$W$. Sometimes we refer to $(x_1,\dots, x_{k-1})$ and $(x_{\l+1},\dots, x_{\l+k-1})$ as the 
\emph{starting $(k-1)$-tuple} and \emph{ending $(k-1)$-tuple} of $W$, respectively. 
We call $x_k,\dots,x_\l$ the {\it inner vertices} of $W$. Counting them 
with their multiplicities we say for $\ell\ge k-1$ that a walk of length~$\ell$
has~$\l-k+1$ inner vertices. We often identify a walk with the sequence of its 
vertices $x_1x_2\dots x_{\l+k-1}$. If the vertices $x_1,\dots,x_{\l+k-1}$ are distinct 
we call the walk $W$ a {\it path}.
For $\ell> k$ a {\it cycle of length $\ell$} is a hypergraph $C$ whose vertices and 
edges can be represented as $V(C)=\{x_i\colon i\in \ZZ/\ell\ZZ\}$ and 
$E(C)=\{x_i\ldots x_{i+k-1}\colon i\in \ZZ/\ell\ZZ\}$.

\subsubsection*{Link hypergraphs}
Given a $k$-uniform hypergraph $H=(V,E)$ and a set $S\subseteq V$ with $|S|\le k-2$ 
we define the $(k-|S|)$-uniform \emph{link hypergraph} $H_S$ by $V(H_S) = V(H)$ and  
\[
	E(H_S)=\{e\setminus S\colon S\subseteq e \in E\}\,.
\]
Clearly the vertices in $S$ are isolated in $H_S$ and sometimes it is convenient to remove them. 
In such cases, we write $\xoverline H_S=H_S- S$. 
For instance, we have 
$H_\varnothing=\xoverline H_\varnothing=H$ for every hypergraph~$H$. If $S=\{v\}$ consists 
of a single vertex, we abbreviate $H_{\{v\}}$ to $H_v$.
		
\subsubsection*{A lemma with two hypergraphs}	
Our next step is to generalise Corollary~\ref{c:23} to hypergraphs. 

\enlargethispage{1em}	
\begin{lemma}\label{l:24}
	Suppose that $k\ge 2$, $\al > 0$, and that $V$ is a set of $n$ vertices. 
	If $H$, $H'$ are two $k$-uniform hypergraphs satisfying	
	\[
		V(H), V(H') \subseteq V
	\]
	and 
	\[
		\delta_{k-2}(H), \,\delta_{k-2}(H')\ge \Big(\frac 59 + \al\Big)\frac{n^2}{2}\,,
	\]
	then the number of $(2k-1)$-tuples $(x_1, \dots, x_{2k-1})\in V^{2k-1}$ such that 
		\begin{itemize}
			\item $x_1 \dots x_{2k-1}$ is a walk in $H$,
			\item $\{x_1, \dots, x_k\}\in E(H')$,
			\item and $d_H(x_2, \dots, x_k),\, d_H(x_{k+1}, \dots, x_{2k-1})\ge \frac n3$
		\end{itemize}
	is at least $\big(\frac{\al}{2}\big)^{2^{k-1}}n^{2k-1}$.
\end{lemma}

\begin{proof}
	For $k=2$ this follows from Corollary~\ref{c:23}. Proceeding by induction on $k$, 
	we assume $k\ge 3$ and that the assertion holds for $k-1$ in place of $k$. 
	Construct an auxiliary bipartite graph $\Gamma$ with vertex classes $V$ and $V^{2k-3}$ 
	by drawing an edge between $x\in V$ and 
	\[
		(x_1, \dots, x_{k-2}, x_k, \dots, x_{2k-2})\in V^{2k-3}
	\]
	if and only if
	\begin{enumerate}[label=\alabel]
		\item $x_1 \dots x_{k-2} x_k \dots x_{2k-2}$ is a walk in $\xoverline H_x$,
		\item $\{x_1, \dots, x_{k-2}, x_k\}\in E(\xoverline H'_x)$,
		\item $d_{\overline H_x}(x_2, \dots, x_{k-2}, x_k)\ge \frac n3$ 
			and $d_{\overline H_x}(x_{k+1}, \dots, x_{2k-2})\ge \frac n3$.
	\end{enumerate}
	The induction hypothesis, applied to the hypergraphs $\xoverline H_x$ and $\xoverline H'_x$, 
	reveals that every vertex $x\in V$ has at least 
	degree $\bigl(\frac \al 2\bigr)^{2^{k-2}}n^{2k-3}$ in $\Gamma$. Thus
	\[
		e(\Gamma)\ge \Big(\frac{\al}{2}\Big)^{2^{k-2}}n^{2k-2}
	\]
	and the Cauchy-Schwarz inequality implies 
	\[
		\sum_{\seq {x}\in V^{2k-3}}|N_\Gamma(\seq x)|^2
		\ge 
		\frac{e(\Gamma)^2}{n^{2k-3}}\ge \Big(\frac{\al}{2}\Big)^{2^{k-1}}n^{2k-1}\,,
	\]
	where $N_\Gamma(\seq x)$ denotes the neighbourhood of the vertex $\seq{x}$ in $\Gamma$.  
	Now if 
	\[
		\seq x = (x_1, \dots, x_{k-2}, x_k, \dots, x_{2k-2})\in V^{2k-3}
		\quad \text{ and } \quad 
		x_{k-1}, x_{2k-1}\in N_\Gamma(\seq x)
	\]
	are arbitrary, then $(x_1, \dots, x_{2k-1})$ has the desired properties. 
\end{proof}

\subsubsection*{Walks in dense hypergraphs}
For later use we now quote a lower bound on the number of walks of given length 
in a given dense hypergraph, that is somewhat related to Sidorenko's conjecture~\cites{Sid, Sim}.
It is well known that this conjecture holds for paths in graphs, i.e., that for $d\in [0,1]$
and $\ell\in\NN$ every graph $G=(V, E)$ satisfying $|E|\ge d|V|^2/2$ contains at least 
$d^\ell |V|^{\ell+1}$ walks of length $\ell$ (see~\cite{BR} for a proof based on linear 
algebra and~\cite{AR}*{Lemma 3.8} for a different approach using vertex deletions and the tensor 
power trick). The latter argument generalises in a straightforward manner to partite hypergraphs
(see Lemma~\ref{l:217} below). An alternative proof based on the entropy method was worked out 
by Fitch~\cite{Fitch}*{Lemma 7} and by Lee~\cite{Lee}*{Theorems 2.6 and 2.7}.

\begin{lemma}\label{l:217}
	Suppose $k\ge 2$, $d\in [0, 1]$, and that $H$ is a $k$-partite $k$-uniform hypergraph with vertex 
	partition $\{V_i\colon i\in \ZZ/k\ZZ\}$. If $H$ has $d\prod_{i\in \ZZ/k\ZZ}|V_i|$
	edges, then for every $r\ge k$ there are at least 
	\[
	d^{r-k+1}\prod_{i\in [r]}|V_i|
	\]
	walks $(x_1, \dots, x_r)$ in $H$ with $x_1\in V_1, \dots, x_k\in V_k$. \qed
\end{lemma}

By identifying the vertex classes one obtains the following, more standard, non-partite version
of this lemma. 

\begin{cor}\label{c:218}
	For $k\ge 2$ and $d\in [0, 1]$ let $H=(V, E)$ be a $k$-uniform hypergraph.  
	If $|E|\ge d|V|^k/k!$, then for every integer $r\ge k$ there are at least $d^{r-k+1}|V|^r$ 
	walks $(x_1, \dots, x_r)$ in $H$. \qed
\end{cor}

\subsection{Abstract connectability}\label{subsec:ac}
Our intended way of using Proposition~\ref{prop:robust} is that given a $k$-uniform hypergraph $H$
satisfying $\delta_{k-2}(H)\ge \bigl(\frac 59 + \al \bigr)|V(H)|^2/2$ we can choose robust 
subgraphs of all the $\binom{|V(H)|}{k-2}$ link graphs. It will be convenient to collect
the data thus arising into a single structure.

\begin{dfn}\label{d:1723}
	For $k\ge 2$ a $k$-uniform {\it constellation} is a pair
		\[
			\P=\bigl(H, \bigl\{R_x\colon x\in V(H)^{(k-2)}\bigr\}\bigr)
		\]	
	consisting of a $k$-uniform hypergraph $H$ and a family of induced subgraphs $R_x\subseteq H_x$ 
	of the 2-uniform link hypergraphs that can be formed in $H$. We write $H(\Psi)=H$ for the 
	{\it underlying hypergraph} of a constellation $\Psi$ and use the abbreviations $V(\Psi)=V(H)$, 
	$E(\Psi)=E(H)$ for its vertex set and edge set, respectively. For a constellation $\Psi$ 
	and $x\in V(\Psi)^{(k-2)}$ we denote the subgraph associated with $x$ by $R_x^{\Psi}=R_x$.
\end{dfn}

\begin{exmpl}
	A $2$-uniform constellation is determined by its underlying graph $H$ and a distinguished 
	induced subgraph $R_{\emptyset}\subseteq H_{\emptyset}=H$.
\end{exmpl}	

Notice that so far the induced subgraphs $R_x\subseteq H_x$ are completely arbitrary and at
this moment there are no restrictions on their orders, sizes, and connectivity properties. 
This allows us to study constellations ``axiomatically'', adding further useful conditions
in each of the following subsections. The central connectability notions are definable without 
any such assumptions and they will be introduced in the present subsection (see 
Definition~\ref{d:210} below). Of course one cannot prove a meaningful Connecting Lemma at this 
level of generality, so our way of organising the material may appear somewhat peculiar on
first sight. When establishing the covering lemma in Section~\ref{sec:cov} however, we need 
to analyse connectability in random subconstellations and for such situations the abstract
approach developed here turns out to be advantageous. 
Subconstellations themselves are defined in the expected way.
	
\begin{dfn}\label{d:29}
	Let
	\[
		\P=\bigl(H, \bigl\{R_x \colon x\in V(H)^{(k-2)}\bigr\}\bigr)
	\]
	be a $k$-uniform constellation, where $k\ge 2$. For $X\subseteq V(\Psi)$ we call
	\[
		\P[X]=\bigl(H[X], \bigl\{R_x[X] \colon x\in X^{(k-2)}\bigr\}\bigr)
	\]
	the {\it subconstellation of $\Psi$ induced by $X$}. 
	Moreover, $\Psi-X=\Psi[V(\Psi)\setminus X]$ denotes the constellation obtained from $\Psi$
	by {\it removing} $X$.
\end{dfn}

We can also form link constellations in the obvious way.

\begin{dfn}\label{d:28}
	Let $k\ge 2$ and let 
	\[
		\P = \bigl(H, \bigl\{R_x\colon x\in V(H)^{(k-2)}\bigr\}\bigr)
	\]
	be a $k$-uniform constellation. If $S\subseteq V(\Psi)$ and $|S|\le k-2$, 
	then the $(k-|S|)$-uniform {\it link constellation} $\Psi_S$ is defined to be 
	\[
		\P_S=\bigl(\xoverline H_S, 
		\bigl\{R_{x\cup S}- S\colon x\in (V(H)\setminus S)^{(k-2-|S|)}\bigr\}\bigr)\,.
	\]
\end{dfn}

Next we tell which $(k-1)$-tuples of vertices of a $k$-uniform constellation are regarded as 
being $\zeta$-leftconnectable for a given real number $\zeta>0$. The definition progresses by recursion on $k$. 

\begin{dfn}\label{d:210}
	Let $k\ge 2$, $\zeta > 0$, let 
	\[
		\P=\bigl(H, \bigl\{R_x\colon x\in V(H)^{(k-2)}\bigr\}\bigr)
	\]
	be a $k$-uniform constellation, and let $\seq x = (x_1, \dots, x_{k-1})\in V(\Psi)^{k-1}$ 
	be a $(k-1)$-tuple of distinct vertices. 
	\begin{enumerate}[label=\alabel]
		\item\label{it:1259a} If $k=2$ we say that $\seq x = (x_1)$ is {\it $\zeta$-leftconnectable} 
			in $\Psi$ if $x_1\in V(R_\emptyset)$.
		\item\label{it:1259b} If $k\ge 3$ we say that $\seq x$ is {\it $\zeta$-leftconnectable} 
			in $\Psi$ if 
			\[
				|U^{\P}_{\seq x}|\ge \zeta |V(\Psi)|\,,
			\]
			where 
			\begin{multline*}
				\hskip5em U^{\P}_{\seq x}=\bigl\{z\in V(\Psi)\colon x_1\dots x_{k-1}z \in E(\Psi)
					 \textrm{ and }\\
				(x_2, \dots, x_{k-1}) \textrm{ is $\zeta$-leftconnectable in } \Psi_z \bigr\}\,.
			\end{multline*}
	\end{enumerate}	
\end{dfn} 

We remark that this is a ``new'' concept in the sense that in the earlier articles~\cites{Y, R} 
we managed to work with symmetric notions of connectability. For this reason, we need
to be careful when quoting the Connecting Lemma from~\cite{R} later. 

\begin{exmpl}
	Let $(x_1, x_2)$ be a pair of distinct vertices from a $3$-uniform constellation~$\Psi$
	and let $\zeta>0$. According to part~\ref{it:1259b} of Definition~\ref{d:210} 
	the pair $(x_1, x_2)$ is $\zeta$-leftconnectable in~$\Psi$ if and only if 
	$|U^{\P}_{(x_1, x_2)}|\ge \zeta |V(\Psi)|$. Due to part~\ref{it:1259a} the definition 
	of this set unravels to 
	\[
		U^{\P}_{(x_1, x_2)}
		=
		\bigl\{z\in V(\Psi)\colon x_1x_2z \in E(\Psi)
					 \textrm{ and } 
				 x_2\in V(R^\Psi_z) \bigr\}\,.
	\]
\end{exmpl}

There is a dual notion of rightconnectability obtained by reversing the ordering of the vertices. 

\begin{dfn}\label{d:212}
	Let $k\ge 2$, $\zeta >0$, $\Psi$, and $\seq x\in V(\Psi)^{k-1}$ be as in Definition~\ref{d:210}. 
	\begin{enumerate}[label=\alabel]
		\item\label{it:d212a} If the reverse tuple $(x_{k-1}, \dots, x_1)$ 
			is $\zeta$-leftconnectable, then 
			$\seq{x}$ itself is said to be~{\it $\zeta$-rightconnectable}.
		\item\label{it:d212b} Further, $\seq x$ is called {\it $\zeta$-connectable} if it is 
			$\zeta$-leftconnectable and $\zeta$-rightconnectable. 
	\end{enumerate}
\end{dfn}

Some readers may react negatively to our choice of the specifiers `left' and `right' in these 
notions, arguing that the definition of leftconnectability of $\seq{x}$ pivots on the right 
end-segment of $\seq{x}$. The reason for our terminological choice is that the 
Connecting Lemma (Proposition~\ref{prop:cl3} below) will assert that under reasonable assumptions 
every leftconnectable tuple can be connected to every rightconnectable tuple in such a way that 
the leftconnectable tuple is `on the left side' in the resulting path, while the rightconnectable 
tuple is `on the right side'. 
  
The following observation follows by a straightforward induction from Definition~\ref{d:210}. 
In later sections we will often use it either tacitly or by referring to `monotonicity'. 
  
\begin{fact}\label{r:monoton}
	For a $k$-uniform constellation $\Psi$ and $\zeta>\zeta'>0$ every $\zeta$-leftconnectable 
	$(k-1)$-tuple is also $\zeta'$-leftconnectable. Similarly statements hold 
	for rightconnectability and connectability. 
\end{fact}

\begin{proof}
	It suffices to display the argument for leftconnectability. We argue by induction on~$k$. 
	In the base case $k=2$ the definition of $\zeta$-leftconnectability does not depend on $\zeta$
	and there is nothing to prove. Now let $k\ge 3$ and suppose that the assertion is true 
	for~$k-1$ playing the r\^{o}le of $k$.
	
	Let $\zeta > \zeta' > 0$, let $\P=\bigl(H, \bigl\{R_x\colon x\in V(H)^{(k-2)}\bigr\}\bigr)$
	be a $k$-uniform constellation, and let $\seq x = (x_1, \dots, x_{k-1})\in V(\Psi)^{k-1}$ 
	be a $\zeta$-leftconnectable $(k-1)$-tuple. We are to prove that~$\seq x$ is 
	$\zeta'$-leftconnectable as well. To this end we consider the sets
	\begin{align*}
		U
		&=
		\{z\in V(\Psi)\colon x_1 \dots x_{k-1}z\in E(\Psi) \textrm{ and } 
			(x_2, \dots, x_{k-1}) \textrm{ is $\zeta$-leftconnectable in } \Psi_z\} 
		\intertext{and} 
		W
		&=
		\{z\in V(\Psi)\colon\,x_1 \dots x_{k-1}z\in E(\Psi) \textrm{ and } 
			(x_2, \dots, x_{k-1}) \textrm{ is $\zeta'$-leftconnectable in } \Psi_z\}\,.
	\end{align*}
	The induction hypothesis discloses $U\subseteq W$
	and the assumption that $\seq{x}$ is $\zeta$-leftconnectable means that
	$|U|\ge \zeta |V(\Psi)|$. So altogether we have 
	\[
		|W| 
		\ge 
		|U|
		\ge 
		\zeta |V(\Psi)| 
		\ge 
		\zeta' |V(\Psi)|\,,
	\]
	for which reason $\seq{x}$ is indeed $\zeta'$-leftconnectable. 
\end{proof}

Next, we study connectability in subconstellations. 

\begin{fact}\label{f:41}
	Suppose that $\Psi$ is a $k$-uniform constellation, that $\Psi'=\Psi[X]$ 
	is a subconstellation induced by some $X\subseteq V(\Psi)$ 
	with $|X|\ge \frac12\bigl(|V(\Psi)|+k-2\bigr)$. If 
	$\seq{x}\in V(\Psi')^{k-1}$ is $(2\zeta)$-leftconnectable in $\Psi'$,
	then it is $\zeta$-leftconnectable in $\Psi$ as well. Similar statements hold for 
	`rightconnectability' and `connectability'. 
\end{fact}

\begin{proof}
	Again we only display the argument for leftconnectability and proceed by induction 
	on $k$. The base case $k=2$ is trivial. For the induction step from $k-1$ to $k$ we 
	recall that the assumption entails $|U|\ge 2\zeta |V(\Psi')|\ge \zeta |V(\Psi)|$, where
	\[
		U=\bigl\{z\in V(\Psi')\colon x_1\dots x_{k-1}z \in E(\Psi')
					 \textrm{ and }
				(x_2, \dots, x_{k-1}) \textrm{ is $(2\zeta)$-leftconnectable in } \Psi'_z \bigr\}\,.
	\]
	Now consider an arbitrary vertex $z\in U$. Since 
	\[
		|V(\Psi'_z)|=|V(\Psi')|-1
		\ge 
		\frac12\bigl(|V(\Psi)|+k-4\bigr)
		=
		\frac12\bigl(|V(\Psi_z)|+k-3\bigr)\,, 
	\]
	the induction hypothesis is applicable to the constellation $\Psi_z$, 
	its subconstellation $\Psi'_z$, and to the $(2\zeta)$-leftconnectable $(k-2)$-tuple 
	$(x_2, \dots, x_{k-1})$. It follows that 
	\[
		U\subseteq \bigl\{z\in V(\Psi)\colon x_1\dots x_{k-1}z \in E(\Psi)
					 \textrm{ and }
				(x_2, \dots, x_{k-1}) \textrm{ is $\zeta$-leftconnectable in } \Psi_z \bigr\}
	\]
	and together with $|U|\ge \zeta |V(\Psi)|$ this shows that $\seq{x}$ is 
	indeed $\zeta$-leftconnectable in $\Psi$.
\end{proof}

We shall frequently have the situation that for some edge $x_1\dots x_k$ of a $k$-uniform 
constellation $\Psi$ we know $x_k \in V(R^\P_{x_1\dots x_{k-2}})$ and we would like to 
conclude from this state of affairs that $(x_2, \dots, x_k)$ is $\zeta$-leftconnectable in $\Psi$.
While such deductions are invalid in general, it turns out that for small values of $\zeta$ there are
only few exceptions to this rule of inference. More precisely, we have the following 
result (cf.~\cite{R}*{Fact~4.1} and \cite{Y}*{Lemma~3.7} for similar statements).  

\begin{lemma}\label{l:214}
	Let $k\ge 2$ and $\zeta >0$ be given. If $\P$ is a $k$-uniform constellation, then there 
	exist at most ${(k-2)\zeta|V(\Psi)|^k}$ $k$-tuples $(x_1, \dots, x_k)\in V(\Psi)^k$ such that 
	\begin{enumerate}[label=\alabel]
		\item\label{it:214a} $\{x_1, \dots, x_k \}\in E(\Psi)$, 
		\item\label{it:214b} $x_k \in V(R^\P_{x_1\dots x_{k-2}} )$, 
		\item\label{it:214c} and $(x_2, \dots, x_k)$ fails to be $\zeta$-leftconnectable in $\P$. 
	\end{enumerate}
\end{lemma}

\begin{proof}
	We argue by induction on $k$. In the base case $k=2$ the demands~\ref{it:214b} 
	and~\ref{it:214c} contradict each other and, hence, there are indeed no such pairs. 
	Now let $k\ge 3$ and suppose that the lemma is true for $k-1$ in place of $k$. 
	Define $A\subseteq V(\Psi)^k$ to be the set of all $k$-tuples 
	satisfying \ref{it:214a}\,--\,\ref{it:214c}, set
	\[
		A'=\bigl\{(x_1, \dots, x_k)\in A\colon  x_1\in U^{\P}_{(x_2, \dots, x_k)}\bigr\}
	\]
	and define 
	\[
		A''_x = \bigl\{(x_2, \dots, x_k)\in V(\Psi)^{k-1}\colon
			(x, x_2, \dots, x_k)\in A\setminus A'\bigr\}
	\]
	for every $x\in V(\Psi)$. Since
	\[
		|A|=|A'|+\sum_{x\in V(\Psi)}|A''_x|\,,
	\]
	it suffices to show
	\begin{enumerate}[label=\nlabel]
		\item\label{it:11} $|A'| \le \zeta|V(\Psi)|^k$
		\item\label{it:12} and $|A''_x| \le (k-3)\zeta|V(\Psi_x)|^{k-1}$ for every $x\in V(\Psi)$.
	\end{enumerate}

	Now~\ref{it:11} follows from the fact that for $(x_1, \dots, x_k)\in A'\subseteq A$ we have
	\[
		\big|U^{\P}_{(x_2, \dots, x_k)}\big| < \zeta |V(\Psi)|
	\]
	by~\ref{it:214c} and the definition of $\zeta$-leftconnectability. 
	For the proof of~\ref{it:12} we apply the induction hypothesis to the link constellation~$\Psi_x$. 
	Notice that if $(x_2, \dots, x_k)\in A''_x$, then
	\begin{itemize}
		\item $\{x_2, \dots, x_k\}\in E(\Psi_x)$
		\item and $x_k\in V(R^{\Psi_x}_{x_2\dots x_{k-2}})$
	\end{itemize}
	follow from~\ref{it:214a},~\ref{it:214b}, and the definition of $\Psi_x$. 
	Moreover $(x, x_2, \dots, x_k)\in A\setminus A'$ yields $x\notin U^{\P}_{(x_2,\dots, x_k)}$, 
	which together with $\{x,x_2,\dots, x_k\}\in E(\Psi)$ reveals that
	\[
		  (x_3, \dots, x_k) \textrm{ fails to be $\zeta$-leftconnectable in }\Psi_x\,.
	\]
	So altogether the induction hypothesis leads to~\ref{it:12} and the induction step is complete. 
\end{proof}

We proceed with a similar statement that will ultimately assist us in the construction 
of the absorbing path.

\begin{lemma}\label{l:215}
	For $k\ge 2$, $\zeta > 0$, and a $k$-uniform constellation $\P$, there are at 
	most ${(k-2)\zeta|V(\Psi)|^{2k-3}}$ walks $x_1\dots x_{2k-3}$ in $H(\Psi)$ such that 
	\begin{enumerate}[label=\alabel]
		\item\label{it:215a} $x_{k-1}\in V(R^{\P}_{x_k\dots x_{2k-3}})$
		\item\label{it:215b} but $(x_1, \dots, x_{k-1})$ fails to be $\zeta$-leftconnectable.
	\end{enumerate}
\end{lemma}

\begin{proof}
	Again we argue by induction on $k$. In the base case $k=2$ condition~\ref{it:215a} 
	reads $x_1\in V(R^{\P}_\emptyset)$, whereas~\ref{it:215b} demands that $(x_1)$ fails 
	to be $\zeta$-leftconnectable in $\Psi$. As these requirements contradict each other, 
	there are indeed no $1$-vertex walks with the required properties. 
	
	Now let $k\ge 3$ and assume that the lemma is true for $k-1$ instead of $k$. 
	Let $A\subseteq V(\Psi)^{2k-3}$ be the set of all walks $x_1\dots x_{2k-3}$ 
	satisfying~\ref{it:215a} and~\ref{it:215b}, set 
	\[
		A'=\bigl\{(x_1, \dots, x_{2k-3})\in A\colon x_k\in U^{\P}_{(x_1,\dots, x_{k-1})}\bigr\}
	\]
	and put
	\begin{multline*}
		A''_{x,y}=\bigl\{(x_2, \dots, x_{k-1}, x_{k+1}, \dots, x_{2k-3})\in V(\Psi)^{2k-5}\colon \\
			(x,x_2, \dots, x_{k-1}, y, x_{k+1}, \dots, x_{2k-3})\in A\setminus A'\bigr\}
	\end{multline*}
	for all $x, y\in V(\Psi)$. In view of
	\[
		|A|=|A'|+\sum_{(x,y)\in V(\Psi)^2}|A''_{x,y}|
	\]
	it suffices to prove
	\begin{enumerate}[label=\nlabel]
		\item\label{it:21} $|A'|\le \zeta |V(\Psi)|^{2k-3}$
		\item\label{it:22} and $|A''_{x,y}|\le (k-3)\zeta|V(\Psi_y)|^{2k-5}$ for all $x,y \in V(\Psi)$.
	\end{enumerate}

	The estimate~\ref{it:21} follows from the fact that due to~\ref{it:215b} 
	every $(x_1, \dots, x_{2k-3})\in A'\subseteq A$ has the property 
	$\big|U^{\P}_{(x_1, \dots, x_{k-1})}|<\zeta|V(\Psi)\big|$. For the proof of~\ref{it:22} we 
	intend to apply the induction hypothesis to $\Psi_y$. Consider any $(2k-5)$-tuple
	\[
		\seq x=(x_2, \dots, x_{k-1}, x_{k+1}, \dots, x_{2k-3})\in  A''_{x,y}\,.
	\]
	Since $(x_2, \dots, x_{k-1}, y, x_{k+1}, \dots, x_{2k-3})$ is a walk in $H(\Psi)$, we know 
	that $\seq x$ itself is a walk in $H(\Psi_y)$. Moreover,~\ref{it:215a} rewrites as
	\[
		x_{k-1}\in V(R^{\Psi_y}_{x_{k+1} \dots x_{2k-3}})\,.
	\]
	Finally, $y \notin U^{\P}_{(x,x_2,\dots, x_{k-1})}$
	and $\{x, x_2, \dots, x_{k-1}, y\}\in E(\Psi)$ imply that
	\[
		 (x_2,\dots, x_{k-1}) \textrm{ fails to be $\zeta$-leftconnectable in }\P_y\,.
	\]

	Altogether, the $(2k-5)$-tuples in $A''_{x, y}$ have the required properties for 
	applying the induction hypothesis to $\Psi_y$. This proves~\ref{it:22} and completes 
	the induction step.
\end{proof}

We conclude this subsection by introducing one further notion.

\begin{dfn}\label{d:bridge}
	Given $k\ge 2$, $\zeta > 0$, and a $k$-uniform constellation 
	\[
		\P=\bigl(H, \bigl\{R_x\colon x\in V(H)^{(k-2)}\bigr\}\bigr)\,,
	\]
	a $k$-tuple $(x_1, \dots, x_k)\in V(\Psi)^k$ is said to be a {\it $\zeta$-bridge in $\Psi$} if 
	\begin{enumerate}[label=\alabel]
		\item $\{x_1, \dots, x_k\}\in E(\Psi)$,
		\item $(x_1, \dots, x_{k-1})$ is $\zeta$-rightconnectable, 
		\item and $(x_2, \dots, x_k)$ is $\zeta$-leftconnectable.
	\end{enumerate}
\end{dfn}

Such bridges will help us later to construct connecting paths between given $(k-1)$-tuples 
of vertices. The fundamental existence result for such bridges (see Corollary~\ref{c:219}
below) asserts, roughly speaking, that under natural assumptions $k$-uniform constellations
contain many $\zeta$-bridges for sufficiently small values of $\zeta$. 

\subsection{On \texorpdfstring{$(\al, \mu)$}{(alpha, mu)}-constellations}  

In this subsection we study some properties of constellations that can be deduced
from the order and size restrictions~\ref{it:rc1} and~\ref{it:rc2} in Proposition~\ref{prop:robust}
alone without taking the $(\beta, \ell)$-robustness into account. We are thus led to the 
following concept.

\begin{dfn}\label{d:217}
	Let $k\ge 2$ and $\al, \mu >0$. A $k$-uniform constellation $\Psi$ is said to be 
	an {\it $(\al, \mu)$-constellation} if 
	\[
		\delta_{k-2}\bigl(H(\Psi)\bigr)\ge \Big(\frac 59 + \al\Big)\frac{|V(\Psi)|^2}{2}
	\]
	and every $x\in V(\Psi)^{(k-2)}$ satisfies 
	\begin{enumerate}[label=\alabel]
		\item\label{it:217a} $|V(R_x^\Psi)|\ge \bigl(\frac 23 + \frac{\al }{2}\bigr)|V(\Psi)|$
		\item\label{it:217b} as well as 
			$e_{H(\Psi_x)}\bigl(V(R^\Psi_x), V(\Psi)\setminus V(R^\P_x)\bigr)\le \mu |V(\Psi)|^2$.
	\end{enumerate}
\end{dfn}

It turns out that the level of generality provided by this concept is fully appropriate 
for discussing the key parts of our absorbing mechanism and for constructing an important 
building block entering the proof of the Connecting Lemma. Before reaching those results
we record a couple of easier observations. 

\begin{fact}\label{r:218}
	If $\Psi$ denotes a $k$-uniform $(\alpha, \frac\alpha 9)$-constellation for some $\alpha>0$,
	then
	\[
		e\bigl(H(\Psi_x)\bigr) - e(R_x^\P)\le \frac{|V(\P)|^2}{18}
	\]
	holds for every $x\in V(\Psi)^{(k-2)}$. 
\end{fact}

\begin{proof}
	Using both parts of Definition~\ref{d:217} we obtain
	\begin{align*}
		e\bigl(H(\Psi_x)\bigr) - e(R_x^\P)
		&=
		e_{H(\Psi_x)}\bigl(V(\Psi)\setminus V(R^\P_x)\bigr)
			+e_{H(\Psi_x)}\bigl(V(R^\Psi_x), V(\Psi)\setminus V(R^\P_x)\bigr)\\
		&\le
		\Big(\frac 13 - \frac{\al}{2}\Big)^2\frac {|V(\P)|^2}2 + \frac{\al}9 |V(\P)|^2 
		=
		\Big(\frac 1{18}+\frac{\al^2}{8}-\frac{\al}{18}\Big)|V(\P)|^2 
	\end{align*}	
	and it remains to observe that the minimum $(k-2)$-degree condition imposed on $H(\Psi)$
	is only satisfiable for $\alpha\le\frac49$.
\end{proof}

\begin{fact}\label{fact:neu} 
	Suppose that $\Psi$ is a $k$-uniform $(\alpha, \mu)$-constellation. 
	If $x\in V(\Psi)^{(k-2)}$ is arbitrary, then there are at 
	most $\frac{2\mu}{\alpha}|V(\Psi)|$ 
	vertices $z\in V(\Psi)\setminus V(R^\Psi_x)$ with $d_{H(\Psi_x)}(z)>\frac13 (|V(\Psi)|-2)$.
\end{fact}

\begin{proof}
	Definition~\ref{d:217}~\ref{it:217a} tells us 
	that $|V(\Psi)\setminus V(R^\Psi_x)|\le (\frac 13-\frac\alpha2)|V(\Psi)|$. 
	Consequently, the number of edges that every vertex $z$ from the set 
	\[
		Z=\bigl\{z\in V(\Psi)\setminus V(R^\Psi_x)\colon d_{H(\Psi_x)}(z)>\tfrac13 (|V(\Psi)|-2)\bigr\}
	\]
	sends to $V(R^\Psi_x)$ is at least 
	\begin{align*}
		d_{H(\Psi_x)}(z)-\big|V(\Psi)\setminus\bigl(V(R^\Psi_x)\cup\{z\}\bigr)\big|
		&\ge 
		\frac13 \bigl(|V(\Psi)|-2\bigr)-\Bigl(\frac 13-\frac\alpha2\Bigr)|V(\Psi)|+1 \\
		&>
		\frac\alpha2|V(\Psi)|\,.
	\end{align*}
	In combination with Definition~\ref{d:217}~\ref{it:217b} this yields 
	\[
		\frac\alpha2|V(\Psi)||Z|
		\le 
		e_{H(\Psi_x)}\bigl(V(R^\Psi_x), V(\Psi)\setminus V(R^\P_x)\bigr)
		\le 
		\mu |V(\Psi)|^2
	\]
	and the upper bound $|Z|\le \frac{2\mu}{\alpha}|V(\Psi)|$ we are aiming for follows.
\end{proof}

Next, there is an obvious monotonicity statement.

\begin{fact}\label{f:monmu}
	For $k\ge 2$, $\al\ge  \al' > 0$, and $\mu'\ge \mu> 0$, every $k$-uniform 
	$(\al, \mu)$-constellation is an $(\al', \mu')$-constellation as well. \qed
\end{fact}

Link constellations `almost' inherit being $(\alpha, \mu)$-constellations, but since 
we are slightly shrinking the vertex set we need to be careful with clause~\ref{it:217b}
of Definition~\ref{d:217}.
 
\begin{fact}\label{f:linkconst}
	Given $k\ge 2$, $\al>0$, and $\mu'>\mu>0$ there exists a natural number~$n_0$ with
	the following property. If $\P$ denotes a $k$-uniform $(\al, \mu)$-constellation having 
	at least~$n_0$ vertices and $S\subseteq V(\P)$ with $|S|\le k-2$ is arbitrary, 
	then $\P_S$ is a $(k-|S|)$-uniform $(\al, \mu')$-constellation. \qed
\end{fact}

Now we estimate the number of walks of any short length in $\P$, whose starting $(k-1)$-tuple 
is rightconnectable and whose ending $(k-1)$-tuple is leftconnectable. Later we will use these 
walks in the proof of the Connecting Lemma thus gaining control over the length of the connections 
modulo $k$.

\begin{lemma}\label{l:219}
	For $k\ge 2$ and $\al > 0$ let $\P$ be a $k$-uniform $(\al,\frac \al 9)$-constellation.
	Provided that $|V(\P)| \ge \frac{k^2}{\al}$, there are for every positive integer $r$ at least 
   $\frac 1{3^{r+1}}|V(\P)|^{r+k-1}$ walks $x_1x_2\dots x_{r+k-1}$ of length $r$ in $H(\P)$ 
   starting with a {$\frac{1}{k3^{r+1}}$-rightconnectable} $(k-1)$-tuple $(x_1,\dots, x_{k-1})$
   and ending with a $\frac{1}{k3^{r+1}}$-leftconnectable $(k-1)$-tuple $(x_{r+1}, \dots, x_{r+k-1})$. 
\end{lemma}

\let\O=\ccA
\begin{proof}
	Consider the auxiliary $k$-partite $k$-uniform hypergraph $\O$ 
	whose vertex classes $V_1, \dots, V_k$ are copies of $V(\P)$ and whose 
	edges $\{x_1,\dots, x_k\}\in E(\O)$ with 
	\[
		x_1\in V_1, \dots, x_k\in V_k
	\]
	signify that
	\begin{enumerate}[label=\nlabel]
		\item\label{it:4a} $\{x_1, \dots, x_k\}\in E(\P)$,
		\item\label{it:4b} $x_1x_2\in E(R^\P_{x_3\dots x_k})$,
		\item\label{it:4c} and $x_{r+k-2}x_{r+k-1}\in E(R^\P_{x_{r}\dots x_{r+k-3}})$,
	\end{enumerate}
	where the indices in~\ref{it:4c} are to be read modulo $k$. 
	
	In view of $|V(\P)|\ge \frac{k^2}{\al}$ 
	and $\delta_{k-2}(H(\P))\ge \big(\frac 59 + \al\big)\frac{|V(\P)|^2}{2}$ there are at least 
	\[
		(|V(\P)|-k)^{k-2}\cdot\Big(\frac 59+\al\Big)|V(\P)|^2
		\ge 
		\frac 59|V(\P)|^k
	\]
	possibilities $(x_1, \ldots, x_k)\in V_1\times \dots \times V_k$ satisfying~\ref{it:4a}. 
	Among them, there are by Fact~\ref{r:218} at most $\frac{1}{9}|V(\P)|^k$ violating~\ref{it:4b} 
	and at most the same number violating~\ref{it:4c}. Consequently, $e(\O) \ge \frac 13|V(\P)|^k$ 
	and Lemma~\ref{l:217} applied to $\O$ and $d=\frac 13$ shows that there are at 
	least $\frac{1}{3^r}|V(\P)|^{r+k-1}$ walks
	\[
		x_1x_2\dots x_{r+k-1}
	\]
	of length $r$ in $\O$ with $x_1\in V_1, \dots, x_k\in V_k$.
	Among them, there are by~\ref{it:4b} and Lemma~\ref{l:214} applied 
	to $\zeta=\frac{1}{k3^{r+1}}$ at most 
	\[
		\frac{k-2}{k3^{r+1}}|V(\P)|^{r+k-1}
		<
		\frac{1}{3^{r+1}}|V(\P)|^{r+k-1}
	\]
	walks for which $(x_1, \dots, x_{k-1})$ fails to be $\frac{1}{k3^{r+1}}$-rightconnectable. 	
	Similarly~\ref{it:4c} and Lemma~\ref{l:214} ensure that at 
	most $\frac{1}{3^{r+1}}|V(\P)|^{r+k-1}$ of our walks have the defect 
	that $(x_{r+1}, \dots, x_{r+k-1})$ fails to be $\frac{1}{k3^{r+1}}$-leftconnectable. 
	This leaves us with at least
	\[
		\Big(\frac{1}{3^r}-\frac{2}{3^{r+1}}\Big)|V(\P)|^{r+k-1}
		=
		\frac{|V(\P)|^{r+k-1}}{3^{r+1}}
	\]
	walks of the desired form. 
\end{proof}
	
\begin{cor}\label{c:219}
	Given $k\ge 2$ and $\al >0$ let $\Psi$ be a $k$-uniform $(\al,\frac \al 9)$-constellation.
	If $\Psi$ has at least $\frac{k^2}{\al}$ vertices, then the number of its $\frac 1{9k}$-bridges 
	is at least $\frac 19 |V(\Psi)|^k$. 
\end{cor}

\begin{proof}
	Plug $r=1$ into Lemma~\ref{l:219}. 
\end{proof}

The following lemma builds a device that will assist us in the inductive proof of the 
Connecting Lemma in the next section. 

\begin{lemma}\label{l:220}
	Given $k\ge 4$, $\al > 0$, and $\zeta\in \bigl(0,\frac{1}{3^{k+2}}\bigr]$, there exists 
	an integer $n_0$ such that the following holds for every $k$-uniform 
	$(\al, \frac{\al}{10})$-constellation $\P$ on $n\ge n_0$ vertices. 
		
	If two subsets $U, W \subseteq V(\P)$ satisfy $|U|, |W|\ge \zeta n$, then there are 
	at least $\zeta^3n^{2k-2}$ $(2k-2)$-tuples $(u, q_1, \dots, q_{2k-4}, w)\in V(\P)^{2k-2}$ 
	such that 
	\begin{enumerate}[label=\rmlabel]
		\item\label{it:2201} $u\in U$ and $w\in W$ are distinct, 
		\item\label{it:2202} $q_1\dots q_{2k-4}$ is a walk in $H(\P_{uw})$,
		\item\label{it:2203} $(q_1, \dots, q_{k-2})$ is $\zeta^3$-rightconnectable in $\P_u$,
		\item\label{it:2204} and $(q_{k-1}, \dots, q_{2k-4})$ is $\zeta^3$-leftconnectable in $\P_w$.
	\end{enumerate} 
\end{lemma}

\begin{proof}
	Assuming that $n_0$ has been chosen sufficiently large for the subsequent arguments, 
	we commence by considering the $(2k-2)$-tuples $(u, q_1, \dots, q_{2k-4}, w)\in V(\P)^{2k-2}$ 
	satisfying~\ref{it:2201},~\ref{it:2202} as well as the conditions
	\begin{enumerate}[label=\rmlabel]
			\setcounter{enumi}{4}
			\item\label{it:2205} $(q_1,\dots, q_{k-3})$ is $\zeta^3$-rightconnectable in $\P_{uw} $,
			\item\label{it:2206} $(q_k,\dots, q_{2k-4})$ is $\zeta^3$-leftconnectable in $\P_{uw}$.
	\end{enumerate}

	First of all, by $|U|, |W| \ge \zeta n$ and $n\ge n_0\ge 2/\zeta$ there are at 
	least $\frac 12 \zeta^2n^2$ pairs $(u,w)$ in $U\times W$ with $u\neq w$. 
	For each of these pairs Fact~\ref{f:linkconst} tells us that $\P_{uw}$ is 
	a $(k-2)$-uniform $(\al, \frac{\al}{9})$-constellation. Applying the 
	case $r=k-1$ of Lemma~\ref{l:219} to this constellation we learn that the number 
	of $(2k-4)$-tuples $(q_1, \dots, q_{2k-4})\in V(\Psi_{uw})^{2k-4}$
	obeying~\ref{it:2202},~\ref{it:2205}, and~\ref{it:2206} is at  
	least $\frac{1}{3^k}(n-2)^{2k-4}\ge\frac{6}{3^{k+2}}n^{2k-4}\ge 6\zeta n^{2k-4}$.  		
		
	Summarising, the number of $(2k-2)$-tuples $(u, q_1, \dots, q_{2k-4}, w)$ 
	satisfying \ref{it:2201},~\ref{it:2202},~\ref{it:2205}, and~\ref{it:2206}
	is at least $\frac 12 \zeta^2n^2\cdot 6\zeta n^{2k-4}=3\zeta^3n^{2k-2}$. 
	So it suffices to prove that among 
	all $(2k-2)$-tuples $(u, q_1, \dots, q_{2k-4}, w)\in V(\P)^{2k-2}$ there are 
	\begin{enumerate}[label=\nlabel]
		\item\label{it:1} at most $\zeta^3n^{2k-2}$ 
				with~\ref{it:2202},~\ref{it:2205},~$\neg$\ref{it:2203}
		\item\label{it:2} and at most $\zeta^3n^{2k-2}$ 
				with~\ref{it:2202},~\ref{it:2206},~$\neg$\ref{it:2204}.
	\end{enumerate}

	For reasons of symmetry we only need to establish~\ref{it:2}.  
	To this end it is enough to check that for fixed vertices $w, q_1, \dots, q_{2k-4}\in V(\Psi)$ 
	the number of vertices $u$ such that 
	\begin{itemize}
		\item $\{u, q_{k-1}, \dots, q_{2k-4}\}\in E(\P_w)$\,,
		\item \ref{it:2206}, but $\neg$\ref{it:2204}.
	\end{itemize}
	is at most $\zeta^3n$. Now by Definition~\ref{d:210}, the first bullet, and~\ref{it:2206} 
	these vertices satisfy $u\in U^{\P_w}_{(q_{k-1}, \dots, q_{2k-4})}$ 
	and by~$\neg$\ref{it:2204} the latter set has size at most $\zeta^3|V(\Psi_w)|$.
\end{proof}

The last lemma of this subsection will help us to exchange arbitrary vertices by `absorbable'
ones in Section~\ref{sec:Abpa}. Roughly speaking it asserts that 
for $\mu\ll \alpha, k^{-1}$, with few exceptions, the links of two vertices in 
a $k$-uniform $(\alpha, \mu)$-constellation intersect in a substantial number of connectable
$(k-1)$-tuples. 

\begin{lemma}\label{l:221}
	Given $k\ge 3$ and $\al > 0$ set $\mu = \frac 1{10k} \big(\frac \al 2\big)^{2^{k-3}+1}$.
	If $\Psi$ denotes a $k$-uniform $(\al, \mu)$-constellation on $n$ vertices 
	and $\zeta > 0$
	is arbitrary, then there is a set $X\subseteq V(\P)$ of size $|X|\le \frac{\zeta}{\mu}n$ 
	such that for every $a\in V(\P)$ and every $x\in V(\P)\setminus (X\cup \{a\})$ 
	the number of $\zeta$-connectable $(k-1)$-tuples $(x_1, \dots, x_{k-1})$ 
	with $\{x_1, \dots, x_{k-1}\}\in E(\Psi_a)\cap E(\P_x)$ is at least $\mu |V(\P)|^{k-1}$ . 
\end{lemma}		

\begin{proof}
	Set 
	\begin{equation}\label{eq:130}
		\eta = \frac{1}{10}\Big(\frac \al 2\Big)^{2^{k-3}}
	\end{equation}
	and $V=V(\Psi)$. Since $\mu=\frac{\alpha\eta}{2k}$, we have 		
	\begin{equation}\label{eq:1721}
		\max\Big\{\frac{2\mu}{\al},\, 2k\mu \Big\} \le \eta\,.
	\end{equation}

	\smallskip
	 
	\noindent {\bf The choice of $X$.} With every $x\in V$ we shall associate two exceptional sets,
	the idea being that on average these sets can be proved to be small. So there will only be
	few vertices for which one of the exceptional sets is very large and these `unpleasant vertices'
	will form the set $X$. 
	For every vertex not belonging to $X$, we will then be able to show that its link constellation
	intersect the link constellations of all other vertices in the desired way. 

	For an arbitrary $x\in V$ the first of the exceptional sets $A_x$ consists of 
	all $(k-1)$-tuples $(x_1, \dots, x_{k-1})\in V^{k-1}$ satisfying 
	\begin{itemize}
		\item $\{x_1,\dots, x_{k-1}, x\}\in E(\P)$
		\item and $x_1\in V(R^{\P}_{x_3 \dots x_{k-1} x})$
		\item that fail to be $\zeta$-rightconnectable in $\P$. 
	\end{itemize}
	We would like to point out that the second bullet does not involve the vertex $x_2$. Moreover,
	in the special case $k=3$ the condition just means that $x_1\in V(R^\Psi_x)$.
	
	The second exceptional set $B_x$ comprises all $(2k-4)$-tuples 
	$(x_1, \dots, x_{k-1}, x_{k+1}, \dots, x_{2k-3})$ in $V^{2k-4}$ such that 
	\begin{itemize}
		\item $x_1\dots x_{k-1} xx_{k+1}\dots x_{2k-3}$ is a walk in $H(\P)$
		\item and $x_{k-1}\in V(R^{\P}_{xx_{k+1} \dots x_{2k-3}})$,
		\item for which $(x_1, \dots, x_{k-1})$ fails to be $\zeta$-leftconnectable in $\P$.
	\end{itemize}
	Now we define 
	\begin{align*}
		X' &= \bigl\{x\in V\colon |A_x| > 2k\mu |V|^{k-1}\bigr\}\,, \\
		X'' &= \bigl\{x\in V\colon |B_x| > 2k\mu |V|^{2k-4}\bigr\},
	\end{align*}
	and set $X=X'\cup X''$. By Lemma~\ref{l:214} and double counting we have
	\[
		2k\mu |X'||V|^{k-1}\le \sum_{x\in X'}|A_x|\le (k-2)\zeta|V|^k\,,
	\]
	whence $|X'|\le \frac{\zeta}{2\mu}|V|$. Similarly, Lemma~\ref{l:215} yields 
	\[
		2k\mu |X''||V|^{2k-4}\le \sum_{x\in X''}|B_x|\le (k-2)\zeta|V|^{2k-3}\,,
	\]
	which shows that $|X''|\le \frac{\zeta}{2\mu}|V|$ holds as well. Altogether we arrive 
	at the desired estimate 
	\[
		|X|\le |X'|+|X''|\le \frac{\zeta}{\mu}|V|\,.
	\]

	For the rest of the proof we fix two distinct vertices $a, x\in V$ with $x\not\in X$. 
	We are to show that the number of $\zeta$-connectable $(k-1)$-tuples 
	$(x_1, \dots, x_{k-1})$ such that 
	\[
		\{x_1, \dots, x_{k-1}\}\in E(\P_a)\cap E(\P_x)
	\]
	is at least $\mu|V|^{k-1}$. The smallest case $k=3$ receives a separate treatment. 
	
	\smallskip
	
	\noindent{\bf The special case $k=3$.}
	We know that both of the graphs $H(\Psi_a)$ and $H(\Psi_x)$ have at 
	least $\bigl(\frac 59 + \al\bigr)\frac{n^2}{2}$ edges and thus
	they have at least $\bigl(\frac 19 + 2\al\bigr)\frac{n^2}{2}$ edges in common.
	Owing to Fact~\ref{r:218} this shows that $H(\Psi_a)$ and $R_x^\Psi$ have at least $\al n^2$ 
	common edges or, in other words, that there are at least $2\al n^2$ ordered pairs $(x_1, x_2)$ 
	such that $x_1x_2\in E(\P_a) \cap E(R^\Psi_x)$. 
	Due to $|A_x|\le 6\mu n^2$ at most $6\mu n^2$ of these pairs fail to be $\zeta$-rightconnectable. 
	By symmetry, at most the same number of pairs under consideration fails to 
	be $\zeta$-leftconnectable. Altogether, this demonstrates that among the ordered 
	pairs $(x_1, x_2)$ with $x_1x_2\in E(\P_a)\cap E(\P_x)$ there are at least $(2\al - 12\mu)n^2$ 
	which are $\zeta$-connectable. Because of $\mu=\frac{\alpha^2}{120}<\frac\alpha7$ this is more 
	than what we need. 
	
	\smallskip
	
	\noindent{\bf The general case $k\ge 4$.} Our first goal is to count $\zeta$-leftconnectable 
	$(k-1)$-tuples in the intersection of $H(\Psi_a)$ and $H(\Psi_x)$ that satisfy a certain 
	minimum degree condition. 
	
	\begin{claim}\label{clm:neu}
		The number of $\zeta$-leftconnectable $(k-1)$-tuples $(x_1, \dots, x_{k-1})$ such that
		 \begin{enumerate}[label=\nlabel]
			\item\label{it:61} $\{x_1, \dots, x_{k-1}\}\in E(\P_a)\cap E(\P_x)$
			\item\label{it:62} and $d(x_2, \dots, x_{k-1}, x)\ge \frac{n-2}3$
		\end{enumerate}
		is at least $3\eta n^{k-1}$. 
	\end{claim} 
	
	\begin{proof}
		For every vertex $x_{k-1}\in V\setminus \{a,x\}$ we apply Lemma~\ref{l:24} to 
		the $(k-2)$-uniform hypergraphs $H(\P_{xx_{k-1}})$ and $H(\P_{ax_{k-1}})$. 
		This yields a lower bound on the number of $(2k-4)$-tuples 
		\[
			(x_1, \dots, x_{k-1}, x_{k+1}, \dots, x_{2k-3})\in V^{2k-4}
		\]
		such that
		\begin{enumerate}[label=\alabel]
			\item\label{it:1a} $x_1\dots x_{k-1}x x_{k+1}\dots x_{2k-3}$ is a walk in $H(\P)$
			\item\label{it:1b} $\{x_1, \dots, x_{k-1}\}\in E(\P_a)$
			\item\label{it:1c} $d(x_2, \dots, x_{k-1}, x)\ge \frac{n-2}3$
			\item\label{it:1d} and $d(x_{k-1}, x, x_{k+1}, \dots, x_{2k-3})\ge \frac{n-2}3$.
		\end{enumerate}
		Notably, there are $n-2$ possibilities for $x_{k-1}$ and for each of 
		them Lemma~\ref{l:24} yields 
		\[
			  \Big(\frac \al 2\Big)^{2^{(k-2)-1}}n^{2(k-2)-1}
			  \overset{\eqref{eq:130}}{=}
			  10\eta n^{2k-5}
		\]
		possibilities for remaining $2k-5$ vertices. 
		\enlargethispage{1.5em}
		Therefore the number of $(2k-4)$-tuples 
		\[
			(x_1, \dots, x_{k-1}, x_{k+1}, \dots, x_{2k-3})\in V^{2k-4}
		\]
		satisfying~\ref{it:1a}\,--\,\ref{it:1d} is at least $10\eta (n-2)n^{2k-5}$. 		
		
		Because of the minimum $(k-2)$-degree condition $\Psi$ needs to have at least one 
		edge, whence $n\ge k\ge 4$. As this implies $n-2\ge \frac12 n$, the total number 
		of $(2k-4)$-tuples satisfying~\ref{it:1a}\,--\,\ref{it:1d} is at least $5\eta n^{2k-4}$.
 		
		In view of~\ref{it:1d} and Fact~\ref{fact:neu} applied 
		to $\{x, x_{k+1}, \ldots, x_{2k-3}\}$ here in place of $x$ there 
		we know that all but at most $\frac{2\mu}{\al}n^{2k-4}$ of these $(2k-4)$-tuples 
		satisfy
		\begin{enumerate}[label=\alabel, resume]
			\item\label{it:1e} $x_{k-1}\in V(R^\P_{xx_{k+1}\dots x_{2k-3}})$.
		\end{enumerate}

		Now $x\notin X''$ yields $|B_x|\le 2k\mu n^{2k-4}$. So at most $2k\mu n^{2k-4}$ of 
		the $(2k-4)$-tuples satisfying~\ref{it:1a} and~\ref{it:1e} violate
		\begin{enumerate}[label=\alabel, resume]
			\item\label{it:1f} $(x_1,\dots, x_{k-1})$ is $\zeta$-leftconnecctable.
		\end{enumerate}

		Summarising, the number of $(2k-4)$-tuples satisfying~\ref{it:1a}\,--\,\ref{it:1f} 
		is at least 
		\[
			(5\eta-\tfrac{2\mu}\alpha-2k\mu)n^{2k-4}
			\overset{\eqref{eq:1721}}{\ge} 
			3\eta n^{2k-4}\,.
		\]
		Ignoring the vertices $x_{k+1}, \ldots, x_{2k-3}$ as well as the 
		conditions~\ref{it:1d},~\ref{it:1e} we arrive at the desired conclusion. 
	\end{proof}
	
	Now we keep working with the $\zeta$-leftconnectable $(k-1)$-tuples 
	satisfying~\ref{it:61} and~\ref{it:62} obtained in Claim~\ref{clm:neu}. 
	According to~\ref{it:62} and  Fact~\ref{fact:neu} applied $\{x_3,\dots, x_{k-1}, x\}$ 
	here in place of~$x$ there all but at most $\frac{2\mu}\alpha n^{k-1}$ of them have the property  
	\begin{enumerate}[label=\nlabel]
		\setcounter{enumi}{2}
		\item\label{it:63} $x_2\in V(R^\Psi_{x_3\dots x_{k-1}x})$.
	\end{enumerate}
	Moreover, by Definition~\ref{d:217}~\ref{it:217b} applied to 
	the $(k-2)$-set $\{x_3,\dots, x_{k-1}, x\}$ at most $\mu n^{k-1}$ tuples of length $k-1$
	satisfy~\ref{it:61} and~\ref{it:63} but not 
	\begin{enumerate}[label=\nlabel, resume]
		\item\label{it:64} $x_1\in V(R^\Psi_{x_3\dots x_{k-1}x})$.
	\end{enumerate}
	Finally, $x\notin X'$ implies $|A_x|\le 2k\mu n^{k-1}$, so among 
	the $\zeta$-leftconnectable $(k-1)$-tuples 
	satisfying~\ref{it:61}\,--\,\ref{it:64} there are at most $2k\mu n^{k-1}$ 
	for which
	\begin{enumerate}[label=\nlabel, resume]
		\item\label{it:65} $(x_1, \dots, x_{k-1})$ is $\zeta$-rightconnectable
	\end{enumerate}
	fails. In particular, the number 
	of $\zeta$-leftconnectable $(k-1)$-tuples $(x_1, \dots, x_{k-1})$ 
	with~\ref{it:61} and~\ref{it:65} is at 
	least 
	\[
		\left(3\eta-\frac{2\mu}\alpha-\mu-2k\mu\right)n^{k-1} 
		\overset{\eqref{eq:1721}}{\ge} 
		\mu n^{k-1}\,. 
	\]
	Altogether this shows that the number of $(k-1)$-tuples $(x_1, \dots, x_{k-1})$
	that are $\zeta$-leftconnect\-able, $\zeta$-rightconnectable, and satisfy 
	$\{x_1, \ldots, x_{k-1}\}\in E(\P_a)\cap E(\P_x)$ is at least $\mu n^{k-1}$. 
	In view of Definition~\ref{d:212}\ref{it:d212b} this concludes the proof of Lemma~\ref{l:221}. 
\end{proof}

The `connectable' edges in $E(\Psi_a)\cap E(\Psi_x)$ considered in the previous lemma can be used to build paths. 

\begin{cor}\label{c:smallabs}
	For given $k\ge 3$ and $\al > 0$ there exists a natural number $n_0$ such that 
	if $\mu = \frac 1{10k} \big(\frac \al 2\big)^{2^{k-3}+1}$, $\Psi$ is 
	a $k$-uniform $(\al, \mu)$-constellation on $n\ge n_0$ vertices, and $\zeta > 0$ 
	then there exists a set $X\subseteq V(\P)$ with $|X|\le \frac{\zeta}{\mu}n$ such that 
	the following holds. For every $a\in V(\P)$ and every $x\in V(\P)\setminus (X\cup \{a\})$ 
	the number of $(k-1)$-uniform paths $b_1b_2\dots b_{2k-2}$ in $H(\P_a)\cap H(\P_x)$
	such that $(b_1, \dots, b_{k-1})$ and $(b_k, \dots, b_{2k-2})$ are $\zeta$-connectable in $\P$
	is at least $\frac12\mu^kn^{2k-2}$. 
\end{cor}

\begin{proof}
	Let $X$ be the set produced by Lemma \ref{l:221}. Consider two distinct vertices 
   $a, x\in V(\P)$ with $x\not\in X$. Form an auxiliary $(k-1)$-partite $(k-1)$-uniform 
   hypergraph 
   \[
   	\ccB = (V_1\dcup\dots\dcup V_{k-1}, E_\ccB)
	\]
	whose vertex classes 
   are $k-1$ disjoint copies of $V(\P)$ and whose edges $\{v_1,\dots, v_{k-1}\} \in E_\ccB$ 
   with $v_i\in V_i$ for $i\in [k-1]$ correspond to $\zeta$-connectable $(k-1)$-tuples 
   $(v_1,\dots, v_{k-1})$ such that $\{v_1, \dots, v_{k-1}\}\in E(\P_a)\cap E(\P_x)$.
	
	Lemma~\ref{l:221} tells us that 
	\[
		|E_{\ccB}|\ge \mu n^{k-1}\,.
	\]
	Thus Lemma~\ref{l:217} applied to $\ccB$ with $(k-1, \mu, 2k-2)$ here in place 
	of $(k,d,r)$ there yields at least $\mu^kn^{2k-2}$ walks $(b_1, \dots, b_{2k-2})$ in $\ccB$ 
	with $b_1\in V_1, \dots, b_{k-1}\in V_{k-1}$. By the definition of~$\ccB$ each of these 
	walks corresponds to a walk in $H(\P_a)\cap H(\P_x)$ whose first and last $k-1$ vertices
	form a $\zeta$-connectable $(k-1)$-tuple in $\Psi$. At most $O(n^{2k-3})$ of these walks 
	can have repeated vertices and, hence, there are at least
	\[
		\mu^kn^{2k-2} - O(n^{2k-3})\ge \frac{\mu^k}2 n^{2k-2} 
	\]
	paths of the desired from.
\end{proof}

\subsection{On \texorpdfstring{$(\al, \beta, \ell, \mu)$}{(alpha, beta, ell, mu)}-constellations}  

This subsection is dedicated to $(\alpha, \mu)$-constellations~$\Psi$ whose distinguished 
graphs $R^\Psi_x$ have the robustness property delivered by Proposition~\ref{prop:robust}.

\begin{dfn}\label{d:222}
	Let $k\ge 2$, $\al, \beta, \mu >0$ and let $\ell\ge 3$ be odd. 
	A $k$-uniform constellation~$\P$ is said to be an $(\al, \beta, \ell, \mu)$-constellation if 
	\begin{enumerate}[label=\alabel]
		\item\label{it:222a} it is an $(\al, \mu)$-constellation,
		\item\label{it:222b} and for all $x\in V(\P)^{(k-2)}$ and all distinct $y,z\in V(R^\P_x)$ 
			the number of $y$-$z$-paths in~$R^\Psi_x$ of length $\ell$ is 
			at least $\beta |V(\P)|^{\ell-1}$.
	\end{enumerate}
\end{dfn}

The main result of this subsection shows how to expand sufficiently large $k$-uniform
hypergraphs whose minimum $(k-2)$-degree is at least $\bigl(\frac59+\alpha\bigr)\frac{n^2}2$
for appropriate choices of the parameters to such $(\al, \beta, \ell, \mu)$-constellations.
Essentially, the proof of this observation proceeds by applying Proposition~\ref{prop:robust} 
to all link graphs.

\begin{fact}\label{l:223}
	For all $k\ge 2$ and $\al, \mu >0$ there exist $\beta=\beta(\al,\mu)>0$ and an odd 
	integer $\ell=\l(\al,\mu)\ge 3$ such that for sufficiently large $n$ every $k$-uniform 
	$n$-vertex hypergraph $H$ with $\delta_{k-2}(H)\ge \bigl(\frac 59 + \al\bigr)\frac{n^2}{2}$ 
	expands to an $(\al, \beta, \ell, \mu)$-constellation.
\end{fact}

Notice that this result is the reason why the study of $(\al, \beta, \ell, \mu)$-constellations
conducted in the subsequent sections sheds light on Theorem~\ref{t:main}.

\begin{proof}[Proof of Fact~\ref{l:223}]
	For $\alpha$ and $\mu$ Proposition~\ref{prop:robust} delivers some constant $\beta'>0$ and 
	an odd integer $\ell\ge 3$. We contend that $\beta=(2/3)^{\ell-1}\beta'$ and $\ell$ have the 
	desired property. 
	
	To see this, we consider a sufficiently large $k$-uniform hypergraph $H$  
	on $n$ vertices satisfying $\delta_{k-2}(H) \ge \bigl(\frac 59 + \al\bigr)\frac{n^2}{2}$.	
	For every $x\in V(H)^{(k-2)}$ Proposition~\ref{prop:robust} applies to the link graph $H_x$ 
	and yields a $(\beta', \ell)$-robust induced subgraph $R_x\subseteq H_x$ satisfying 
	\begin{enumerate}[label=\rmlabel]
		\item\label{it:71} $|V(R_x)|\ge \bigl(\frac 23 + \frac{\al}{2}\bigr)n$
		\item\label{it:72} and $e_{H_x}(V(R_x), V(H)\setminus V(R_x))\le \mu n^2$.
	\end{enumerate}
	We shall show that 
	\[
		\P = \bigl(H, \bigl\{R_x\colon x\in V(H)^{(k-2)}\bigr\}\bigr)
	\]
	is the desired $(\al, \beta, \ell, \mu)$-constellation. 
	By Definition~\ref{d:217} and~\ref{it:71},~\ref{it:72} above, $\Psi$ is 
	an $(\alpha, \mu)$-constellation, meaning that part~\ref{it:222a} of Definition~\ref{d:222}
	holds. 
	
	Moving on to the second part we fix an arbitrary $(k-2)$-set $x\subseteq V(H)$
	as well as two distinct vertices $y$, $z$ of $R_x$. Since $R_x$ is $(\beta', \ell)$-robust, 
	the number of $y$-$z$-paths in $R_x$ of length~$\l$ is indeed at least 
	\[
		\beta'|V(R_x)|^{\l-1} 
			\overset{\ref{it:71}}{\ge}  
		\Big(\frac 32\Big)^{\ell-1}\beta \cdot \Big(\frac 23 + \frac{\al}{2}\Big)^{\ell-1}n^{\ell-1}
			\ge
		 \beta n^{\l-1}\,. \qedhere
	\]
\end{proof}

The remainder of this subsection deals with the question to what extent being an
$(\al, \beta, \ell, \mu)$-constellation is preserved under taking link constellations
and removing a small proportion of the vertices. Let us observe that if $\Psi$ denotes 
a $k$-uniform $(\al, \beta, \ell, \mu)$-constellation, then for each $x\in V(\P)^{(k-2)}$
the vertices in $x$ are isolated in $H_x$, which by Definition~\ref{d:222}~\ref{it:222b}
implies that they cannot be vertices of $R^\Psi_x$. 
Thus we have $V(R^\Psi_x)\cap x = \varnothing$ for each $x\in V(\P)^{(k-2)}$.

Let us now consider for some $S\subseteq V(\Psi)$ of size $|S|\le k-2$ the $(k-|S|)$-uniform
link constellation $\Psi_S$. For every $x\in V(\Psi_S)^{(k-2-|S|)}$ 
we have $R^{\Psi_S}_x=R^\Psi_{S\cup x}\setminus S=R^\Psi_{S\cup x}$. Therefore,~$\Psi_S$
inherits the property in Definition~\ref{d:222}~\ref{it:222b} from $\Psi$ and together 
with Fact~\ref{f:linkconst} this leads to the following conclusion. 

\begin{fact}\label{f:linkfullc}
	Given $k\ge 2$, $\al,\beta> 0$, $\mu'>\mu>0$ and an odd integer $\ell\ge 3$, 
	there exists a natural number $n_0$ such that the following holds. 
	
	If $\P$ is a $k$-uniform $(\al, \beta, \l, \mu)$-constellation with at least $n_0$ vertices 
	and $S\subseteq V(\P)$ consists of at most $k-2$ vertices, then the $(k-|S|)$-uniform link 
	constellation $\P_S$ is an $(\al, \beta, \l, \mu')$-constellation. \qed
\end{fact}

Next we deal with a similar result allowing vertex deletions as well.

\begin{lemma}\label{lem:1816}
	Given $k\ge 2$, $\al, \beta, \mu >0$ and an odd integer $\ell\ge 3$ set
	\[
		\theta =\min \Big\{ \frac \al 4\,,\, \frac{\beta}{2\ell}\Big\}\,,
	\]
	and let $\P$ be a $k$-uniform $(\al, \beta, \ell, \mu)$-constellation on $n\ge 6k$ vertices. 
	If $S, X\subseteq V(\P)$ are disjoint, $|S|\le k-2$, and $|X|\le \theta n$, then $\P_S-X$ is 
	an $\bigl(\frac \al 2, \frac \beta 2, \ell, 2\mu\bigr)$-constellation. 
\end{lemma}

\begin{proof}
	Let $\P=\bigl(H, \bigl\{R_x\colon x\in V(H)^{(k-2)}\bigr\}\bigr)$ be 
	a $k$-uniform $(\al, \beta, \ell, \mu)$-constellation on $n\ge 6k$ vertices. Recall that this 
	means 		
	\begin{equation}\label{e:m1}
		\delta_{k-2}(H) \ge \Big(\frac 59 + \al \Big)\frac{n^2}{2}\,,
	\end{equation}
	and that for every $x\in V(\P)^{(k-2)}$ the graph $R_x\subseteq H_x$ has the following 
	properties:
	\begin{enumerate}[label = \rmlabel]
		\item\label{it:m1} $|V(R_x)|\ge \bigl(\frac 23 +\frac \al 2 \bigr)n$,
		\item\label{it:m2} $e_{H_x}\bigl(V(R_x), V(\P)\setminus V(R_x)\bigr)\le \mu n^2$,
		\item\label{it:m3} and for all distinct $y,z \in V(R_x)$ the number of $y$-$z$-paths 
			in $R_x$ of length~$\ell$ is at least $\beta n^{\ell-1}$.
		\end{enumerate} 	
	
	Further, let $S, X\subseteq V(\P)$ be any disjoint sets such that $|S|\le k-2$ 
	and $|X|\le \theta n$. We are to prove that 
	\[
		\P_\star 
		= 
		\P_S-X 
		= 
		\bigl(\xoverline H_S-X, \bigl\{R_{x\cup S}- X\colon 
				x\in \bigl(V(H)\setminus (S\cup X)\bigr)^{(k-2-|S|)}\bigr\}\bigr)
	\]
	is a $(k-|S|)$-uniform $\bigl(\frac \al 2, \frac\beta2, \ell 2\mu\bigr)$-constellation,
	i.e., that its underlying hypergraph satisfies an appropriate minimum degree conditions 
	and that the distinguished subgraphs of its link graphs have properties analogous 
	to~\ref{it:m1}\,--\,\ref{it:m3}. 
	
	Because of
	\begin{align*}
		\delta_{k-|S|-2}(\xoverline H_S-X)
		&\ge
		\delta_{k-2}(H-X) 
		\ge 
		\Big(\frac 59 + \al \Big)\frac{n^2}{2}-|X|n \\
		&\ge 
		\Big(\frac 59 + \al \Big)\frac{n^2}{2}-\theta n^2 
		\ge		
		\Big(\frac 59 + \frac \al 2 \Big)\frac{n^2}{2}
		\ge 
		\Big(\frac 59 + \frac \al 2 \Big)\frac{|V(\P_\star)|^2}{2}\,,
	\end{align*}
	where we utilised $\theta\le \frac\alpha 4$ in the penultimate step, the minimum 
	degree of the hypergraph $H(\Psi_\star)=\xoverline H_S-X$ is indeed as large as we 
	need it to be. 
			
	Now let $x\in \bigl(V(\P_\star)\bigr)^{(k-2-|S|)}$ be arbitrary. 
	Since $x\dcup S \in \bigl(V(\P)\setminus X\bigr)^{(k-2)}$, the above statement~\ref{it:m1} 
	entails 
	\begin{align*}
		|V(R^{\Psi_\star}_x)|
		&=
		|V(R_{x\cup S}-  X)|
		\ge 
		\Big(\frac 23 +\frac \al 2 \Big)n - |X|\\
		&\ge 
		\Big(\frac 23 +\frac \al 2 \Big)n - \theta n
		\ge
		\Big(\frac 23 +\frac \al 4 \Big)n 
		\ge 
		\Big(\frac 23 +\frac \al 4 \Big)|V(\P_\star)|\,,
	\end{align*}
	which shows that the required variant of~\ref{it:m1} holds for $\Psi_\star$.
	 
	Next, the graph $H(\Psi_\star)_x=(\xoverline H_S-X)_x$ is a subgraph of $H_{x\cup S}$, 
	so~\ref{it:m2} tells us that 
	\[
		e_{H(\Psi_\star)_x}\bigl(V(R_x^{\Psi_\star}), V(\P_\star)\setminus V(R_x^{\Psi_\star})\bigr)
			\le e_{H_{x\cup S}}\bigl(V(R_{x\cup S}),V(\P)\setminus V(R_{x\cup S})\bigr) 
			\le \mu n^2\,.
	\]		
	From $\theta \le \frac \al 4 \le \frac 19$ and $n \ge 6k$ we conclude
	\[
		|V(\Psi_\star)|
		= 
		n-|X|-|S|
		\ge \Big(1-\frac19-\frac16\Big)n
		=
		\frac{13}{18}n
		>
		\frac n{\sqrt{2}}
	\]
	and thus we arrive indeed at
	\[
		e_{H(\Psi_\star)_x}\bigl(V(R_x^{\Psi_\star}), V(\P_\star)\setminus V(R_x^{\Psi_\star})\bigr)
		\le
		2\mu |V(\P_\star)|^2\,,
	\]
	which concludes the proof that the appropriate modification of~\ref{it:m2} holds for $\Psi_\star$.
	Altogether, we have thereby shown that $\Psi_\star$ is 
	an $\bigl(\frac \al 2, 2\mu\bigr)$-constellation.
	
	Finally we consider distinct vertices $y,z \in V(R_{x\cup S}-X)$ and recall that 
	by~\ref{it:m3} above the number of \mbox{$y$-$z$-paths} in $R_{x\cup S}$ is at 
	least $\beta n^{\ell-1}$. At most $(\ell-1)|X| n^{\ell-2} \le \frac \beta 2 n^{\ell-1}$ 
	of these paths can have an inner vertex in $X$ and, consequently, $R_{x\cup S}-X$ contains 
	at least~$\frac \beta 2 n^{\ell-1}$ such paths. Therefore $\Psi_\star$ is indeed an 
	$\bigl(\frac \al 2, \frac \beta 2, \ell, 2\mu\bigr)$-constellation.
\end{proof}

\section{The Connecting Lemma}
\label{sec:conn}

In this section we establish the Connecting Lemma (Proposition \ref{prop:clk}). Given
an $(\alpha, \beta, \ell, \mu)$-constellation with appropriate parameters 
this result allows us to connect every leftconnectable $(k-1)$-tuple to every 
rightconnectable $(k-1)$-tuple by means of a short path. In the course of proving 
Theorem~\ref{t:main} the Connecting Lemma gets used $\Omega(n)$ times and, essentially, 
it allows us to convert an almost spanning path cover into an almost spanning cycle. 
For some reasons related to our way of employing the absorption method, it will turn out to 
be enormously helpful later if we can guarantee that the number of left-over vertices outside this 
almost spanning cycle is a multiple of $k$. There are several possibilities how one might 
try to accomplish this and our approach is to prove a version of the Connecting Lemma 
with absolute control over the length of the connecting path modulo $k$. 
When closing the almost spanning cycle 
by means of a final application of the Connecting Lemma, we will then be able to prescribe 
in which residue class modulo $k$ the number of left-over vertices is going to be. 
(For a different way to handle such a situation we refer to recent work of Schacht and 
his students~\cite{Mathias}).

The following result is implicit in~\cite{R}*{Proposition 2.6} and after stating it we
shall briefly explain how it can be derived from the argument presented there.

\vbox{
\begin{prop}\label{prop:cl3}
	Depending on $\al, \beta, \zetas >0$ and an odd integer $\ell \ge 3$ there exist 
	a constant $\vartheta_\star = \vartheta_\star(\al, \beta, \ell, \zetas )>0$ and a 
	natural number $n_0$ with the following property. 
	
	If $\Psi$ is a $3$-uniform $(\alpha, \beta, \ell, \frac \al 4)$-constellation 
	on $n\ge n_0$ vertices, $\sa, \sb\in V(\P)^2$ are two disjoint pairs of  
	vertices such that $\sa$ is $\zetas$-leftconnectable and $\sb$ 
	is $\zetas$-rightconnectable, then the number of $\sa$-$\sb$-paths in $H(\P)$ 
	with $3\ell+1$ inner vertices is at least $\vartheta_\star n^{3\ell +1}$. \qed
\end{prop}
}

Observe that the Setup 2.4 we are assuming in~\cite{R}*{Proposition 2.6} is tantamount
to an $(\alpha, \beta, \ell, \frac \al 4)$-constellation. The connectabilty assumptions
in~\cite{R} are slightly different. Writing $\seq{a}=(x, y)$ we were using 
in the proof of~\cite{R}*{Proposition 2.6} that a set called~$U_{xy}$ there, and 
defined to consist of all vertices $u$ with $xy\in E(R^\Psi_u)$, has at least the 
size~$\zeta |V(\Psi)|$. When working with vertices $u\in U_{xy}$, however, we were only 
using $y\in V(R^\Psi_u)$ and $xyu\in E(\Psi)$. For this reason, the entire proof can
also be carried out with the set called~$U^\Psi_{(x, y)}$ here, or in other words it 
suffices to suppose that $\seq{a}$ is $\zeta$-leftconnectable. Similarly, we may 
assume that $\sb$ is $\zeta$-rightconnectable rather than being $\zeta$-connectable in the 
sense of~\cite{R}. Now we introduce the function giving the number of inner vertices in our connections.

\begin{dfn}\label{d:f}
	Given integers $k\ge 3$, $0\le i < k$, and $\ell \ge 3$ we set 
	\[
		f(k,i,\ell)=[4^{k-3}(2\ell+4)-2]k+i\,.
	\]
\end{dfn}

We are now ready to state the $k$-uniform Connecting Lemma. 

\begin{prop}[Connecting Lemma]\label{prop:clk}
	For all $k\ge 3$, $\al, \beta, \zeta >0$, and odd integers $\ell\ge 3$ there 
	exist $\vartheta>0$  
	and $n_0\in \NN$ with the following property.
	
	If $\P$ is a $k$-uniform $(\al, \beta, \ell, \frac{\al}{k+6})$-constellation 
	on $n\ge n_0$ vertices, $\sa, \sb\in V(\P)^{k-1}$ are two disjoint $(k-1)$-tuples 
	such that $\sa$ is $\zeta$-leftconnectable and $\sb$ is $\zeta$-rightconnectable, 
	and $0\le i < k$, then the number of $\seq a$-$\seq b$-paths in $H(\P)$ 
	with $f=f(k,i,\ell)$ inner vertices is at least~$\vartheta n^f$. 
\end{prop}

The proof of this result occupies the remainder of this section and before we begin
we provide a short overview over the main ideas. The plan is to proceed by induction on~$k$.
When we reach a certain value of $k$, most of the work is devoted to showing the weaker 
assertion~$(\Phi_k)$ that there exists at least one number $f_\star=f_\star(k, \ell)$ 
such that the statement of the Connecting Lemma holds for connections with $f_\star$ 
inner vertices. Once we know~$(\Phi_k)$ the induction can be completed by putting 
short `connectable' walks as obtained by Lemma~\ref{l:219} in the middle and connecting 
them with two applications of~$(\Phi_k)$ to $\seq{a}$ and $\seq{b}$. 

The proof of~$(\Phi_k)$
itself is more complicated and starts by applying Lemma~\ref{l:220} to $U^\Psi_{\seq{a}}$
and~$U^\Psi_{\seqq{b}}$ here in place of $U$ and $W$ there. This yields many 
$(2k-2)$-tuples $(u, q_1, \dots, q_{2k-4}, w)$ in $V(\P)^{2k-2}$ which, after some reordering, 
have good chances to end up being middle segments of the desired connections. Applying the 
induction hypothesis to~$\Psi_u$ and $\Psi_w$ we can connect~$\seq{a}$ and~$\seq{b}$ 
by many $(k-1)$-uniform paths to these middle segments and it remains to `augment' these
connections to $k$-uniform paths, which can be done by averaging over many possibilities 
for $u$ and $w$, respectively (see Figure~\ref{fig:con}).      

\begin{proof}[Proof of Proposition~\ref{prop:clk}]
	We proceed by induction on $k$, keeping $\alpha$, $\beta$, and $\ell$ fixed.
	
	\smallskip
	
	\noindent {\bf Choice of constants.}
	Due to monotonicity (see Fact \ref{r:monoton}) we may suppose 
	that $\zeta \le \frac 1{k3^{2k}}$. By recursion on $k\ge 3$ we define for every 
	$\zeta\in\bigl(0, \frac1{k3^{2k}}\bigr]$ a positive real number $\theta(k, \zeta)$.
	Starting with $k=3$ we set 
	\[
		\theta(3, \zeta)=\zeta\bigl(\theta_\star(\alpha, \beta, \ell, \zeta)\bigr)^2
		\quad \text{ for } \quad
		\zeta\in \bigl(0, 3^{-7}\bigr]\,,
	\]
	where $\theta_\star(\alpha, \beta, \ell, \zeta)$ is given by Proposition~\ref{prop:cl3}.
	For $k\ge 4$ and $\zeta\in\bigl(0, \frac1{k3^{2k}}\bigr]$ we stipulate 
	\begin{equation}\label{eq:1253}
		\theta(k, \zeta)=\zeta^{6s+1}\bigl(\theta(k-1, \zeta^3)\bigr)^{4s}\,,
		\quad \text{where } 
		s=4^{k-4}(2\ell+4)\,.
	\end{equation}
	Our goal is to prove the Connecting Lemma with $2\theta(k, \zeta)$ playing the r\^{o}le
	of $\theta$.	

	\smallskip
	
	\noindent {\bf The base case $k=3$.}	
	Suppose that $\Psi$ is a sufficiently 
	large $3$-uniform $(\alpha, \beta, \ell, \frac\alpha9)$-constellation, 
	$i\in \{0,1,2\}$, the pair $\seq a = (a_1, a_2)\in V(\P)^2 $ is $\zeta$-leftconnectable, 
	$\seq b = (b_1, b_2)$ is $\zeta$-rightconnectable, the four vertices $a_1$, $a_2$, $b_1$,
	and $b_2$ are distinct, and $\zeta \le \frac 1{3^7}$. Lemma~\ref{l:219} applied to 
	$(3, i+2)$ here in place of $(k, r)$ there tells us that there are at 
	least $\frac{1}{3^{i+3}}n^{i+4}$ walks $x_1\dots x_{i+4}$ of length $i+2$ in $H(\P)$ 
	whose starting pair $(x_1, x_2)$ is $\zeta$-rightconnectable and whose ending 
	pair $(x_{i+3}, x_{i+4})$ is $\zeta$-leftconnectable. Among these walks at least
	\[
		\Big(\frac{1}{3^{i+3}}-\frac {4(i+4)}n \Big)n^{i+4} 
		> 
		\frac{n^{i+4}}{3^{i+4}}
		\ge 
		\frac{n^{i+4}}{3^6}
		\ge 3\zeta n^{i+4}
	\]
	avoid $\{a_1, a_2, b_1, b_2\}$. 
	
	Now for each of them two applications of Proposition~\ref{prop:cl3} 
	to the $(\al, \beta, \ell, \frac{\al}{9})$-constellation~$\P$ enable us to find
	in $H(\P)$ at least $\vartheta_\star n^{3\ell+1}$ paths $a_1a_2p_1\dots p_{3\ell+1}x_1x_2$
	and at least $\vartheta_\star n^{3\ell+1}$ paths $x_{i+3}x_{i+4}r_1\dots r_{3\ell+1}b_1b_2$
	where $\theta_\star=\theta_\star(\alpha, \beta, \ell, \zeta)$. 
	Altogether, this reasoning leads to at least $3\zeta\theta_\star^2n^f$ 
	walks 
	\[
		a_1a_2p_1\dots p_{3\ell+1}x_1x_2\dots x_{i+3}x_{i+4}r_1\dots r_{3\ell +1}b_1b_2
	\]
	with $f$ inner vertices, where
	\[
		f = 2(3\ell + 1)+(i+4) = 6\ell + 6 + i = f(3, i, \ell)\,.
	\]
	At most $f^2n^{f-1}=o(n^f)$ of these walks fail to be paths and thus the assertion 
	follows.
	
	\smallskip
	
	\noindent {\bf Induction Step.}
	Suppose $k\ge 4$ and that the Connecting Lemma is already known for~$k-1$ instead of $k$. 	
	Set 
	\begin{equation}\label{eq:mte}
		 t=2k(s-1)+2 
		\qand 
		\eta = \zeta^{3s}\bigl(\theta(k-1, \zeta^3)\bigr)^{2s}\,,
	\end{equation}
	where, let us recall, $s=4^{k-4}(2\ell+4)$ was introduced in~\eqref{eq:1253} while we chose 
	our constants.
	Following the plan outlined above, our first step is to prove a Connecting Lemma for 
	connections with~$t$ inner vertices. 
	\begin{claim}\label{c:con}
		For any two disjoint $(k-1)$-tuples $\seq a = (a_1, \dots, a_{k-1})$ 
		and  $\seq b = (b_1, \dots, b_{k-1})$ such that $\seq a$ is $\zeta$-leftconnectable 
		and $\seq b$ is $\zeta$-rightconnectable, the number of $\seq a$-$\seq b$-walks 
		with~$t$ inner vertices in $H(\P)$ is at least $2\eta n^t$. 
	\end{claim}
	\begin{proof}
		The connectability assumptions mean that the sets 
		\begin{align*}
			U&=\{u\in V(\Psi)\colon a_1\dots a_{k-1}u\in E(\Psi) \textrm{ and } 
				(a_2, \dots, a_{k-1}) \textrm{ is $\zeta$-leftconnectable in } \Psi_u\}
			\intertext{ and }
			W&=\{w\in V(\Psi)\colon wb_1\dots b_{k-1}\in E(\Psi) \textrm{ and } 
				(b_{1}, \dots, b_{k-2}) \textrm{ is $\zeta$-rightconnectable in } \Psi_w\}
		\end{align*}
		satisfy $|U|, |W|\ge \zeta n$. Now by $\frac{\al}{k+6}\le \frac{\al}{10}$
		and Fact~\ref{f:monmu} $\Psi$ is an $(\alpha, \frac\alpha{10})$-constellation. 
		Combined with $\zeta \le \frac{1}{3^{k+2}}$ and Lemma~\ref{l:220} this 
		shows that the number of $(2k-2)$-tuples 
		\[
			(u, \seq q, w)= (u, q_1, \dots, q_{2k-4}, w)\in U\times V(\Psi)^{2k-4}\times W			
		\]
		such that
		\begin{enumerate}[label = \alabel]
			\item\label{it:3j} $u\ne w$, 
			\item\label{it:3a} $q_1\dots q_{2k-4}$ is a walk in $H(\Psi_{uw})$,
			\item\label{it:3b} $(q_1, \dots, q_{k-2})$ is $\zeta^3$-rightconnectable in $\P_u$,
			\item\label{it:3c} and $(q_{k-1}, \dots, q_{2k-4})$ is $\zeta^3$-leftconnectable in $\P_w$.
		\end{enumerate}
		is at least $\zeta^3n^{2k-2}$.				
		For later reference we recall that $u\in U$ and $w\in W$ mean
		\begin{enumerate}[label = \alabel, resume]
			\item\label{it:3d} $(a_2, \dots, a_{k-1})$ is $\zeta$-leftconnectable in $\P_u$,
			\item\label{it:3e} $\{a_1, \dots, a_{k-1}, u\}\in E(\Psi)$,
			\item\label{it:3f} $(b_1, \dots, b_{k-2})$ is $\zeta$-rightconnectable in $\P_w$,
			\item\label{it:3g} and $\{w,b_1, \dots, b_{k-1}\}\in E(\Psi)$.
		\end{enumerate}
		Notice that by Fact~\ref{f:linkfullc} the link constellation of every vertex 
		is a $(k-1)$-uniform $(\al, \beta, \ell, \frac{\al}{k+5})$-constellation
		and that $f(k-1, 1, \ell)=(k-1)(s-2)+1$. Now for every $(2k-2)$-tuple $(u, \seq q, w)$
		satisfying~\ref{it:3j}\,--\,\ref{it:3g} 
		we apply the induction hypothesis twice with $(\zeta^3, 1)$ here in place 
		of $(\zeta, i)$ there. First, by~\ref{it:3b} and~\ref{it:3d} we can connect 
		$(a_2, \dots, a_{k-1})$ to $(q_1, \dots, q_{k-2})$ in $\Psi_u$, thus getting at least 
		$2\theta(k-1, \zeta^3)(n-1)^{(k-1)(s-2)+1}$
		\begin{enumerate}[label=\alabel, resume]
				\item\label{it:3h} walks $a_2\dots a_{k-1}p_1\dots p_{(k-1)(s-2)+1}q_1\dots q_{k-2}$ 
					in $\Psi_u$
		\end{enumerate}
		with $f(k-1,1,\l)$ inner vertices. 
		Second,~\ref{it:3c} and~\ref{it:3f} allow us to connect $(q_{k-1}, \dots, q_{2k-4})$ to 
		$(b_1, \dots, b_{k-2})$ in $\Psi_w$ by at least $2\theta(k-1, \zeta^3)(n-1)^{(k-1)(s-2)+1}$
		\begin{enumerate}[label=\alabel, resume]
				\item\label{it:3i} walks $q_{k-1}\dots q_{2k-4}r_{{(k-1)(s-2)+1}}\dots r_1b_1\dots b_{k-2}$ 
					in $\Psi_w$. 
		\end{enumerate}
		Altogether, the number of $\bigl((k-1)(2s-2)+2\bigr)$-tuples 
		\[
			(u,  \seq p, \seq q,\seq r,w) \in 
				U\times V(\P)^{(k-1)(s-2)+1}\times V(\P)^{2k-4}\times V(\P)^{(k-1)(s-2)+1}\times W
		\]
		with \ref{it:3j}\,--\,\ref{it:3i}, where
		\[
			\seq p=(p_1, \dots, p_{{(k-1)(s-2)+1}}) \qand 
			\seq r=(r_{{(k-1)(s-2)+1}}, \dots, r_1)\,,
		\]
		is at least $4\zeta^3\bigl(\theta(k-1, \zeta^3)\bigr)^2(n-1)^{(k-1)(2s-4)+2}n^{2k-2}\ge 
			2\zeta^3\bigl(\theta(k-1, \zeta^3)\bigr)^2n^{(k-1)(2s-2)+2}$.
		
		Roughly speaking, we plan to derive the $\seq{a}$-$\seq{b}$-paths we are supposed to 
		construct from these 
		$\bigl((k-1)(2s-2)+2\bigr)$-tuples by taking many copies of $u$ and $w$ and inserting 
		them in appropriate positions into $(\seq{a}, \seq{p}, \seq{q}, \seq{r}, \seq{b})$. 
		To analyse the number of ways of doing this, we consider the auxiliary $3$-partite 
		$3$-uniform hypergraph $\ccA$ with vertex 
		classes $U^\star$, $M$, and~$W^\star$, where $U^\star$ and $W^\star$ are two disjoint copies 
		of $V(\P)$, while $M=V(\Psi)^{(k-1)(2s-2)}$.
		
		We represent the vertices in $M$ as sequences
		\[
			\seq{m}
			=
			(\seq p, \seq q, \seq r)
			=
			(p_1, \ldots, p_{(k-1)(s-2)+1},q_1,\dots,q_{2k-4},r_{(k-1)(s-2)+1}, \ldots, r_1)
			\,.
		\]
		The edges of $\ccA$ are defined to be the triples $\{u, \seq{m}, w\}$ 
		with $u\in U\subseteq U^\star$, $\seq{m}\in M$, and $w\in W\subseteq W^\star$, for which  
		the $\bigl((k-1)(2s-2)+2\bigr)$-tuple $(u, \seq{m}, w)$ satisfies~\ref{it:3j}\,--\,\ref{it:3i}. 		
		We have just proved that  
		\begin{equation}\label{eq:1244}
			e(\ccA)\ge 2\zeta^3\bigl(\theta(k-1, \zeta^3)\bigr)^2n^{(k-1)(2s-2)+2}
			= 
			2\zeta^3\bigl(\theta(k-1, \zeta^3)\bigr)^2 |U^\star||M||W^\star|\,.
		\end{equation}
		By the (ordered) bipartite link graph of a vertex $\seq{m}\in M$ we mean the set of pairs
		\[
			\ccA_{\seq{m}}=\bigl\{(u, w)\in U\times W\colon u\seq{m}w\in E(\ccA)\bigr\}\,.
		\]
		The convexity of the function $x\longmapsto x^s$ on $\RR_{\ge 0}$ yields
		\begin{align}\label{eq:1251}
			\sum_{\seq{m}\in M} |\ccA_{\seq{m}}|^{s}
			&\ge
			|M|\biggl(\frac{e(\ccA)}{|M|}\biggr)^{s}
			\overset{\eqref{eq:1244}}{\ge}
			n^{(k-1)(2s-2)}\bigl(2\zeta^3\bigl(\theta(k-1, \zeta^3)\bigr)^2 n^2\bigr)^s \nonumber\\
			& \ge
			2\zeta^{3s} \bigl(\theta(k-1, \zeta^3)\bigr)^{2s} n^{k(2s-2)+2}
			\overset{\eqref{eq:mte}}{=} 2\eta n^t\,.
		\end{align}
	 	In other words, the number of $t$-tuples 
		\[
			(\seq u, \seq p, \seq q, \seq r, \seq w)\in 
			U^{s}\times V(\Psi)^{(k-1)(s-2)+1}\times V(\Psi)^{2k-4}\times 
				V(\Psi)^{(k-1)(s-2)+1}\times W^{s}
		\]
		with 
		\[
			(u_1, w_1), \ldots, (u_{s}, w_{s})\in \ccA_{\seq{m}}\,,
		\]
		where
		\[
			\seq u = (u_1, \dots, u_s)\,, 
			\quad \seq w=(w_1, \dots, w_s)\,,
			\qand \seq{m}=(\seq p, \seq q, \seq r)\in M\,,
		\]
		is at least $2\eta n^t$. So to conclude the proof of Claim~\ref{c:con} it suffices 
		to show that for every such $t$-tuple the sequence 
		\begin{multline*}
			a_1\dots a_{k-1}u_1p_1\dots p_{k-1} u_2 p_k\dots p_{2k-2}u_3\dots u_{s-2} 
				p_{(k-1)(s-3)+1}\dots p_{(k-1)(s-2)}
			u_{s-1}p_{(k-1)(s-2)+1}\\
			q_1\dots q_{k-2} u_{s} w_{s}q_{k-1}\dots q_{2k-4}\\
			r_{(k-1)(s-2)+1} w_{s-1} r_{(k-1)(s-2)} \dots r_{(k-1)(s-3)+1} w_{s-2}\dots w_3
				r_{2k-2} \dots r_{k} w_2 r_{k-1}\dots r_{1}w_{1}b_{1}\dots b_{k-1}
		\end{multline*}
		indicated in Figure~\ref{fig:con} is an $\seq a$-$\seq b$-walk in $H(\Psi)$.
		
			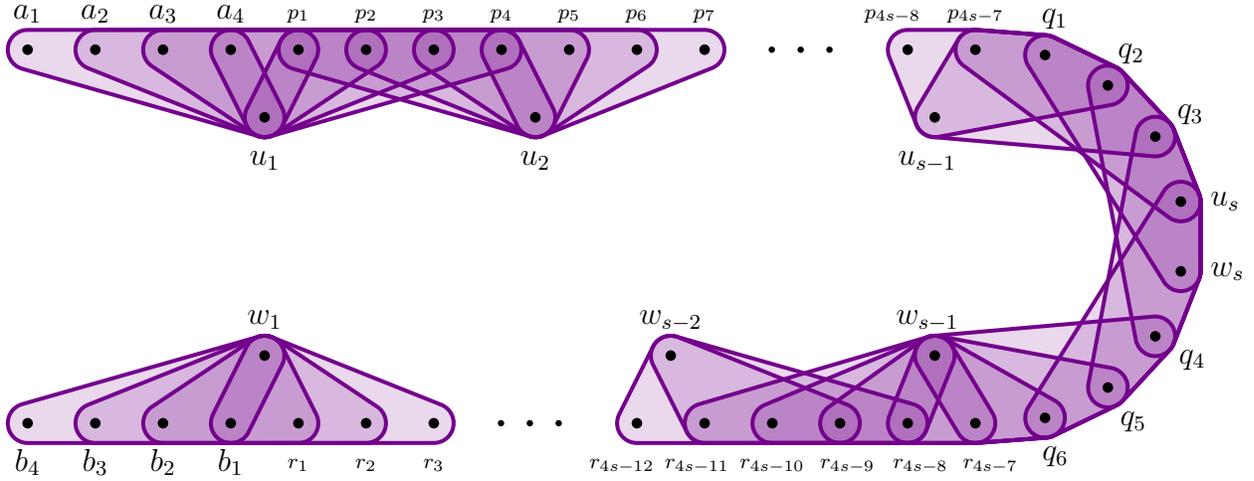
\begin{figure}[th]
	\begin{tikzpicture}[scale=1]
	
		\def\an{10.7};			
		\def\ra{2.5cm};        
		\def\s{.9};

		\coordinate (c6) at (-9*\an:\ra);	
		\coordinate (q6) at (-7*\an:\ra);
		\coordinate (q5) at (-5*\an:\ra);
		\coordinate (q4) at (-3*\an:\ra);
		\coordinate (q3) at (3*\an:\ra);
		\coordinate (q2) at (5*\an:\ra);
		\coordinate (w) at (-\an:\ra);
		\coordinate (u) at (\an:\ra);
		\coordinate (q1) at (7*\an:\ra);
		\coordinate (a7) at (9*\an:\ra);
						
		\coordinate (a1) at ($(a7)+( -14*\s, 0)$);
		\coordinate (a2) at ($(a7)+( -13*\s, 0)$);
		\coordinate (a3) at ($(a7)+( -12*\s, 0)$);
		\coordinate (a4) at ($(a7)+( -11*\s, 0)$);
		\coordinate (p1) at ($(a7)+( -10*\s, 0)$);
		\coordinate (p2) at ($(a7)+(-9*\s, 0)$);
		\coordinate (p3) at ($(a7)+(-8*\s, 0)$);
		\coordinate (p4) at ($(a7)+(-7*\s, 0)$);
		\coordinate (p5) at ($(a7)+(-6*\s, 0)$);
		\coordinate (p6) at ($(a7)+(-5*\s,0)$);
		\coordinate (p7) at ($(a7)+(-4*\s, 0)$);
		\coordinate (p8) at ($(a7)+(-\s, 0)$);
		\coordinate (p65) at ($(p7)+(-\s, 0)$);
		\coordinate (p78) at ($(p7)!.5!(p8)$);
		
		\coordinate (b4) at ($(c6)+(-14*\s, 0)$);
		\coordinate (b3) at ($(c6)+(-13*\s, 0)$);
		\coordinate (b2) at ($(c6)+(-12*\s, 0)$);
		\coordinate (b1) at ($(c6)+(-11*\s, 0)$);
		\coordinate (r1) at ($(c6)+(-10*\s, 0)$);
		\coordinate (r2) at ($(c6)+(-9*\s, 0)$);
		\coordinate (r3) at ($(c6)+(-8*\s, 0)$);
		\coordinate (r4) at ($(c6)+(-5*\s, 0)$);
		\coordinate (r5) at ($(c6)+(-4*\s, 0)$);
		\coordinate (r34) at ($(r3)!.5!(r4)$);
		\coordinate (r6) at ($(c6)+(-3*\s, 0)$);
		\coordinate (r7) at ($(c6)+(-2*\s, 0)$);
		\coordinate (r8) at ($(c6)+(-\s, 0)$);
				
		\def\p{.9};     
		
		\coordinate (w1) at ($(b1)!.5!(r1)+(0,\p)$);
		\coordinate (w2) at ($(r4)!.5!(r5)+(0,\p)$);
		\coordinate (w3) at ($(r8)!.4!(c6)+(0,\p)$);
		
		\coordinate (u1) at ($(a4)!.5!(p1)+(0,-\p)$);
		\coordinate (u2) at ($(p4)!.5!(p5)+(0,-\p)$);
		\coordinate (u3) at ($(p8)!.4!(a7)+(0,-\p)$);
		
		\def \co{violet!90!blue}

		\foreach \ii/\ij/\ik/\im/\in in {
				u1/a1/a2/a3/a4,
				u1/a2/a3/a4/p1,
				u1/a3/a4/p1/p2,
				u1/a4/p1/p2/p3,
				u1/p1/p2/p3/p4,
				u2/p1/p2/p3/p4,
				u2/p2/p3/p4/p5,
				u2/p3/p4/p5/p6,
				u2/p4/p5/p6/p7,
				u3/p8/a7/q1/q2,
				u3/a7/q1/q2/q3,
				a7/q1/q2/q3/u,
				q1/q2/q3/u/w,
				q2/q3/u/w/q4,
				q3/u/w/q4/q5,
				u/w/q4/q5/q6,
				w1/b1/b2/b3/b4,
				w1/r1/b1/b2/b3,
				w1/r2/r1/b1/b2,
				w1/r3/r2/r1/b1,				
				w2/r7/r6/r5/r4,
				w2/r8/r7/r6/r5,
				w3/r8/r7/r6/r5,
				w3/c6/r8/r7/r6,
				w3/q6/c6/r8/r7,
				w3/q5/q6/c6/r8,
				w3/q4/q5/q6/c6}
			{\pedge{(\ii)}{(\ij)}{(\ik)}{(\im)}{(\in)}{7.5pt}{1.5pt}{\co}{\co,opacity=0.15};}

		\def\x{.42};  
		\def\y{.5}; 
		\def \yy{.63};
	
		\foreach \i in {1,...,6}{
			\fill  (q\i) circle (2pt);}
		\foreach \i in {a7,c6,w,u}{
			\fill  (\i) circle (2pt);}
		\foreach \i in {1,...,8}{
			\fill  (p\i) circle (2pt);
			\fill  (r\i) circle (2pt);}
		\foreach \i in {1,...,4}{
			\fill  (a\i) circle (2pt);
			\fill  (b\i) circle (2pt);
			\node[anchor=base] at ($(a\i)+(0,\x)$) {$a_{\i}$}; 
			\node[anchor=base] at ($(b\i)+(0,-\yy)$) {$b_{\i}$};}
		\foreach \i in {1,...,3}{
			\fill  (w\i) circle (2pt);
			\fill  (u\i) circle (2pt);}
	
		\node at (r34) {\Huge $\dots$};
		\node at (p78) {\Huge $\dots$};
		
	\foreach \i in {1,...,7}{
		\node[anchor=base] at ($(p\i)+(0,\x)$) {\tiny $p_{\i}$};}
		\node[anchor=base] at ($(p8)+(-.2,\x)$) {\tiny $p_{4s-8}$};
		
		\node[anchor=mid] at ($(r1)+(0,-\y)$) {\tiny $r_1$};
		\node[anchor=mid] at ($(r2)+(0,-\y)$) {\tiny $r_{2}$};
		\node[anchor=mid] at ($(r3)+(0,-\y)$) {\tiny $r_{3}$};
		\node[anchor=mid] at ($(r4)+(-.2,-\y)$) {\tiny $r_{4s-12}$};
		\node[anchor=mid] at ($(r5)+(-.1,-\y)$) {\tiny $r_{4s-11}$};
		\node[anchor=mid] at ($(r6)+(0,-\y)$) {\tiny $r_{4s-10}$};
		\node[anchor=mid] at ($(r7)+(.1,-\y)$) {\tiny $r_{4s-9}$};
		\node[anchor=mid] at ($(r8)+(.17,-\y)$) {\tiny $r_{4s-8}$};
		\node[anchor=mid] at ($(c6)+(.2,-\y)$) {\tiny $r_{4s-7}$};
		\node[anchor=base] at ($(a7)+(0,\x)$) {\tiny $p_{4s-7}$};
		
		\node at ($(q5)+(-5*\an:.57)$) {$q_5$};
		\node at ($(q4)+(-3*\an:.58)$) {$q_4$};
		\node at ($(q6)+(-7*\an:.53)$) {$q_6$};
		\node[anchor=west] at ($(w)+(.25,0)$) {$w_{s}$};
		\node[anchor=west] at ($(u)+(.25,0)$) {$u_{s}$};
		\node at ($(q3)+(3*\an:.55)$) {$q_3$};
		\node at ($(q2)+(5*\an:.52)$) {$q_2$};
		\node at ($(q1)+(7*\an:.50)$) {$q_1$};
		
		\node[anchor=base] at ($(w1)+(0,\x)$) {$w_1$};
		\node[anchor=base] at ($(w2)+(0,\x)$) {$w_{s-2}$};
		\node[anchor=base] at ($(w3)+(-.1,\x)$) {$w_{s-1}$};
		
		\node[anchor=base] at ($(u1)+(0,-\yy)$) {$u_1$};
		\node[anchor=base] at ($(u2)+(0,-\yy)$) {$u_2$};
		\node[anchor=base] at ($(u3)+(-.1,-\yy)$) {$u_{s-1}$};
	\end{tikzpicture}
	\caption{Connecting $(a_1,a_2,a_3,a_4)$ and $(b_1,b_2,b_3,b_4)$ in a $5$-uniform constellation.}
	\label{fig:con}
\end{figure}

		We shall now argue that this follows from the 
		fact that for each~$j\in [s]$ the conditions~\hbox{\ref{it:3j}\,--\,\ref{it:3i}} hold for $u_j$ 
		and $w_j$ here in place of $u$ and $w$ there.
		
		The first of the required edges is provided by the case $u=u_1$ of~\ref{it:3e}. 
		Together with~\ref{it:3h} this shows that the initial segment
		\begin{multline*}
			a_1a_2\dots a_{k-1}u_1p_1\dots p_{k-1} u_2 p_k\dots p_{2k-2}u_3\dots u_{s-2} 
			p_{(k-1)(s-3)+1}\dots p_{(k-1)(s-2)} \\ u_{s-1}p_{(k-1)(s-2)+1}q_1\dots q_{k-2} u_{s}
		\end{multline*}
		is a walk in $H(\P)$. Similarly, by~\ref{it:3g} and~\ref{it:3i} the terminal segment 		
		\begin{multline*}
			w_{s}q_{k-1}\dots q_{2k-4}r_{(k-1)(s-2)+1} w_{s-1} r_{(k-1)(s-2)} \dots r_{(k-1)(s-3)+1} 
				w_{s-2}\dots w_3 \\
			r_{2k-2} \dots r_{k} w_2 r_{k-1}\dots r_{1}w_{1}b_{1}\dots b_{k-2}b_{k-1}
		\end{multline*}
		is a walk in $H(\Psi)$. Finally, the middle part
		\[
			q_1\dots q_{k-2} u_{s} w_{s}q_{k-1}\dots q_{2k-4}
		\]
		is a walk in $H(\P)$, because by~\ref{it:3a} we know that $q_1\dots q_{2k-4}$ is a walk 
		in $H(\Psi_{u_s w_s})$. 
	\end{proof}
	
	Returning to the induction step, we consider $i\in \{0,1,\dots, k-1\}$, a $\zeta$-leftconnectable
	$(k-1)$-tuple $\seq a\in V(\P)^{k-1}$, and a $\zeta$-rightconnectable $(k-1)$-tuple $\seq b$
	such that $\seq{a}$ and~$\seq{b}$ have no vertices in common. Plugging $r=i+k-3$ into 
	Lemma~\ref{l:219} we obtain at least $\frac 1 {3^{i+k-2}}n^{i+2k-4}$ walks $x_1 \dots x_{i+2k-4}$ 
	of length $i+k-3$ in $H(\P)$ that start with a \hbox{$\zeta$-rightconnectable} $(k-1)$-tuple and 
	end with a $\zeta$-leftconnectable $(k-1)$-tuple. Of these walks, at least
	\[
		\Big(\frac{1}{3^{i+k-2}}-\frac {2(k-1)(i+2k-4)}n \Big)n^{i+2k-4} 
		> 
		\frac{n^{i+2k-4}}{3^{i+k-1}}
		> 
		\frac{n^{i+2k-4}}{3^{2k}}
		> \zeta n^{i+2k-4}
	\]
	have no common vertices with $\seq{a}$ and $\seq{b}$. For each such walk, Claim~\ref{c:con} 
	tells us that we can connect $\seq{a}$ to $(x_1,\dots, x_{k-1})$ in at least $2\eta n^t$
	ways by a walk with $t$ inner vertices, and the same applies to connections from
	$(x_{i+k-2},\dots, x_{i+2k-4})$ to $\seq b$. 
	
	Altogether this reasoning leads to $4\zeta \eta^2 n^f=4\theta(k, \zeta) n^f$ walks in $H(\P)$ 
	from $\seq a$ to $\seq b$ with~$f$ inner vertices, where
	\begin{align*}
		f &= 2t + (i+2k-4)=2(2ks-2k+2)+(i+2k-4)\\
		 &=(4s-2)k+i=[4^{k-3}(2\ell+4)-2]k+i=f(k,i,\ell)\,.
	\end{align*}

	As usual, at most $O(n^{f-1})$ of these walks can fail to be paths. So, in particular, 
	there exist at least $2\theta(k, \zeta) n^f$ paths from $\seq{a}$ to $\seq{b}$ possessing $f$
	inner vertices. This completes the induction step and, hence, the proof of the Connecting Lemma.
\end{proof}

\section{Reservoir Lemma}\label{sec:reservoir}

In this section we discuss a standard device occurring in many applications of the absorption 
method: the reservoir. The problem addressed by the Reservoir Lemma is that while the Connecting 
Lemma delivers many connections for any two disjoint connectable $(k-1)$-tuples, it gives us 
no control where the inner vertices of these connections are. Thus it might happen that each of these
connections has an inner vertex which is `unavailable' to us, because we already assigned 
a different r\^{o}le to it in the Hamiltonian cycle we are about to construct. To avoid this 
problem, one fixes a small random subset of the vertex set, called the reservoir, and decides that 
the vertices in the reservoir will only be used for the purpose of connecting pairs of $(k-1)$-tuples
by means of short paths. 

\begin{prop}[Reservoir Lemma]\label{prop:reservoir}
	Suppose that $k\ge 3$, $\al, \beta, \xi, \zetass >0$, and that $\ell\ge 3$ is an odd integer. 
	If $\thetass=\theta(k,\alpha,\beta,\l,\zetass)$ is provided by Proposition~\ref{prop:clk}, then
	there exists some $n_0\in \NN$ such that for every 
	$k$-uniform $(\al, \beta, \ell, \frac\alpha{k+6})$-constellation $\P$ on $n\ge n_0$ vertices 
	there exists a subset $\cR\subseteq V(\P)$ with the following properties.
	\begin{enumerate}[label=\rmlabel]
		\item\label{it:R1} We have $\frac12 \xi n \le |\cR| \le \xi n$.
		\item\label{it:R2} For all pairs of disjoint $(k-1)$-tuples $\sa,\sb\in V(\P)^{k-1}$ 
			such that $\sa$ is $\zetass$-leftconnect\-able and $\sb$ is $\zetass$-rightconnectable 
			in $\P$, and for every $i\in [0, k)$, the number of $\sa$-$\sb$-paths in~$H(\P)$
			with $f=f(k,i,\ell)$ inner vertices all of which belong to $\cR$ 
			is at least~$\frac12\thetass|\cR|^{f}$.
	\end{enumerate}
\end{prop}
	
Since the proof of this result is quite standard, we will only provide a brief sketch here.
It suffices to prove that the binomial random subset $\cR\subseteq V(\Psi)$ including every 
vertex independently with probability $\frac34\xi$ a.a.s.\ has the 
properties~\ref{it:R1} and~\ref{it:R2}. 
Now~\ref{it:R1} is a straightforward consequence of Chernoff's inequality. As there are only 
polynomially many possibilities for $(\sa, \sb, i)$ in~\ref{it:R2}, it suffices to show that for 
each of them the probability that there are fewer than $\frac12\thetass|\cR|^{f}$ paths of the 
desired form is at most $\exp\bigl(-\Omega(n)\bigr)$. This can in turn be established by applying 
the Azuma-Hoeffding inequality to the at least~$\thetass n^f$ such paths in $V(\Psi)^f$
delivered by Proposition~\ref{prop:clk}.
For further details we refer to~\cite{Y}*{Proposition~4.1}, where we gave a full account of the 
argument for $k=4$.  

Let us emphasise again that the set $\cR$ provided by Proposition~\ref{prop:reservoir}
is called the {\it reservoir}. The connections in~\ref{it:R2} whose inner vertices belong 
to $\cR$ are called paths {\it through} $\cR$.
  
In the proof of Theorem~\ref{t:main} we shall repeatedly connect suitable tuples through 
the reservoir. Whenever such a connection 
is made, some of the vertices of the reservoir are used and the part of $\cR$ still 
available for further connections shrinks. Although the reservoir gets used~$\Omega(|V(\P)|)$ 
times, we shall be able to keep an appropriate version of property~\ref{it:R2} of the reservoir 
intact throughout this process. 
	
\begin{cor}\label{lem:use-reservoir}
	Let a sufficiently large $k$-uniform 
	$(\alpha, \beta, \ell, \frac\alpha{k+6})$-constellation~$\Psi$
	as well as a reservoir $\cR\subseteq V(\Psi)$ as in Proposition~\ref{prop:reservoir} be given. 
	Moreover, let $\cR'\subseteq \cR$ be an arbitrary set 
	with $|\cR'|\le \frac{\xi\thetass}{4^kk\ell}n$. If $\sa, \sb\in V(\P)^{k-1}$ are two 
	disjoint $(k-1)$-tuples such that~$\sa$ is \hbox{$\zetass$-leftconnectable} and $\sb$ 
	is $\zetass$-rightconnectable, then for every $i\in [0, k)$ there is an $\sa$-$\sb$-path 
	through $\cR\setminus \cR'$ with $f(k,i,\ell)$ inner vertices. 
\end{cor}
	
	\begin{proof}
		Set $f=f(k, i, \ell)$ and recall that $f(k,i,\l)=(4^{k-3}(2\l+4)-2)k+i< 4^{k-2}k\ell$. 
		So the lower bound in Proposition~\ref{prop:reservoir}\,\ref{it:R1} together with 
		the bound on $|\cR'|$ yields
		\[
			|\cR'|
			\le 
			\frac{\thetass |\cR|}{4^{k-1}k\ell}
			\le
			\frac{\thetass |\cR|}{4f} \,.
		\]

		Consider all $\seq{a}$-$\seq{b}$-paths through $\cR$ with $f$ inner vertices. 
		On the one hand, there are at least $\frac\thetass2 |\cR|^f$ of them due to  
		Proposition~\ref{prop:reservoir}\,\ref{it:R2}. On the other hand, there are
		at most 
		\[
			f|\cR'||\cR|^{f-1} \le \frac\thetass4 |\cR|^f 
		\]
		such paths having an inner vertex in $\cR'$. Consequently, 
		at least $\frac\thetass2 |\cR|^f-\frac\thetass4 |\cR|^f>0$ of our paths
		have all their inner vertices in $\cR\setminus\cR'$.
	\end{proof}

\section{The absorbing path}
\label{sec:Abpa}

\subsection{Overview}
\label{sec:abs_intro}
In this section we establish that for $\mu\ll\alpha$ every sufficiently 
large $(\alpha, \beta, \ell, \mu)$-constellation contains an~\emph{absorbing path~$P_A$}, 
whose main property is that it can `absorb' an arbitrary but not too large set
of vertices whose cardinality is a multiple of~$k$. Thus the problem of proving 
Theorem~\ref{t:main} gets reduced to the simpler task of finding an almost spanning cycle 
containing the absorbing path and missing a number of vertices that is divisible by $k$.
In order to have a realistic chance to include the absorbing path into such a cycle we 
make sure that its first and last $(k-1)$-tuple is connectable. Moreover, we will need to be 
able to work outside a forbidden `reservoir set' that later will have been selected in advance.  

\begin{prop}[Absorbing Path Lemma]\label{prop:apl}
	Given $k\ge 3$, $\al > 0$, $\beta >0$, and an odd integer $\ell \ge 3$ there exist
	constants $\zeta = \zeta(\al, k)$,  $\thetas=\thetas(k, \al, \beta, \l, \zeta)>0$ 
	and an integer $n_0$ with the following property. 
	
	Suppose that $\P$ is a $k$-uniform $(\al, \beta, \ell, \mu)$-constellation 
	with $\mu = \frac{1}{10k}\big(\frac{\al}{2}\big)^{2^{k-3}+1}$ on $n\ge n_0$ vertices. 
	If $\cR\subseteq V(\P)$ with $|\cR|\le \thetas^2n$ is arbitrary, then there exists a 
	path $P_A\subseteq H(\Psi)-\cR$ such that 
	\begin{enumerate}[label=\rmlabel]
			\item\label{it:apl1} $|V(P_A)|\le \thetas n$,
			\item\label{it:apl2} the starting and ending $(k-1)$-tuple of $P_A$ 
				are $\zeta$-connectable, 
			\item\label{it:apl3} and for every subset $Z\subseteq V(\P)\setminus V(P_A)$ 
				with $|Z|\le 2 \thetas^2 n$ and $|Z|\equiv 0\pmod{k}$, there exists a 
				path $Q\subseteq H(\P)$ with $V(Q)=V(P_A)\cup Z$ having the same end-$(k-1)$-tuples 
				as $P_A$. 			
		\end{enumerate}
\end{prop}

Our absorbers will be analogous to those in~\cite{Y} and we refer to~\cite{Y}*{Section 5.1} 
for further motivation. Here we will only recall that the absorbers
have two kinds of main components reflecting the following observations.

\begin{itemize}
	\item A complete $k$-partite subhypergraph $S$ of $H(\Psi)$ whose vertex classes
		$\{x_i, x_{i+k}, x_{i+2k}\}$ are of size $3$ (where $i\in [k]$) contains a spanning 
		path $P=x_1\ldots x_{3k}$. Moreover,~$S$ also contains the 
		path $P'=x_1\ldots x_kx_{2k+1}\ldots x_{3k}$, which has the same first and 
		last $(k-1)$-tuple as $P$. Thus if the absorbing path contains $P'$ as a subpath but 
		avoids the vertices $x_{k+1}, \ldots, x_{2k}$, then it can absorb these vertices 
		simultaneously (see Figure~\ref{fig:xy1}).
		However, not every $k$-element subset of $V(\Psi)$ is {\it absorbable}
		in this manner. 
	\item If the links of two vertices $a$ and $x$ intersect in a $(k-1)$-uniform path 
		$b_1\ldots b_{2k-2}$, then we can form two $k$-uniform paths in~$H(\Psi)$, namely
		$P_a=b_1\ldots b_{k-1}ab_k\ldots b_{2k-1}$ and $P_x=b_1\ldots b_{k-1}xb_k\ldots b_{2k-1}$
		(see Figure~\ref{fig:xy3}).
		Now if the absorbing path contains~$P_x$, then we can remove $x$ and insert $a$ instead.
		We call such a structure an {\it $(a, x)$-exchanger}.
\end{itemize} 

Now the plan for absorbing an arbitrary set $\{a_1, \ldots, a_k\}$ of $k$ vertices is that 
we will find an `absorbable' set $\{x_1, \ldots, x_k\}$ such that for every $i\in [k]$ there 
is an $(a_i, x_i)$-exchanger. The main difficulty in executing this strategy is that we need 
to pay a lot of attention to connectability issues, because ultimately we need to connect all 
parts of the absorbers we are about to construct to the rest of the Hamiltonian cycle we 
intend to exhibit. For this reason, the definition of absorbers reads as follows. 

\begin{dfn}\label{d:abs}
	Suppose that $k\ge 3$, $\al, \mu, \zeta >0$, that $\P$ is 
	a $k$-uniform $(\al, \mu)$-constellation, and that  $\seq a = (a_1, \dots, a_k)\in V(\P)^k$
	is a $k$-tuple consisting of distinct vertices. We say that 
	\[
		\seq{A}=(\seq u, \seq x, \seq w, \seq b_1, \dots, \seq b_k)\in V(\P)^{2k^2+k}
	\]
	with $\seq u = (u_1, \dots, u_k)$, $\seq x = (x_1, \dots, x_k)$, $\seq w = (w_1, \dots, w_k)$, 
	and $\seq b_i = (b_{i1}, \dots, b_{i(2k-2)})$ for $i\in [k]$ is an {\it $(\sa, \zeta)$-absorber} 
	in $\P$, if
	\begin{enumerate}[label=\alabel]
			\item\label{it:aa} all $2k^2+k$ vertices of $\seq{A}$ are distinct and different 
				from those in $\seq a$, 
			\item\label{it:ab} $\seq u\seq x\seq w$ and $\seq u\seq w$ are paths in $H(\Psi)$,
			\item\label{it:ac} $(u_1, \dots, u_{k-1})$ is $\zeta$-rightconnectable 
				and $(w_2, \dots, w_k)$ is $\zeta$-leftconnectable in $\P$,
			\item\label{it:ad} and for every $i\in [k]$ the $(2k-2)$-tuple $\seq b_i$ 
				is a path in $H(\Psi_{a_i})\cap H(\Psi_{x_i})$ whose first and last
				$(k-1)$-tuple is $\zeta$-connectable in $\Psi$.
	\end{enumerate}
\end{dfn}

\begin{figure}[h]
	\centering
	\begin{subfigure}[b]{0.45\textwidth}
		\centering
		\begin{tikzpicture}[scale = .7]

	\def\ra{3cm};
	
	\foreach \i in {1,...,5}{
		\coordinate (x\i) at (18+\i*72:\ra);
		\coordinate (u\i) at (18+\i*72:\ra+1cm);
		\coordinate (w\i) at (18+\i*72:\ra-1cm);
		\coordinate (n\i) at (5+\i*72:\ra+1.7cm);
		%	\fill (n\i) circle (3pt);
		\fill (x\i) circle (3pt);
		\fill (u\i) circle (3pt);
		\fill (w\i) circle (3pt);
		\foreach \i in {1,...,36}{
			\coordinate (v\i) at (90+\i*10:\ra-1cm);
			\coordinate (s\i) at (90+\i*10:\ra+1cm);
			\coordinate (t\i) at (90+\i*10:\ra);}
		
		\coordinate (y1) at (42:\ra+.1);
		\coordinate (y2) at (66:\ra-.6cm);
		\coordinate (y3) at (42:\ra +.6cm);
		\coordinate (y4) at (66:\ra+.2cm);
		\coordinate (y5) at (42:\ra-.4cm);
		\coordinate (y6) at (66:\ra-.8cm);

		\draw[ line width=2pt,fill opacity=1, rotate = 18+\i*72] (x\i) ellipse (2cm and 25pt);}

	\coordinate (z1) at (150:\ra-.15cm);
	\coordinate (z2) at (135:\ra-.44cm);
	\coordinate (z3) at (120:\ra-.7cm);
	\coordinate (z4) at (105:\ra-.9cm);

	\coordinate (g1) at (150:\ra+.85cm);
	\coordinate (g2) at (135:\ra+.56cm);
	\coordinate (g3) at (120:\ra+.3cm);
	\coordinate (g4) at (105:\ra+.1cm);

	\draw [yellow, line width = 10pt] plot [smooth]  coordinates {(s36) (s35) (s34) (s33) (s32) (s31) (s30) (s29) (s28) (s27)(s26) (s25) (s24) (s23) (s22) (s21) (s20) (s19) (s18) (s17) (s16) (s15) (s14) (s13) (s12) (s11) (s10) (s9) (s8) (s7) (g1) (g2)(g3)(g4)	(t36) (t35) (t34) (t33) (t32) (t31) (t30) (t29) (t28) (t27)(t26) (t25) (t24) (t23) (t22) (t21) (t20) (t19) (t18) (t17) (t16) (t15) (t14) (t13) (t12) (t11) (t10) (t9) (t8) (t7) (z1)(z2)(z3)(z4)
		(v36) (v35) (v34) (v33) (v32) (v31) (v30) (v29) (v28) (v27)(v26) (v25) (v24) (v23) (v22) (v21) (v20) (v19) (v18) (v17) (v16) (v15) (v14) (v13) (v12) (v11) (v10) (v9) (v8) (v7) 	};

	\coordinate (h1) at (150:\ra+.7cm);
	\coordinate (h2) at (135:\ra-.12cm);
	\coordinate (h3) at (120:\ra-.62cm);
	\coordinate (h4) at (105:\ra-.9cm);

	\draw [blue, line width = 4pt] plot [smooth]  coordinates {(s36) (s35) (s34) (s33) (s32) (s31) (s30) (s29) (s28) (s27)(s26) (s25) (s24) (s23) (s22) (s21) (s20) (s19) (s18) (s17) (s16) (s15) (s14) (s13) (s12) (s11) (s10) (s9) (s8) (s7) (h1) (h2)(h3)(h4)
		(v36) (v35) (v34) (v33) (v32) (v31) (v30) (v29) (v28) (v27)(v26) (v25) (v24) (v23) (v22) (v21) (v20) (v19) (v18) (v17) (v16) (v15) (v14) (v13) (v12) (v11) (v10) (v9) (v8) (v7) 	};
	
	\foreach \i in {1,...,5}{
		\fill (x\i) circle (5pt);
		\fill (u\i) circle (5pt);
		\fill (w\i) circle (5pt);}
	
	\node at (n1){\large $V_1$};
	\node at (n5){\large $V_2$};
	\node at ($(n3)+(-.3,0)$){\large $V_{k-1}$};
	\node at (n2){\large $V_k$};

		\node at (18+72:\ra+.35cm){\tiny $x_1$};
		\node at (18+5*72:\ra+.45cm){\tiny $x_2$};
		\node at (20+3*72:\ra+.4cm){\tiny $x_{k-1}$};
		\node at (18+2*72:\ra+.4cm){\tiny $x_k$};
		
		\node at (18+72:\ra+1.33cm){\tiny $u_1$};
		\node at (18+5*72:\ra+1.43cm){\tiny $u_2$};
		\node at (20+3*72:\ra+1.5cm){\tiny $u_{k-1}$};
		\node at (18+2*72:\ra+1.43cm){\tiny $u_k$};
		
		\node at (18+72:\ra-.7cm){\tiny $w_1$};
		\node at (18+5*72:\ra-.6cm){\tiny $w_2$};
		\node at (22+3*72:\ra-.7cm){\tiny $w_{k-1}$};
		\node at ($(18+2*72:\ra-.65cm)+(.35,.25)$){\tiny $w_k$};

		\end{tikzpicture}
		\caption{The $K^{(k)}_k(3)$ with two paths}\label{fig:xy1}
	\end{subfigure}
	\hfill    
	\begin{subfigure}[b]{0.45\textwidth}
		\centering
		\begin{tikzpicture}[scale = .9]

		\phantom{\fill (0,-9) circle (2pt);}

			\def\s{.8};
			
			\def\an{5};
			\def\s{270-9*\an};
			\def\ra{6cm};
			
			\foreach \i in {1,...,8}{
				\coordinate (b\i) at (\s+2*\i*\an:\ra);
			}
			
			\coordinate (a) at ($(b4)!.5!(b5)+(0,2.7)$);
			\coordinate (x) at ($(b4)!.5!(b5)-(0,1.8)$);
			
			\def\w{9pt};

		\pedge{(a)}{(b4)}{(b3)}{(b2)}{(b1)}{\w}{1.5pt}{violet}{violet,opacity=0.15};
		\pedge{(a)}{(b5)}{(b4)}{(b3)}{(b2)}{\w}{1.5pt}{violet}{violet,opacity=0.15};
		\pedge{(a)}{(b6)}{(b5)}{(b4)}{(b3)}{\w}{1.5pt}{violet}{violet,opacity=0.15};
		\pedge{(a)}{(b7)}{(b6)}{(b5)}{(b4)}{\w}{1.5pt}{violet}{violet,opacity=0.15};
		\pedge{(a)}{(b8)}{(b7)}{(b6)}{(b5)}{\w}{1.5pt}{violet}{violet,opacity=0.15};
		
		\pedge{(x)}{(b1)}{(b2)}{(b3)}{(b4)}{\w}{1.5pt}{violet!100!black}{violet,opacity=0.15};
		\pedge{(x)}{(b2)}{(b3)}{(b4)}{(b5)}{\w}{1.5pt}{violet!100!black}{violet,opacity=0.15};
		\pedge{(x)}{(b3)}{(b4)}{(b5)}{(b6)}{\w}{1.5pt}{violet!100!black}{violet,opacity=0.15};
		\pedge{(x)}{(b4)}{(b5)}{(b6)}{(b7)}{\w}{1.5pt}{violet!100!black}{violet,opacity=0.15};
		\pedge{(x)}{(b5)}{(b6)}{(b7)}{(b8)}{\w}{1.5pt}{violet!100!black}{violet,opacity=0.15};
		
		\def\ww{5pt};
		
		\redge{(b4)}{(b3)}{(b2)}{(b1)}{\ww}{1.5pt}{yellow!100!black}{yellow,opacity=0.25};
		\redge{(b5)}{(b4)}{(b3)}{(b2)}{\ww}{1.5pt}{yellow!100!black}{yellow,opacity=0.25};
		\redge{(b6)}{(b5)}{(b4)}{(b3)}{\ww}{1.5pt}{yellow!100!black}{yellow,opacity=0.25};
		\redge{(b7)}{(b6)}{(b5)}{(b4)}{\ww}{1.5pt}{yellow!100!black}{yellow,opacity=0.25};
		\redge{(b8)}{(b7)}{(b6)}{(b5)}{\ww}{1.5pt}{yellow!100!black}{yellow,opacity=0.25};
		
		\foreach \i in {1,...,8}{
			\fill (b\i) circle (2pt);
		}
		\fill (a) circle (2pt);
		\fill (x) circle (2pt);
		
		\node at ($(a)	+(0,.55)$) {$a$};
		\node at ($(x)	-(0,.55)$) {$x$};

	\end{tikzpicture}
	\caption{A $5$-uniform $(a,x)$-exchanger}\label{fig:xy3}
\end{subfigure}    
\caption{The building blocks of an absorber.}	\label{fig:abspart2}
\vspace{-1em}
\end{figure}
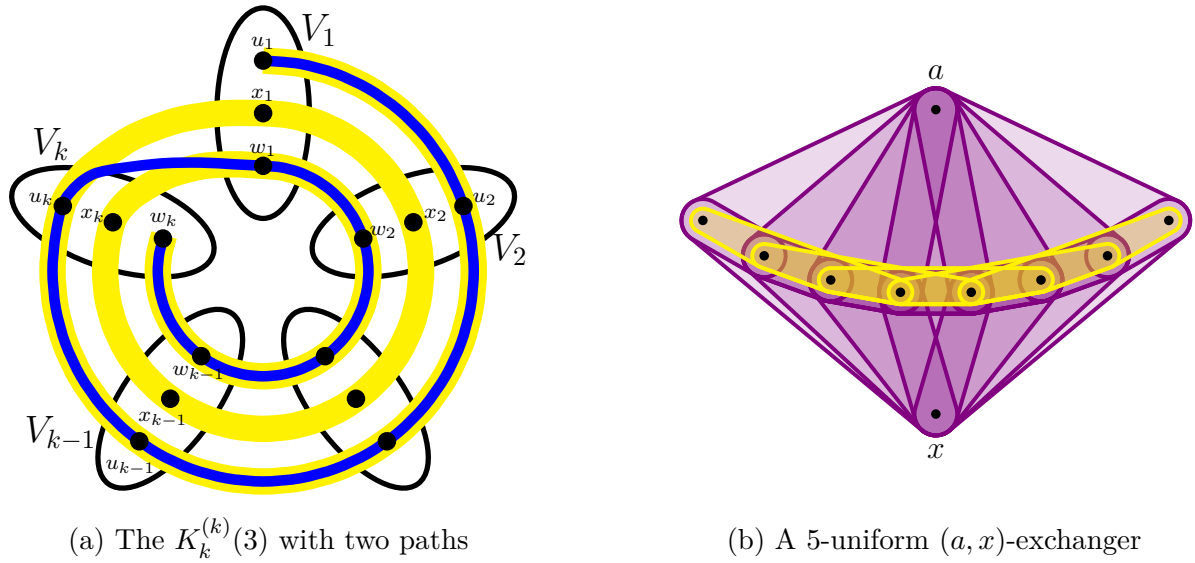

\begin{figure}[t]
	\begin{tikzpicture}[scale=1]
	
	\def \dd{2.1};
	\def \md{.45};
	\def \co{yellow!90!black};

	\coordinate (a1) at (-3*\dd, 0);
	\coordinate (a2) at (-\dd,0);
	\coordinate (a3) at (\dd,0);
	\coordinate (a4) at (3*\dd,0);
	
	\foreach \i in {1,2,3,4}{
		\coordinate (x\i) at ($(a\i)-(0,1.5)$);
		\foreach \j in {1,...,7}{
			\coordinate (b\i\j) at ($(x\i)+(\j*\md,0)-(4*\md,0)$);
		}}
	
	\foreach \i in {1,2}{
		\node at (b\i1) at ($(b\i1)+(0,.5)$){\tiny $ b_{\i1}$};
		\node at (b\i2) at ($(b\i2)+(.2,.5)$){\tiny $\dots$};
		\node at (b\i6) at ($(b\i6)+(-.4,.5)$){\tiny $\dots$};
		\node at (b\i7) at ($(b\i7)+(0,.5)$){\tiny $b_{\i (2k-2)}$};
	}
	
	\node at (b41) at ($(b41)+(0,.5)$){\tiny $ b_{k1}$};
	\node at (b42) at ($(b42)+(.2,.5)$){\tiny $\dots$};
	\node at (b46) at ($(b46)+(-.4,.5)$){\tiny $\dots$};
	\node at (b47) at ($(b47)+(0,.5)$){\tiny $b_{k (2k-2)}$};
	
	\coordinate (y1) at (-5,-4);
	\coordinate (y2) at (-4,-4);
	\coordinate (y3) at (-3,-4);
	\coordinate (y4) at (-2,-4);
	\coordinate (z1) at (2,-4);
	\coordinate (z2) at (3,-4);
	\coordinate (z3) at (4,-4);
	\coordinate (z4) at (5,-4);
	
	\coordinate (v1) at ($(y4)!.2!(z1)$);
	\coordinate (v2) at ($(y4)!.4!(z1)$);
	\coordinate (v3) at ($(y4)!.6!(z1)$);
	\coordinate (v4) at ($(y4)!.8!(z1)$);
	
	\foreach \i in {1,2,4}{
		\redge{(b\i1)}{(b\i2)}{(b\i3)}{(x\i)}{5.5pt}{1.5pt}{\co}{yellow,opacity=0.25}
		\redge{(b\i2)}{(b\i3)}{(x\i)}{(b\i5)}{5.5pt}{1.5pt}{\co}{yellow,opacity=0.25}
		\redge{(b\i3)}{(x\i)}{(b\i5)}{(b\i6)}{5.5pt}{1.5pt}{\co}{yellow,opacity=0.25}
		\redge{(x\i)}{(b\i5)}{(b\i6)}{(b\i7)}{5.5pt}{1.5pt}{\co}{yellow,opacity=0.25}
		}
	
	\redge{(y1)}{(y2)}{(y3)}{(y4)}{5.5pt}{1.5pt}{\co}{yellow,opacity=0.25}
	\redge{(y2)}{(y3)}{(y4)}{(z1)}{5.5pt}{1.5pt}{\co}{yellow,opacity=0.25}
	\redge{(y3)}{(y4)}{(z1)}{(z2)}{5.5pt}{1.5pt}{\co}{yellow,opacity=0.25}
	\redge{(y4)}{(z1)}{(z2)}{(z3)}{5.5pt}{1.5pt}{\co}{yellow,opacity=0.25}
	\redge{(z1)}{(z2)}{(z3)}{(z4)}{5.5pt}{1.5pt}{\co}{yellow,opacity=0.25}

	\begin{pgfonlayer}{front}

		\foreach \i in {1,2,4}{
			\fill (y\i) circle (2pt);
			\fill (a\i) circle (3pt);
			\fill (x\i) circle (3pt);
			\fill (z\i) circle (2pt);
			\foreach \j in {1,...,7}{
				\fill (b\i\j) circle (2pt);}
			\draw[black!50!white, line width=1.5pt, line cap=round, shorten <=3.8pt,shorten >=6.5pt,->] (a\i) -- (x\i);
		
	}
	
		\fill (y3) circle (2pt);
		\fill (z3) circle (2pt);
	
		\node at ($(x1) -(.3,.5)$) {$x_1$};
		\node at ($(x2) -(.3,.5)$) {$x_2$};
		\node at ($(x4) -(-.3,.5)$) {$x_k$};
		
		\node at ($(z1) -(0,.5)$) {$w_{1}$};
		\node at ($(z2) -(0,.5)$) {$w_{2}$};
		\node at ($(z4) -(0,.5)$) {$w_{k}$};
		\node at ($(z3) -(0,.5)$) {\Huge $\dots$};
		
			\node at ($(a1)+(0,.3)$) {$a_{1}$};
			\node at ($(y1) -(0,.5)$) {$u_{1}$};
				\node at ($(a2)+(0,.3)$) {$a_{2}$};
				\node at ($(y2) -(0,.5)$) {$u_{2}$};
					\node at ($(a4)+(0,.3)$) {$a_{k}$};
					\node at ($(y4) -(0,.5)$) {$u_{k}$};
					\node at ($(x3) $) {\Huge $\dots$};
					\node at ($(v3) +(0,.5)$) {\Huge $\dots$};
					\node at ($(y3) -(0,.5)$) {\Huge $\dots$};

	\end{pgfonlayer}

	\draw[black!50!white, line width=1.5pt, line cap=round, shorten <=6.8pt,shorten >=6.5pt,->] (x1) -- +(0,-1.5)-- ($(v1)+(0,1)$) --(v1);
	\draw[black!50!white, line width=1.5pt, line cap=round, shorten <=6.8pt,shorten >=6.5pt,->] (x2) -- +(0,-.95)-- ($(v2)+(0,1.55)$) --(v2);
	\draw[black!50!white, line width=1.5pt, line cap=round, shorten <=6.8pt,shorten >=6.5pt,->] (x4) -- +(0,-1.5)-- ($(v4)+(0,1)$) --(v4);

	\begin{pgfonlayer}{background}
	\draw[ultra thick, decorate, decoration={snake,segment length=10,amplitude=3}, yellow!95!red]
			(b17)--(b21);
	\draw[ultra thick, decorate, decoration={snake,segment length=10,amplitude=3}, yellow!95!red]
			(b27)--(b33);
	\draw[ultra thick, decorate, decoration={snake,segment length=10,amplitude=3}, yellow!95!red]
			(b35)--(b41);
	
	\draw[ultra thick, decorate, decoration={snake,segment length=10,amplitude=3}, yellow!95!red]
			(b47) to[out=270,in=0] (z4);
	\end{pgfonlayer}

	\end{tikzpicture}

\bigskip 

	\begin{tikzpicture}[scale=1]
	
	\def \dd{2.1};
	\def \md{.45};
	\def \co{yellow!90!black};

	\coordinate (a1) at (-3*\dd, 0);
	\coordinate (a2) at (-\dd,0);
	\coordinate (a3) at (\dd,0);
	\coordinate (a4) at (3*\dd,0);
	
	\foreach \i in {1,2,3,4}{
		\coordinate (x\i) at ($(a\i)-(0,1.5)$);
		\foreach \j in {1,...,7}{
			\coordinate (b\i\j) at ($(x\i)+(\j*\md,0)-(4*\md,0)$);
		}}

		\foreach \i in {1,2}{
			\node at (b\i1) at ($(b\i1)+(0,.5)$){\tiny $ b_{\i1}$};
			\node at (b\i2) at ($(b\i2)+(-.1,.3)$){\tiny $\dots$};
			\node at (b\i3) at ($(b\i3)+(0,.5)$){\tiny $ b_{\i (k-1)}$};
			\node at (b\i5) at ($(b\i5)+(0,.5)$){\tiny $ b_{\i k}$};
			\node at (b\i6) at ($(b\i6)+(-.1,.3)$){\tiny $\dots$};
			\node at (b\i7) at ($(b\i7)+(0,.5)$){\tiny $b_{\i (2k-2)}$};
		}
		
		\node at (b41) at ($(b41)+(0,.5)$){\tiny $ b_{k1}$};
		\node at (b42) at ($(b42)+(-.1,.3)$){\tiny $\dots$};
		\node at (b43) at ($(b43)+(0,.5)$){\tiny $ b_{k (k-1)}$};
		\node at (b45) at ($(b45)+(0,.5)$){\tiny $ b_{k k}$};
		\node at (b46) at ($(b46)+(-.1,.3)$){\tiny $\dots$};
		\node at (b47) at ($(b47)+(0,.5)$){\tiny $b_{k (2k-2)}$};

	\coordinate (y1) at (-5,-4);
	\coordinate (y2) at (-4,-4);
	\coordinate (y3) at (-3,-4);
	\coordinate (y4) at (-2,-4);
	\coordinate (z1) at (2,-4);
	\coordinate (z2) at (3,-4);
	\coordinate (z3) at (4,-4);
	
	\coordinate (v1) at ($(y4)!.2!(z1)+(0,.2)$);
	\coordinate (v2) at ($(y4)!.4!(z1)+(0,.4)$);
	\coordinate (v3) at ($(y4)!.6!(z1)+(0,.4)$);
	\coordinate (v4) at ($(y4)!.8!(z1)+(0,.2)$);
	
	\foreach \i in {1,2,4}{
		\redge{(b\i1)}{(b\i2)}{(b\i3)}{(x\i)}{5.5pt}{1.5pt}{\co}{yellow,opacity=0.25}
		\redge{(b\i2)}{(b\i3)}{(x\i)}{(b\i5)}{5.5pt}{1.5pt}{\co}{yellow,opacity=0.25}
		\redge{(b\i3)}{(x\i)}{(b\i5)}{(b\i6)}{5.5pt}{1.5pt}{\co}{yellow,opacity=0.25}
		\redge{(x\i)}{(b\i5)}{(b\i6)}{(b\i7)}{5.5pt}{1.5pt}{\co}{yellow,opacity=0.25}
		}
	
	\redge{(y1)}{(y2)}{(y3)}{(y4)}{5.5pt}{1.5pt}{\co}{yellow,opacity=0.25}
	\redge{(y4)}{(y3)}{(y2)}{(v1)}{5.5pt}{1.5pt}{\co}{yellow,opacity=0.25}
	\redge{(y4)}{(y3)}{(v2)}{(v1)}{5.5pt}{1.5pt}{\co}{yellow,opacity=0.25}
	\redge{(y4)}{(v1)}{(v2)}{(v3)}{5.5pt}{1.5pt}{\co}{yellow,opacity=0.25}
	\redge{(v1)}{(v2)}{(v3)}{(v4)}{5.5pt}{1.5pt}{\co}{yellow,opacity=0.25}
	\redge{(v2)}{(v3)}{(v4)}{(z1)}{5.5pt}{1.5pt}{\co}{yellow,opacity=0.25}
	\redge{(z2)}{(z1)}{(v4)}{(v3)}{5.5pt}{1.5pt}{\co}{yellow,opacity=0.25}
	\redge{(z3)}{(z2)}{(z1)}{(v4)}{5.5pt}{1.5pt}{\co}{yellow,opacity=0.25}
	\redge{(z4)}{(z3)}{(z2)}{(z1)}{5.5pt}{1.5pt}{\co}{yellow,opacity=0.25}

	\begin{pgfonlayer}{front}
	
		\foreach \i in {1,2,3,4} {
		\fill (z\i) circle (2pt);
		}
	
		\foreach \i in {1,2,4}{
			\fill (y\i) circle (2pt);
			\fill (v\i) circle (3pt);
			\fill (x\i) circle (3pt);
			\foreach \j in {1,...,7}{
				\fill (b\i\j) circle (2pt);}
	}
		\fill (v3) circle (3pt);
			\fill (y3) circle (2pt);
	
	\node at ($(y1) -(0,.5)$) {$u_{1}$};
	\node at ($(x1) -(0,.5)$) {$a_{1}$};
	\node at ($(y2) -(0,.5)$) {$u_{2}$};
	\node at ($(x2) -(0,.5)$) {$a_{2}$};
	\node at ($(y4) -(0,.5)$) {$u_{k}$};
	\node at ($(x4) -(0,.5)$) {$a_{k}$};
	\node at ($(y3) -(0,.5)$) {\Huge $\dots$};
	\node at ($(x3) $) {\Huge $\dots$};

	\node at ($(v1) +(0,.5)$) {$x_{1}$};
	\node at ($(v4) +(0,.5)$) {$x_{k}$};
	\node at ($(v2) +(0,.4)$) {$x_{2}$};
	\node at ($(v3) +(0,.4)$) {\Huge $\dots$};

\node at ($(z1) -(0,.5)$) {$w_{1}$};
\node at ($(z2) -(0,.5)$) {$w_{2}$};
\node at ($(z3) -(0,.5)$) {\Huge $\dots$};
\node at ($(z4) -(0,.5)$) {$w_{k}$};

	\end{pgfonlayer}

	\begin{pgfonlayer}{background}
	
	\draw[ultra thick, decorate, decoration={snake,segment length=10,amplitude=3}, yellow!95!red]
			(b17)--(b21);
	\draw[ultra thick, decorate, decoration={snake,segment length=10,amplitude=3}, yellow!95!red]
			(b27)--(b33);
	\draw[ultra thick, decorate, decoration={snake,segment length=10,amplitude=3}, yellow!95!red]
			(b35)--(b41);
	
	\draw[ultra thick, decorate, decoration={snake,segment length=10,amplitude=3}, yellow!95!red]
			(b47) to[out=270,in=0] (z4);

	\end{pgfonlayer}

				\end{tikzpicture}
				\caption{Absorber for $(a_1,\dots, a_k)$ before and after absorption.}
				\label{fig:abs}
			\end{figure}
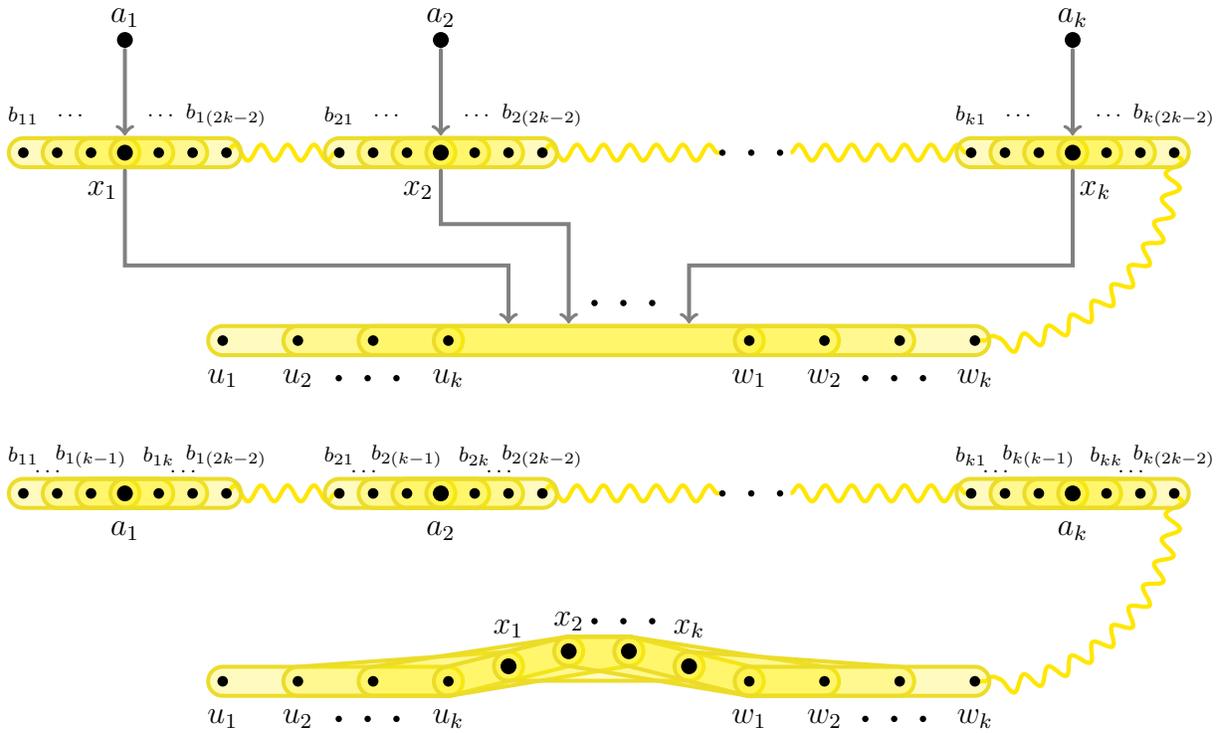

We conclude this subsection with an explicit description how these absorbers are going to 
be utilised (see Figure~\ref{fig:abs}). Suppose to this end that for some $k$-tuple 
$\sa = (a_1, \dots, a_k)$ consisting of $k$ distinct vertices and some $(\sa, \zeta)$-absorber 
$(\seq u, \seq x, \seq w, \seq b_1, \dots, \seq b_k)$ it turns out that the paths
\begin{equation}\label{eq:1818}
	\su\sw    
	\qand
	 b_{i1}\dots b_{i(k-1)} x_ib_{ik}\dots b_{i(2k-2)}
	\quad\textrm{for }i\in [k]
\end{equation}
end up being subpaths of the absorbing path $P_A$ we are about to construct, while $a_1, \ldots, a_k$
are not in $V(P_A)$. We may then replace for each $i\in [k]$ the path
\[
	b_{i1}\dots b_{i(k-1)} x_ib_{ik}\dots b_{i(2k-2)}
		\quad\textrm{ by the path }\quad
	b_{i1}\dots b_{i(k-1)} a_ib_{ik}\dots b_{i(2k-2)}\,,
\]
and then	
\[
	\su\sw			
	\quad\textrm{ by  }\quad
	\su\sx\sw\,.
\]
In this manner we transform $P_A$ into a new path $Q$ with $V(Q)=V(P_A)\cup \{a_1, \ldots, a_k\}$
having the same first and last $(k-1)$-tuple as $P_A$. We say in this 
situation that~$Q$ arises from $P_A$ by {\it absorbing} $\{a_1, \ldots, a_k\}$. 
The $k+1$ paths enumerated in~\eqref{eq:1818} are called the {\it pre-absorption paths} of the 
absorber $(\seq u, \seq x, \seq w, \seq b_1, \dots, \seq b_k)$. So there is one pre-absorption 
path with $2k$ vertices, namely $\su\sw$, and there are $k$ pre-absorption paths with $2k-1$
vertices having a vertex $x_i$ in the middle. 

\subsection{Construction of the building blocks}
	We commence with the first part $(\su, \sx, \sw)$ of our absorbers consisting of $3k$ vertices. 
	As we have already indicated, we shall find $(3k)$-tuples 
	satisfying clause~\ref{it:ab} of Definition~\ref{d:abs} by looking for complete $k$-partite 
	subhypergraphs of~$H(\Psi)$ whose vertex classes are of size three.  
	
	Let us recall for this purpose that by a classic result of Erd\H{o}s~\cite{E} the Tur\'an 
	density of every $k$-partite $k$-uniform hypergraph vanishes. This means that, given a 
	$k$-partite $k$-uniform hypergraph $F$ and a constant $\eps>0$, every sufficiently large 
	$k$-uniform hypergraph $H$ satisfying $|E(H)|\ge \eps |V(H)|^k$ contains a copy of $F$. Due to the
	so-called `supersaturation' phenomenon later discovered by Erd\H{o}s and Simonovits~\cite{ES}, 
	the same assumption actually implies that $H$ contains $\Omega\bigl(|V(H)|^{|V(F)|}\bigr)$
	copies of $F$. For later reference, we record this fact as follows.
	\begin{lemma}\label{thm:super}
		Given a $k$-partite $k$-uniform hypergraph $F$ and $\eps>0$, there are a constant $\xi>0$
		and a natural number $n_0$ such that every $k$-uniform hypergraph $H$ on $n\ge n_0$ vertices 
		with at least $\eps n^k$ edges contains at least $\xi n^{|V(F)|}$ copies of $F$. \qed
	\end{lemma}

	We shall now apply this result to $F=K^{(k)}_k(3)$, the complete $k$-partite hypergraph with 
	vertex classes of size $3$, and to an auxiliary hypergraph whose edges are derived from bridges.
	This will establish the following statement, whose conditions~\ref{it:e1} and~\ref{it:e2}
	coincide with~\ref{it:ab} and~\ref{it:ac} from Definition~\ref{d:abs}.

	\begin{lemma}\label{l:elf}
		For every $k\ge 2$ there exists $\xi=\xi(k)>0$ such that for every $\al>0$ there is 
		an integer $n_0$ with the following property. 
		
		For every $k$-uniform $(\al, \frac \al 9)$-constellation $\P$ on $n\ge n_0$ vertices 
		the number of $(3k)$-tuples $(\su, \sx, \sw)\in V(\P)^k\times V(\P)^k\times V(\P)^k$ 
		such that writing $\su=(u_1, \dots, u_k)$, $\sx=( x_1, \dots, x_k)$, 
		and $\sw=(w_1, \dots, w_k)$
		\begin{enumerate}[label=\rmlabel]
			\item\label{it:e1} both $\su\sx\sw$ and $\su\sw$ are $k$-uniform paths in $\P$,
			\item\label{it:e2} $(u_1, \dots, u_{k-1})$ is $\frac 1{9k}$-rightconnectable 
				and $(w_2, \dots, w_k)$ is $\frac 1{9k}$-leftconnectable in $\P$
		\end{enumerate}
		is at least $\xi n^{3k}$.
	\end{lemma}
	
	\begin{proof}
		Throughout the argument we assume that $\xi\ll k^{-1}$ is sufficiently small and 
		that $n_0\gg \alpha^{-1}, \xi^{-1}$ is sufficiently large. Let $\P$ be a $k$-uniform 
		$(\al, \frac\alpha9)$-constellation on $n\ge n_0$ vertices. Construct an 
		auxiliary $k$-partite $k$-uniform hypergraph $\ccB=(V_1\dcup\dots\dcup V_k, E_{\ccB})$ 
		whose vertex classes are $k$ disjoint copies of $V(\P)$ and whose 
		edges $\{v_1,\dots, v_k\}\in E_{\ccB}$ with $v_i\in V_i$ for $i\in [k]$ correspond to 
		the $\frac 1{9k}$-bridges $(v_1, \dots, v_k)$ of $\P$. Corollary~\ref{c:219} tells us that 
		\[
			|E_{\ccB}| \ge \frac 19 n^k = \frac {1}{9k^k}|V(\ccB)|^k\,.
		\]
		So Lemma~\ref{thm:super} applied to $\ccB$ and $F=K^{(k)}_k(3)$  
		leads to $\Omega(n^{3k})$ copies of $K^{(k)}_k(3)$ in $\ccB$,
		where the implied constant only depends on $k$. In other words, for some 
		constant $\xi=\xi(k)$ depending only on $k$ there are at least $\xi n^{3k}$ 
		tuples  $(\su, \sx, \sw)\in V(\P)^k\times V(\P)^k\times V(\P)^k$ 
		such that, writing $\su=(u_1, \dots, u_k)$, $\sx=( x_1, \dots, x_k)$, 
		and $\sw=(w_1, \dots, w_k)$, we have a copy of $K^{(k)}_k(3)$ in $\ccB$ with 
		$u_i, x_i, w_i\in V_i$
		for all $i\in [k]$. Clearly, these $(3k)$-tuples satisfy the demand~\ref{it:e1}
		of the lemma and, since $\su$ and $\sw$ are $\frac 1{9k}$-bridges, they have 
		property~\ref{it:e2} as well (cf. Definition~\ref{d:bridge}). 
\end{proof}

Armed with this result and with Corollary \ref{c:smallabs} we can now prove that if
$\zeta, \mu\ll \alpha, k^{-1}$, then for every $k$-tuple $\sa$ of distinct vertices
from a sufficiently large $(\alpha, \mu)$-constellation the number of $(\sa,\zeta)$-absorbers  
is at least $\Omega(n^{2k^2+k})$.

\begin{lemma}\label{l:menyabs}
	 For every $k\ge 3$ and $\al > 0$ there exist constants $\zeta = \zeta(\al,k)$ and 
	 $\xi = \xi(\al,k)$ as well as an integer $n_0$ with the following property. 
	 
	 If $\Psi$ is a $k$-uniform $(\al, \mu)$-constellation on $n\ge n_0$ vertices, 
	 where $\mu = \frac{1}{10k}\big(\frac{\al}{2}\big)^{2^{k-3}+1}$, and $\sa\in V(\P)^k$ 
	 is an arbitrary $k$-tuple of distinct vertices, then the number of $(\sa,\zeta)$-absorbers 
	 in~$\P$ is at least $\xi n^{2k^2+k}$.
\end{lemma}
\begin{proof}
	Starting with the constant $\xi''=\xi''(k)>0$ provided by Lemma~\ref{l:elf} we set 
	\begin{equation}\label{eq:zx}
		\xi'=\frac{\mu^k}{2},
			\quad
		\zeta =\frac{\xi''\mu}{7k},
			\qand 
		\xi = \frac 14 (\xi')^k\xi''
	\end{equation}
	and we suppose that $n_0$ is sufficiently large. 
	
	In order to show that $\zeta$ and $\xi$ have the 
	desired property, we consider a $k$-uniform $(\al,\mu)$-constellation $\P$ on $n\ge n_0$ 
	vertices as well as a $k$-tuple $\sa = (a_1, \dots, a_k)\in V(\P)^k$ consisting of distinct 
	vertices. The set $X\subseteq V(\P)$ delivered by Corollary~\ref{c:smallabs} (with the same 
	meaning of~$\Psi$,~$\alpha$,~$\mu$, and~$\zeta$ as here) satisfies		
	\begin{equation}\label{eq:x}
		|X|\le \frac \zeta\mu n 
			\overset{\eqref{eq:zx}}{=}
			\frac{\xi''}{7k}n\,.
	\end{equation}

	By $\mu \le \frac \al 9$,  $\zeta \le \frac 1{9k}$, and monotonicity, Lemma~\ref{l:elf} 
	yields at least $\xi'' n^{3k}$ paths $(\su, \sx, \sw)$ in $V(\P)^{3k}$ with the 
	properties~\ref{it:e1} and~\ref{it:e2} of that lemma. Since the number of these paths 
	sharing a vertex with $X\cup\{a_1, \dots, a_k\}$ can be bounded from above by 
	\[
		3k(|X|+k)n^{3k-1}
		\overset{\eqref{eq:x}}{\le}
		3k \frac{\xi''}{7k} n^{3k}+ 3k^2n^{3k-1}
		<
		\frac {\xi''}2 n^{3k}\,,
	\]
	there are at least $\frac {\xi''}2 n^{3k}$ such paths avoiding both $X$ and $\sa$.
	Now it suffices to establish that each of them participates in at 
	least $\frac 12(\xi')^kn^{2k^2-2k}$ absorbers. 
	
	For the rest of the proof we fix some such path $(\su, \sx, \sw)\in V(\P)^{3k}$
	and, as usual, we write $\sx = (x_1, \dots, x_k)$. Now we apply Corollary~\ref{c:smallabs} for 
	every $i\in [k]$ to the vertices $a_i$ and $x_i$, thus obtaining $\xi' n^{2k-2}$ 
	paths $\sb_i=(b_{i1}, \dots, b_{i(2k-2)})\in V(\P)^{2k-2}$ 
	in $H(\Psi_{a_i})\cap H(\Psi_{x_i})$ whose first and last $(k-1)$-tuples are $\zeta$-connectable 
	in $\Psi$. Altogether, this yields $(\xi')^k n^{2k^2-2k}$ 
	possibilities for $(\seq{b}_1, \ldots, \seq{b}_k)$ and for most of 
	them $(\su, \sx, \sw, \seq{b}_1, \ldots, \seq{b}_k)$ is an $(\sa, \zeta)$-absorber.
	The only exceptions occur when some of these $2k^2+k$ vertices coincide, but this can 
	happen in at most $(2k^2+k)(2k^2-2k) n ^ {(2k-2)k-1}<\frac12 (\xi')^k n^{2k^2-2k}$ ways.
	Thus $(\su, \sx, \sw)$ is indeed extendable in at least $\frac12 (\xi')^k n^{2k^2-2k}$
	distinct ways to an $(\sa, \zeta)$-absorber.
\end{proof}

\subsection{Construction of the absorbing path}\label{subsec:1838}
After these preparations the Absorbing Path Lemma can be shown in a rather standard fashion.
The argument starts by observing that a random selection of $(2k^2+k)$-tuples 
contains, with high probability, for every $k$-tuple~$\seq{a}$ a positive proportion of
$(\seq{a}, \zeta)$-absorbers. Moreover, if we generate $\Theta(n)$ such random tuples with a small 
implied constant, then most of 
them will be disjoint to all others and it remains to connect the paths they consist of by means 
of the Connecting Lemma. 
\def\tt{t}

\begin{proof}[Proof of Proposition~\ref{prop:apl}]
	
	Given to us are $k\ge 3$, $\al, \beta >0$, an odd integer $\l \ge 3$, and 
	$\mu = \frac{1}{10k}\big(\frac{\al}{2}\big)^{2^{k-3}+1}$. Let $\zeta = \zeta(\al,k)>0$ 
	and ${\xi=\xi(\al,k)>0}$ be the constants supplied by Lemma~\ref{l:menyabs}, let 
	$\theta=\theta(k, \al, \beta, \l, \zeta)$ be provided by Proposition \ref{prop:clk},
	define an auxiliary constant by 
	\begin{equation}\label{eq:gamma}
		\gamma =\min\Big\{\frac{\xi}{48k^2M^2}, \frac{\theta}{8kM^2}\Big\},
		\quad\textrm{where}\quad
		M= 4^{k-2}k\ell \ge 12k\,,
	\end{equation}
	and finally set 
	\[
		\thetas=4kM\gamma\,.
	\]
	We contend that $\zeta$ and $\thetas$ have the desired properties.
	
	To verify this we consider a $k$-uniform $(\al, \beta, \l, \mu)$-constellation $\P$ 
	on $n$ vertices, where $n$ is sufficiently large, as well as an arbitrary 
	subset $\cR \subseteq V(\P)$ whose size is at most $\thetas^2 n$. Let 
	\[
			\tt = 2k^2+k< 3k^2
	\]
	be the length of our absorbers. Since the desired absorbing path needs to be disjoint 
	to~$\cR$, only the absorbers avoiding~$\cR$ are relevant in the sequel. For every 
	$k$-tuple $\sa\in V(\P)^k$ consisting of distinct vertices we denote the collection of 
	appropriate absorbers by 
	\[
		\ccA(\sa)=\bigl\{\seq{A}\in (V(\P)\setminus \cR)^t\colon 
			\seq{A} \text { is an $(\sa, \zeta)$-absorber}\bigr\}\,.
	\]
	Lemma~\ref{l:menyabs} tells us that the total number of $(\sa, \zeta)$-absorbers is at 
	least $\xi n^{\tt}$ and by subtracting those which meet $\cR$ we obtain 
	\begin{equation}\label{eq:asa}
			|\ccA(\sa)|
			\ge 
			\xi n^{\tt} - \tt|\cR|n^{\tt-1} 
			\ge 
			(\xi-t\thetas^2)n^t
			\ge
			\frac \xi 2n^{\tt}\,.
	\end{equation}
	Let 
	\[
		\ccA
		=
		\bigcup\bigl\{\ccA(\sa)\colon \sa\in V(\P)^k \text{ consists of $k$ distinct vertices}\bigr\}
		\subseteq 
		\bigl(V(\P)\setminus \cR\bigr)^{\tt}
	\]
	be the set of all relevant absorbers. The probabilistic argument we have been alluding to
	earlier leads to the following result. 
	
	\begin{claim}\label{clm:1905}
		There is a set $\ccB\subseteq \ccA$ of mutually disjoint absorbers of size $|\ccB|\le 2\gamma n$
		satisfying $|\ccA(\sa)\cap \ccB|\ge \thetas^2n$ for every $k$-tuple $\sa\in V(\P)^k$ 
		consisting of distinct vertices. 
	\end{claim}
		
	\begin{proof}
		Let $\ccA_p\subseteq \ccA$ be a random subset including every absorber in~$\ccA$ independently 
		with probability $p=\gamma n^{1-t}$. As $|\ccA_p|$ is binomially distributed with 
		expectation $p|\ccA| \le pn^{\tt}=\gamma n$, Markov's inequality yields
		\begin{equation}\label{eq:m1}
			\PP\bigl(|\ccA_p|\ge 2\gamma n\bigr)
			\le 
			\PP\bigl(|\ccA_p|\ge 2p |\ccA|\bigr)
			\le
			\frac 12\,.
		\end{equation}

		Next we observe that the set 
		\[
			\bigl\{\{\seq{A}, \seq{A'}\}\in \ccA^{(2)}\colon \seq{A} \tand \seq{A'} 
				\textrm{ share a vertex}\bigr\}
		\]
		of {\it overlapping pairs of absorbers} has at most the cardinality $t^2n^{2t-1}$.
		So the expected size of its intersection with $\ccA_p^{(2)}$ is at most 
		$p^2t^2n^{2t-1}=\gamma^2t^2n$.  
		Since 
		\[	
			\gamma t\le 3k^2\gamma \le \tfrac14\theta_\star\,,
		\]
		a further application of Markov's inequality reveals
		\begin{equation}\label{eq:m2}
			\PP\Bigl(\big|\bigl\{\{\seq{A},\seq{A'}\}\in \ccA^{(2)}_p\colon\,\seq{A}\tand \seq{A'} \textrm{ share a vertex}\bigr\}\big| 
				\ge \tfrac14 \theta_\star^2 n\Bigr)\le
			 \frac 14\,.
		\end{equation}

		Finally, for every $k$-tuple $\sa\in V(\P)^k$ of distinct vertices the random 
		variable $|\ccA_p\cap \ccA(\sa)|$ is binomially distributed with expectation $p|\ccA(\sa)|$.
		By~\eqref{eq:asa} we know that 
		\[
			p|\ccA(\sa)|
			\ge 
			\frac12\gamma\xi n
			\ge
			24 k^2M^2\gamma^2n
			= 
			\tfrac32\thetas^2n
		\]
		and, therefore, Chernoff's inequality yields
		\[
			\PP\big(|\ccA_p\cap \ccA(\sa)|\le \tfrac54 \thetas^2n \big)
			\le
			e^{-\Omega(n)}
			<
			\frac 1{4n^k}\,.
		\]		
		As there are at most $n^k$ possibilities for $\sa$, the union bound leads to 
		\begin{equation}\label{eq:c1}
			\PP\big(|\ccA_p\cap \ccA(\sa)|\le \tfrac54 \thetas^2n \text{ holds for some $\sa$}\big)
			<
			\frac 14\,.
		\end{equation}

		Taken together, the probabilities estimated in \eqref{eq:m1}\,--\,\eqref{eq:c1} amount 
		to less than $1$. Thus there exists a deterministic set $\ccB_\star\subseteq \ccA$
		of size $|\ccB_\star|\le 2\gamma n$ containing at most $\frac14\theta_\star^2 n$ pairs of 
		overlapping absorbers and satisfying $|\ccB_\star\cap \ccA(\sa)|\ge \tfrac54\thetas^2 n$
		for all $k$-tuples $\sa\in V(\P)^k$ of distinct vertices. 
		
		Now it suffices to check that a maximal subcollection $\ccB\subseteq \ccB_\star$ 
	   of mutually disjoint absorbers has the desired properties. The upper bound 
	   $|\ccB|\le |\ccB_\star|\le 2\gamma n$ is clear and due 
	   to $|\ccB_\star\setminus \ccB|\le \frac14\theta_\star^2 n$ we have 
	   \[
	   	|\ccB\cap \ccA(\sa)|
			\ge 
			\tfrac54\thetas^2n - \tfrac14\thetas^2n
			=
			\thetas^2n
		\]
	   for every $\sa$.		 
	\end{proof}	 
		
	It remains to connect the absorbers we have just selected into a path. 
	Recall that every member of $\ccB$ possesses $k+1$ pre-absorptions paths
	introduced in the last paragraph of Subsection~\ref{sec:abs_intro}. 
	Each of these paths has at most~$2k$ vertices, starts with a $\zeta$-rightconnectable $(k-1)$-tuple, 
	and ends with a $\zeta$-leftconnectable $(k-1)$-tuple. In fact, most of the 
	pre-absorptions paths even have $\zeta$-connectable end-tuples (see 
	Definition~\ref{d:abs}~\ref{it:ad}).
	
	Setting $r=(k+1)|\ccB|\le 4k\gamma n$, let $P_1, \ldots, P_r$ be the pre-absorption paths 
	of the absorbers in $\ccB$ enumerated in such a way that the 	end-tuples of $P_1$ and $P_r$ 
	are $\zeta$-connectable.
	We shall construct our absorbing path $P_A$ to be of the form 
	\[
		P_A=P_1C_1P_2C_2\ldots P_{r-1}C_{r-1}P_r\,,
	\]
	where $C_1, \ldots, C_{r-1}$ are connections that will be provided by 
	Proposition~\ref{prop:clk}. Since we intend to use the Connecting Lemma with $i=0$, 
	each of these connections is going to have   
	\[
		f=f(k, 0, \ell)=[4^{k-3}(2\ell+4)-2]k\le M-2k
	\]
	vertices, which will yield
	\begin{equation}\label{eq:1859}
		|V(P_A)|\le r\bigl(2k+(M-2k)\bigr)=rM\le 4kM\gamma n\,.
	\end{equation}

	We will determine the connections $C_1, \ldots, C_{r-1}$ one by one. When choosing $C_j$
	for some $j\in [r-1]$, the Connecting Lemma (Proposition~\ref{prop:clk}) offers us at 
	least $\theta n^f$ possible ways
	to connect $P_j$ with~$P_{j+1}$ by means of a path with $f$ inner vertices. As we need 
	to avoid both the already constructed parts of $P_A$ and the set $\cR$, there are at most 
	\[
		f(|\cR|+4kM\gamma n)n^{f-1}
		<
		(M\thetas^2+4kM^2\gamma)n^f
		\overset{\eqref{eq:gamma}}{<}
		8kM^2\gamma n^f
		\overset{\eqref{eq:gamma}}{\le}
		\theta n^f 
	\]
	potential connections we cannot use, and thus the choice of $C_j$ is indeed possible. 
	This concludes the description of the construction of $P_A$ and it remains to check that 
	the path we just defined has all required properties. 
	
	Condition~\ref{it:apl1} follows from~\eqref{eq:1859} and~\ref{it:apl2} is guaranteed by our 
	choice of the enumeration $P_1, \ldots, P_r$.	 
	For the proof of~\ref{it:apl3} we consider any set $Z\subseteq V(\P)\setminus V(P_A)$ 
	satisfying $|Z|\le 2 \thetas^2n$ and $|Z|\equiv 0\pmod{k}$.
	Let $\seq{a}_1, \ldots, \seq{a}_z\in V(\Psi)^k$ with $z=\frac{|Z|}k\le \thetas^2 n$ 
	be disjoint $k$-tuples with the property that every vertex from $Z$ occurs in exactly 
	one of them. By Claim~\ref{clm:1905} we can find 
	distinct absorbers $\seq{A}_1, \ldots, \seq{A}_z\in \ccB$ such that $\seq{A}_j$ 
	is a $(\seq{a}_j, \zeta)$-absorber for every $j\in [z]$. It remains to utilise these 
	absorbers one by one. 
\end{proof}

\section{Covering}
\label{sec:cov}

The aim of this section is to prove that under natural assumptions on the parameters almost all
vertices of every large~$k$-uniform $(\alpha, \beta, \ell, \mu)$-constellation can be covered  
by long paths whose first and last~$(k-1)$-tuples are connectable. Before formulating the precise statements let us give an overview of the argument, which will proceed by induction on $k$.  

In the induction step from $k-1$ to $k$ we study a largest possible collection~$\ccC$ 
of mutually vertex-disjoint \hbox{$M$-vertex} paths with connectable end-tuples and we denote 
the set of currently uncovered vertices by~$U$. If~$U$ is not small enough already, i.e., 
if $|U|=\Omega(|V(\Psi)|)$, then we partition~$V(\Psi)$ into sets of size $M$, the so-called blocks, 
such that the vertex set of each path in~$\ccC$ is one such block. Next, we show by probabilistic 
arguments that there is a special selection of~$M$ blocks, called a useful society below, such that 
their union $S$ has the property that for `many' vertices $u\in U$ the induction hypothesis applies 
to $\Psi_u[S]$. For such vertices $u$ we can then find $M+1$ (actually even more) long 
disjoint $(k-1)$-uniform paths in $\Psi_u[S]$ starting and ending with connectable~$(k-2)$-tuples.

In fact, for some still not too small set~$U''\subseteq U'$ these paths 
will coincide for all $u\in U''$, meaning that inserting vertices from~$U''$ at 
every~$k^{\mathrm{th}}$ position will yield~$M+1$ paths in~$\Psi$ with connectable end-tuples
(see Figure~\ref{fig:aug}).
This allows us to take the original paths contained in $S$ out of~$\ccC$ and to add the 
newly constructed paths instead, thus increasing the size of~$\ccC$. 
The following covering principle lies at the heart of this inductive argument. 

\begin{dfn}\label{d:herz}
	For $k\ge 3$ the statement $\heartsuit_k$ asserts that given~$\alpha,\beta, \thetas >0$ 
	and an odd integer~$\ell\ge 3$ there exists a constant~$\zeta_{\star\star}>0$
	such that for every $M_0\in\NN$ there exist a natural number~$M\ge M_0$ 
	with~$M\equiv -1 \pmod k$ and the following property: 
	
	For every sufficiently 
	large~$k$-uniform~$(\alpha,\beta,\ell,\frac{4\alpha}{17^k})$-constellation~$\Psi$ we can cover 
	all but at most~$\vartheta _\star^2\vert V(\Psi)\vert$ vertices by mutually 
	vertex-disjoint~$M$-vertex paths whose first and last~$(k-1)$-tuples 
	are~$\zeta_{\star\star}$-connectable.
\end{dfn}

For the base case $k=3$ we quote~\cite{Y}*{Lemma~2.14}. One needs to be a little bit careful here, 
because~\cite{Y} uses a slightly different notion of $\zeta_{\star\star}$-connectable pairs 
in \hbox{$3$-uniform hypergraphs}. However, every pair that 
is \hbox{$\zeta_{\star\star}$-connectable} in
the sense of~\cite{Y} is \hbox{$\zeta_{\star\star}$-connectable} in the sense of 
Definition~\ref{d:212} as well and, therefore,~\cite{Y}*{Lemma~2.14} is strictly stronger 
than~$\herz_3$.

\begin{fact}\label{f:herz3}
	The assertion $\heartsuit_3$ holds. \qed
\end{fact}

There is one issue with the inductive proof of $\herz_k$ sketched above: when 
applying the induction hypothesis to a $(k-1)$-uniform constellation of the form $\Psi_u[S]$, 
where~$S$ is the vertex set of a useful society, we would prefer to get a covering of almost 
all vertices in~$S$ by paths of length $\Omega(\sqrt{|S|})$ rather than $\Omega(1)$, but prima 
facie~$\heartsuit_{k-1}$ does not 
seem to deliver this. For this reason we also have to deal with the following statement capable
of providing coverings by very long paths. 

\begin{dfn}\label{d:pik}
	For $k\ge 3$ the covering principle $\spadesuit_k$ asserts that given~$\alpha,\beta,\xi>0$ 
	and an odd integer~$\ell\geq 3$, there exists an infinite arithmetic 
	progression~$P\subseteq k\NN$ with the following property. 
	
	If~$\Psi$ is a~$k$-uniform~$(\alpha,\beta,\ell,\frac{\alpha}{17^k})$-constellation,~$M\in P$, 
	and~$\mathfrak{B}\subseteq V(\Psi)^k$ is a collection of~$\xi$-bridges in~$\Psi$ 
	with~$\vert \mathfrak{B}\vert\geq\xi\vert V(\Psi)\vert^k$, then all but at 
	most~$\bigl\lfloor \xi\vert V(\Psi)\vert\bigr\rfloor +M$ vertices of~$\Psi$ can be covered 
	with mutually disjoint~$M$-vertex paths starting and ending with bridges from~$\mathfrak{B}$.
\end{dfn}

Observe that for a fixed $k$-uniform $(\alpha,\beta,\ell,\frac{\alpha}{17^k})$-constellation~$\Psi$ 
we can apply $\pik_k$ with every $M\in P$. For a larger value of $M$ we have to cover fewer vertices,
but, on the other hand, we need to cover them with longer paths. Thus there is no obvious monotonicity 
in~$M$.

Now we plan to establish the implication~$\herz_{k-1}\Rightarrow\pik_{k-1}\Rightarrow\herz_k$,
thus decomposing the induction step of the proof of $\heartsuit_k$ into two simpler tasks. 
They will be treated in Lemma~\ref{lem:herzzupik} and Lemma~\ref{lem:indstep}, respectively.

\begin{lemma}\label{lem:herzzupik}
	If~$k\geq 3$ and $\heartsuit_k$ holds, then so does $\spadesuit_k$.
\end{lemma}

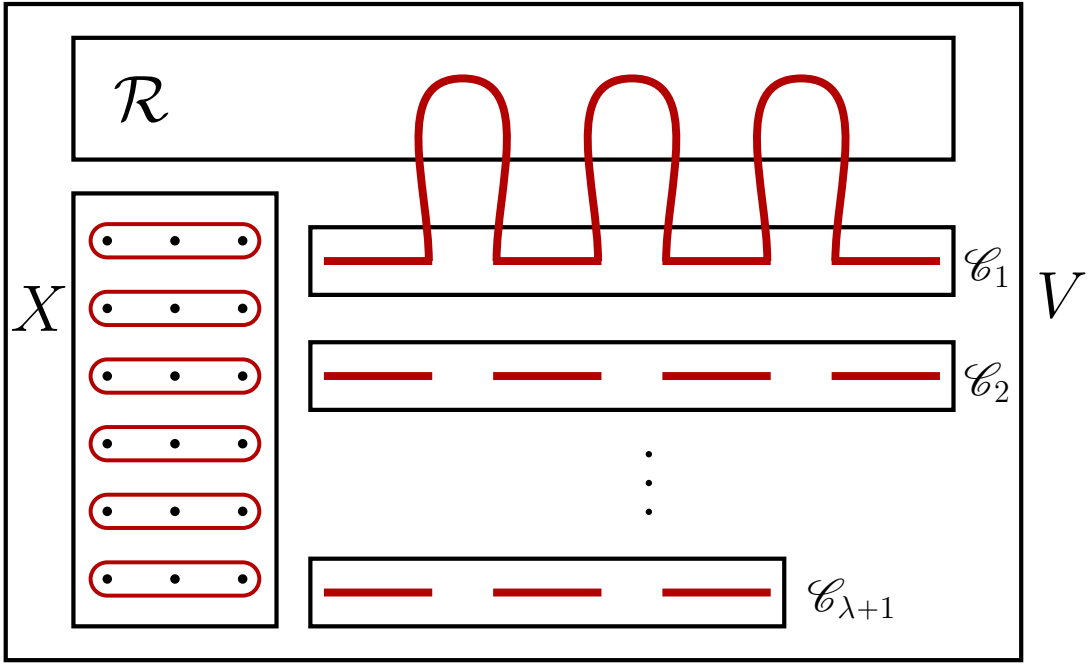
\begin{figure}[t]
	\begin{tikzpicture}[scale=0.9]
	
	\coordinate (a11) at (0,0);
	\coordinate (a12) at (0,0);
	\coordinate (a13) at (0,0);
	
	\foreach \i in {0, ..., 5}{
		\foreach \j in {0,1,2}{
			\coordinate (a\i\j) at (\j,\i);
			\fill (a\i\j) circle (2pt);
		}}
	
	\foreach \i in {0, ..., 5}{
		\qedge{(a\i0)}{(a\i1)}{(a\i2)}{7pt}{1.5pt}{red!70!black}{red!50!white,opacity=0};
		}
	
	\draw [ultra thick] (-1.5,-1.2) rectangle (13.5,8.5);
	\draw [ultra thick] (-.5,-.7) rectangle (2.5,5.7);
	\draw [ultra thick] (-.5,6.2) rectangle (12.5,8);
	\draw [ultra thick] (3,4.2) rectangle (12.5,5.2);
	\draw [ultra thick] (3,2.5) rectangle (12.5,3.5);
	\draw [ultra thick] (3,-.7) rectangle (10,.3);
	
	\node [rotate=90] at (8,1.5){{\Huge $\dots$}};
	\node at (-1,4){{\Huge $X$}};
	\node at (.5,7.1){{\Huge $\mathcal{R}$}};
	\node at (11,-.3){{\Large $\mathscr{C}_{\lambda+1}$}};
	\node at (13,2.9){{\Large $\mathscr{C}_{2}$}};
	\node at (13,4.6){{\Large $\mathscr{C}_{1}$}};
	\node at (14.1,4.2){{\Huge $V$}};
		
	\coordinate (a) at (3.2,-.2);
	
	\foreach \i in {0,1,2,3}{
		\coordinate (a0\i) at ($(a)+(\i*2.5,0)$);
		\coordinate (a1\i) at ($(a)+(\i*2.5,3.2)$);
		\coordinate (a2\i) at ($(a)+(\i*2.5,4.9)$);
		\draw [line width=3pt, color=red!70!black] (a1\i) -- ($(a1\i)+(1.6,0)$);
		\draw [line width=3pt, color=red!70!black] (a2\i) -- ($(a2\i)+(1.6,0)$);
		}
	
	\foreach \i in {0,1,2}{
		\draw [line width=3pt, color=red!70!black] (a0\i) -- ($(a0\i)+(1.6,0)$);
	}
	
	\draw [line width= 3pt, color = red!70!black] ($(a20)+(1.55,0)$) to [out = 90, in = 180]
	 +(.5,2.7)to [out = 0, in = 90] ($(a21)+(.05,0)$);
	
	\draw [line width= 3pt, color = red!70!black] ($(a21)+(1.55,0)$) to [out = 90, in = 180]
	+(.5,2.7)to [out = 0, in = 90] ($(a22)+(.055,0)$);
	
	\draw [line width= 3pt, color = red!70!black] ($(a22)+(1.55,0)$) to [out = 90, in = 180]
	+(.5,2.7)to [out = 0, in = 90] ($(a23)+(.055,0)$);

	\end{tikzpicture}
	\caption {The case $k=3$ of Lemma~\ref{lem:herzzupik}. The set $X$ of vertices is reserved for 
	bridges.}
	\label{fig:res}
\end{figure}

The idea behind the proof of this implication is the following (see Figure~\ref{fig:res}). 
Given an appropriate constellation~$\Psi$, our first step is to take out a reservoir set $\cR$. 
Next we decide which bridges from~$\gB$ are going to appear at the ends of the paths we are supposed 
to construct. After these choices are made,
we apply~$\herz_k$ to the constellation obtained from~$\Psi$ by removing~$\cR$ and the vertices 
reserved for the bridges, thus getting a covering of almost all remaining vertices with `short' paths.  
Now we partition the set of these paths into groups of size~$p$, where~$p$ denotes an 
arbitrary natural number. For each group we connect all its paths through the 
reservoir. Moreover, we connect the ends of the resulting paths to some of the bridges 
that have been put aside. In this manner we obtain a covering of almost all vertices of $\Psi$
with longer paths, whose precise length depends linearly on~$p$. Thus by varying~$p$ we can 
reach an arithmetic progression of possible lengths for the paths in the new covering. 

\begin{proof}[Proof of Lemma~\ref{lem:herzzupik}]
	Let~$\alpha$,~$\beta$,~$\xi>0$ and an odd integer~$\ell\ge 3$ be given. 
	Choose some auxiliary constants obeying the hierarchy 
	\[
		\alpha, \beta, \xi, k^{-1}, \ell^{-1} \gg \thetas \gg \zetass \gg \thetass \gg M^{-1} 
		\gg n_0^{-1}\,,
	\]
	where $M$ is an integer with $M\equiv -1\pmod{k}$. 
	
	We contend that 
	\[
		P=\bigl\{M'\in k\NN\colon M'>n_0\text{ and }M'\equiv f(k,0,\ell)+2k \pmod{M+f(k,0,\ell)}\bigr\}
	\]
	has the property demanded by $\spadesuit_k$.  
	
	By Definition~\ref{d:f} the number $f(k,0,\ell)$ is divisible by $k$ and, consequently, $P$ 
	is indeed an infinite arithmetic progression. Now let~$\Psi$ be 
	a~$k$-uniform~$(\alpha,\beta,\ell,\frac{\alpha}{17^k})$-constellation with $n$ vertices, 
	let~$M'\in P$ be 
	arbitrary, and let~$\mathfrak{B}\subseteq V(\Psi)^k$ be a set of~$\xi$-bridges in~$\Psi$ 
	with~$\vert \mathfrak{B}\vert \geq \xi \vert V(\Psi)\vert ^k$. 
	We are to cover all but at most $\xi\vert V(\Psi)\vert +M'$ vertices of~$\Psi$ by mutually 
	disjoint~$M'$-vertex paths starting and ending with bridges from~$\mathfrak{B}$.
	If~$|V(\Psi)|\le M'$, then the empty set is such a collection of paths. 
	Thus, we may assume that~$\vert V(\Psi)\vert >M'>n_0$. 
	
	Let~$\mathcal{R}\subseteq V(\Psi)$ with $|\cR|\le \thetas n$ be a the reservoir set
	provided by Proposition~\ref{prop:reservoir} with~$\thetas$,~$\frac\zetass2$ here in 
	place of $\xi$, $\zetass$ there. 
	For later use we record that due to $\thetass\ll \thetas, k^{-1}, \ell^{-1}$ 
	the case~$i=0$ of Corollary~\ref{lem:use-reservoir} yields:
	\begin{enumerate}
		\item[$(\star)$]\label{it:301} If $\cR'\subseteq \cR$ is an arbitrary set 
				with $|\cR'|\le \thetass^2 |V(\Psi)|$, the $(k-1)$-tuple $\sa\in V(\P)^{k-1}$
				is \hbox{$\frac\zetass2$-leftconnectable}, and $\sb\in V(\P)^{k-1}$ is 
				$\frac\zetass2$-rightconnectable and disjoint to $\seq{a}$, then there is 
				an $\sa$-$\sb$-path through $\cR\setminus \cR'$ with $f(k, 0, \ell)$ inner vertices. 
	\end{enumerate}

	Let~$b_1,\dots ,b_r$ be a maximal sequence of bridges from~$\mathfrak{B}$ that are mutually 
	disjoint and disjoint to~$\mathcal{R}$. Since the selected bridges and~$\mathcal{R}$ together 
	involve~$kr+\vert\mathcal{R}\vert$ vertices, the maximality implies
	\[
		k(kr+\vert\mathcal{R}\vert)\vert V(\Psi)\vert^{k-1}
		\ge 
		\vert\mathfrak{B}\vert
		\ge
		\xi \vert V(\Psi)\vert ^k\,,
	\]
	whence 
	\begin{equation}\label{eq:1926}
		r
		\ge 
		\frac{(\xi -k\thetas)\vert V(\Psi)\vert}{k^2}
		\ge 
		\vartheta_\star\vert V(\Psi)\vert\,.
	\end{equation}

	Set $x=\lfloor \thetas |V(\Psi)|\rfloor$ and let~$X$ be the set of vertices 
	constituting $b_1, \ldots, b_x$. Lemma~\ref{lem:1816} reveals that~$\Psi'=\Psi-(X\cup\mathcal{R})$ 
	is an~$\bigl(\frac{\alpha}{2},\frac{\beta}{2},\ell,\frac{2\alpha}{17^k}\bigr)$-constellation.
	Therefore, the principle~$\heartsuit_k$ yields a family~$\ccC$ of disjoint~$M$-vertex paths 
	in~$\Psi'$ which together cover all but 
	at most~$\vartheta_\star^2\vert V(\Psi')\vert$ vertices of~$\Psi'$ and whose  
	end-tuples are~$\zeta_{\star\star}$-connectable in~$\Psi '$. For later use we remark that 
	owing to Fact~\ref{f:41} the end-tuples of the paths in $\ccC$ are $\frac\zetass2$-connectable 
	in~$\Psi$.
	
	By the definition of $P$ there is a natural number $p$ such that 
	\[
		M'= \bigl(M+f(k,0,\ell)\bigr)p+f(k,0,\ell)+2k\,.
	\]
	Fix an arbitrary partition~$\ccC=\ccC_1\dcup\dots\dcup\ccC_{\lambda+1}$ 
	with~$\vert \ccC_1\vert =\dots =\vert \ccC_{\lambda}\vert 
	=p>\vert\ccC_{\lambda +1}\vert$. 
	
	Now we declare our strategy for constructing vertex-disjoint 
	paths $P_1, \ldots, P_\lambda\subseteq H(\Psi)$ 
	witnessing the conclusion of~$\spadesuit_k$. For every $j\in [\lambda]$ we first intend to form a 
	path $P'_j$ by connecting the~$p$ paths in~$\ccC_j$ through the reservoir~$\mathcal{R}$. 
	Subsequently, we plan to derive~$P_j$ from~$P'_j$ by connecting its ends with two bridges from the 
	list $b_1,\dots ,b_x$, say with~$b_{2j-1}$ and~$b_{2j}$. For all $p+1$ connections required 
	for this construction of $P_j$, we want to appeal to~$(\star)$.
	Clearly, if the paths $P_1, \ldots, P_\lambda$ can be constructed, then each of them will consist
	of~$M'$ vertices.  
	  
	Altogether, we are aiming for $(p+1)\lambda$ connections that 
	require a total number of
	\[
		(p+1)f(k, 0, \ell)\lambda
	\]
	vertices from the reservoir. 
	If this number is less than $\thetass^2n$, then repeated applications of~$(\star)$
	allow us to choose our connections disjointly. 
	Since $M\gg\thetass^{-1}\gg k, \ell$, we have indeed
	\[
		 (p+1)f(k,0,\ell)\lambda 
		 \le
		 2p \cdot 4^k k\ell\cdot \frac{|V(\Psi)|}{Mp} 
		 =
		 \frac{2\cdot 4^k k\ell |V(\Psi)|}M
		 <
		\thetass^2|V(\Psi)|\,.
	\]
	Similarly, 
	\[
		2\lambda
		\le 
		\frac{2|V(\Psi)|}{Mp}
		\le 
		\frac{2|V(\Psi)|}M
		\le 
		\thetas |V(\Psi)|
	\]
	proves that we have sufficiently many bridges at our disposal. 
	 
	Altogether, the vertex-disjoint paths $P_1, \ldots, P_\lambda\subseteq H(\Psi)$ 
	can indeed be constructed. The number of vertices of $\Psi$ they fail to cover can be bounded 
	from above by  
	\begin{align*}
		|X|+|\cR|+\Big|V(\Psi')\setminus \bigcup_{P\in \ccC}V(P)\Big|
			+
			\Big|\bigcup_{P\in \ccC_{\lambda+1}}V(P)\Big|
		&\le
		kx+\thetas |V(\Psi)|+\thetas^2 |V(\Psi)| +Mp \\
		&\le
		Mp+\bigl((k+1)\thetas+\thetas^2\bigr)|V(\Psi)| \\
		&\le
		M'+\xi \vert V(\Psi)\vert\,, 
	\end{align*}
	which concludes the proof of $\pik_k$.
\end{proof}

The proof of our next result involves some probabilistic arguments based on the following 
consequence of Janson's inequality (see~\cite{Y}*{Corollary A.3}). 

\begin{lemma}\label{lem:probappcor}
	Let $m\ge k$ and $M$ be positive integers, and let $\eta\in (0, \frac1{2k})$. Suppose that~$V$ 
	is a finite set and that 
	\[
	V=B_1\dcup\ldots\dcup B_\nu\dcup Z
	\]
	is a partition with $|B_1|=\ldots =|B_\nu|=M<\eta |V|$, $|Z|<\eta |V|$, and $\nu\ge m$. 
	Let $\ccS\subseteq \{B_1, \ldots, B_\nu\}$ be an $m$-element subset chosen uniformly
	at random and set $S=\bigcup \ccS$. Further, let~$\xi$ be a real number 
	with $\max(8k^2\eta, 16k^2/m) < \xi < 1$.
	\begin{enumerate}[label=\alabel]
		\item\label{it:1041a} If $Q\subseteq V^k$ has size $|Q|=d |V|^k$, then 
		\[
			\PP\bigl(\big||Q\cap S^k|-d(Mm)^k\big|\ge \xi (Mm)^k\bigr) 
			\le 
			12\sqrt{m}\exp\left(-\frac{\xi^2m}{48k^{2k+2}}\right)\,. 
		\]
		\item\label{it:1041b} Similarly, if $G$ denotes a $k$-uniform hypergraph with vertex 
		set $V$ and $d|V|^k/k!$ edges, then 
		\[
			\pushQED{\qed} 
			\PP\bigl(\big|e_G(S)-d(Mm)^k/k!\big|\ge \xi (Mm)^k/k!\bigr) 
			\le 
			12\sqrt{m}\exp\left(-\frac{\xi^2m}{48k^{2k+2}}\right)\,. \qedhere
			\popQED
		\]
	\end{enumerate} 
\end{lemma}

This has the following consequence on random subconstellations. 

\begin{lemma}\label{l:cr}
	Given $k\ge 2$, $\alpha, \beta, \mu, \xi>0$, and an odd integer $\ell\ge 3$ there exists 
	a natural number $M_0$ such that the following holds for every $M\ge M_0$.  
	If $\Psi$ is a sufficiently large
	$k$-uniform $(\alpha, \beta, \ell, \mu)$-constellation,
	\[
		V(\Psi)=B_1\dcup \ldots \dcup B_\nu\dcup B'
	\]
	is a partition with $|B_1|=\ldots = |B_\nu|=M$ and $|B'|<2M$,  
	and $\gB\subseteq V(\Psi)^k$ is a set of $\xi$-bridges in $\Psi$ of size $|\gB|\ge \xi|V(\Psi)|^k$, 
	then there are at least $\frac34\binom\nu M$ sets $\ccS\subseteq \{B_1, \ldots, B_\nu\}$
	of size $M$ such that their union $S=\bigcup \ccS$ has the properties that $\Psi[S]$ 
	is a $(\frac\alpha2, \frac\beta2, \ell, 2\mu)$-constellation and   
	\[
		\gB_\star=\bigl\{\seq{x}\in \gB\cap S^k \colon 
		\seq{x} \text{ is a $\tfrac\xi2$-bridge in }\Psi[S]\bigr\}
	\]
	has at least the size $|\gB_\star|\ge \frac\xi2 |S|^k$.	
\end{lemma}

\begin{proof}
	Let $M_0\gg \alpha^{-1}, \beta^{-1}, \mu^{-1}, \xi^{-1}, k, \ell$ be sufficiently large.
	We call the sets $B_1, \ldots, B_\nu$ {\it blocks}. 
	Choose a set $\ccS\subseteq \{B_1, \ldots, B_\nu\}$ of $M$ blocks uniformly at random 
	among all~$\binom\nu M$ possibilities. We shall prove that the probability 
	that $S=\bigcup \ccS$ fails to have the desired properties is at most~$\exp\bigl(-\Omega(M)\bigr)$,
	where the implied constant only depends on~$\alpha$,~$\beta$,~$\mu$,~$\xi$,~$k$, and~$\ell$. 
	Hence, by choosing~$M_0$ sufficiently large, this probability can be pushed below~$\frac14$, 
	as desired. It will be convenient to set $V'=V\setminus B'$. For $y\in V'$ we denote the unique 
	block containing~$y$ by $B_y$. 
	
	\begin{claim}\label{clm:1855}
		The event that $\Psi[S]$ fails to be 
		a $(\frac\alpha2, \frac\beta2, \ell, 2\mu)$-constellation has at most 
		the probability $\exp\bigl(-\Omega(M)\bigr)$.  
	\end{claim}
	
	\begin{proof}
		We begin by estimating the probability of the unfortunate event~$\gU$ 
		that~$\Psi[S]$ fails to be a~$(\frac{\alpha}{2}, 2\mu)$-constellation.
		For an arbitrary set $x\in (V')^{(k-2)}$ 
		we define
		\[
			\ccZ_x=\{B_y\colon y\in x\}\,, \quad t_x=\vert \ccZ_x\vert\in [k-2]\,, 
			\text{ and } \quad Z_x=\bigcup\ccZ_x\,.
		\]
 		Further, we consider the conditional probabilities 
		\begin{align*}
			P_1(x)&=\PP\left(e_{\Psi_{x}}(S\setminus Z_x)
				<\left(\frac{5}{9}+\frac{2\alpha}{3}\right)\frac{(M-t_x)^2M^2}{2}
				\copr x\in S^{(k-2)}\right),\\
			P_2(x)&=\PP\left(\big| V(R_{x}^{\Psi}[S])\big| 
				< \left(\frac{2}{3}+\frac{\alpha}{3}\right) (M-t_x)M
				\copr x\in S^{(k-2)}\right)\,,
			\intertext{ and }		
				P_3(x)&=\PP\left(e_{\Psi_{x}[S]}\left(V(R_{x}^{\Psi}[S]),
					S\setminus V(R_{x}^{\Psi}[S])\right)
					> 2\mu(M-t_x)^2M^2\copr x\in S^{(k-2)}\right)
		\end{align*}
		and observe that
		\begin{align}\label{eq:P(U_1)first}
			\PP(\mathfrak{U})
			\le
			\sum_{x\in (V')^{(k-2)}}\PP(x\in S^{(k-2)})\big(P_1(x)+P_2(x)+P_3(x)\big)\,.
		\end{align}
		So if we manage to prove 
		\begin{equation}\label{eq:1932}
			P_1(x), P_2(x), P_3(x)\le \exp\bigl(-\Omega(M)\bigr)\,,
		\end{equation}
		then 
		\begin{equation}\label{eq:P(notU_1)}
			\PP(\gU)
			\le
			(M^2)^{k-2}\exp\bigl(-\Omega(M)\bigr) 
			\le
			\exp\bigl(-\Omega(M)\bigr)
		\end{equation}
		will follow. Thus our next goal is to establish~\eqref{eq:1932}. 
	
		To this end, we will repeatedly apply Lemma~\ref{lem:probappcor} 
		with 
		\[
		 	M-t_x\,, \frac{kM}{n}\,, B'\cup Z_x\,, \nu -t_x\,, \text{ and } 
			\min\bigl\{\tfrac\alpha6, \mu\bigr\}
		\]
		here in place of 
		\[	
				m\,, \eta\,, Z\,, \nu\,,\text{ and }\xi
		\]
		there and relocating the elements of~$\ccZ_x$ to the exceptional set of the partition. 
				
		First, the minimum degree condition imposed on~$H(\Psi)$ implies 
		that the graph~$H(\Psi_{x})$ has at 
		least~$\left(\frac{5}{9}+\alpha\right) \frac{|V(\Psi)|^{2}}{2}$ edges. 
		So Lemma~\ref{lem:probappcor}~\ref{it:1041b} applied with~$2$ and~$H(\Psi_{x})$
		here in place of~$k$ and~$G$ there 
		yields $P_1(x)\le\exp\bigl(-\Omega(M)\bigr)$.
		
		Second, we know that
		$\vert V(R_{x}^{\Psi})\vert\geq \left(\frac23+\frac{\alpha}{2}\right)|V(\Psi)|$,
		since~$\Psi$ is an~$(\alpha, \mu)$-constellation. 
		Hence, applying Lemma~\ref{lem:probappcor}~\ref{it:1041a} with~$1$ and~$V(R_{x}^{\Psi})$ 
		here instead of~$k$ and~$Q$ there entails $P_2(x)\le\exp\bigl(-\Omega(M)\bigr)$.
		
		Lastly, from~$\Psi$ being a~$(\alpha,\mu)$-constellation it also follows that 
		\[	
			e_{\Psi_{x}}\left(V(R_{x}^{\Psi}), V\setminus V(R_{x}^{\Psi})\right)
			\le \mu |V(\Psi)|^2\,.
		\]
		Hence, Lemma~\ref{lem:probappcor}~\ref{it:1041b} applied to the bipartite subgraph
		of $H(\Psi_x)$ between $V(R^\Psi_x)$ and its complement  
		tells us that $P_3(x)\le\exp\bigl(-\Omega(M)\bigr)$.
		This concludes the proof of~\eqref{eq:1932} and, hence, of~\eqref{eq:P(notU_1)}. 
		An analogous proof allows us to transfer part~\ref{it:222b} of Definition~\ref{d:222}
		from~$\Psi$ to~$\Psi[S]$ and we omit the details.
	\end{proof}
	
	It remains to prove that the event $|\gB_\star|\ge \frac\xi2 |S|^k$ has high probability 
	as well. Here we start with the estimate
	\[	
		\PP\bigl(|\gB_\star|\le \tfrac\xi2 |S|^k\bigr)
		\le
		\PP\bigl(|\gB\cap S^k|\le \tfrac\xi2 |S|^k\bigr)
		+
		\PP(\neg \gE)\,,
	\]
	where $\gE$ denotes the event that every $\xi$-bridge $\seq{x}\in \gB\cap S^k$ is 
	a $\frac\xi2$-bridge in~$\Psi[S]$. 
	Another application of Lemma~\ref{lem:probappcor}~\ref{it:1041a} tells us that the 
	first summand is at most $\exp\bigl(-\Omega(M)\bigr)$ and thus it remains to prove that
	\begin{equation}\label{eq:endziel}
		\PP(\neg \gE)\le \exp\bigl(-\Omega(M)\bigr)\,.
	\end{equation}
	Towards this goal we analyse how connectability transfers to $\Psi[S]$.
	
	\begin{claim}\label{clm:2234}
		If $k'\in [k-1]$, $z, z'\in (V')^{(k-1-k')}$, 
		and~$\seq{x}\in \bigr(V'\setminus (z\cup z')\bigr)^{k'}$ 
		is a~$\xi$-leftconnectable tuple in~$\Psi_z$, then 
		\[
			\PP\bigl(\seq{x} \text{ fails to be $\tfrac\xi2$-leftconnectable in } 
				\Psi_z[S]\coprn \seq{x}\in S^{k'}
			 \text{ and } z'\subseteq S\bigr)
			\le
			\exp\bigl(-\Omega(M)\bigr)\,.
		\]
	\end{claim}
	
	\begin{proof}
		We argue by induction on $k'$. In the base case $k'=1$ the probability under consideration 
		vanishes. This is because a $1$-tuple $\seq{x}=(x)$ is $\xi$-leftconnectable in $\Psi_z$
		if and only if $x\in V(R^\Psi_z)$. Moreover, if $x\in S\setminus z$, 
		then $(x)$ is $\frac\xi2$-leftconnectable in~$\Psi_z[S]$ if and only if~$x\in R^{\Psi[S]}_z$.
		Due to $R^{\Psi[S]}_z=R^\Psi_z[S]$ these two statements are equivalent to each other. 
		
		For the induction step from $k'-1$ to $k'$ we write $\seq{x}=(x_1, \ldots, x_{k'})$
		and recall that the $\xi$-leftconnectability of $\seq{x}$ in $\Psi_z$ means that 
		$|U|\ge \xi|V(\Psi_z)|$, where 
		\begin{multline*}
			 U=\bigl\{u\in V(\Psi_z)\colon x_1\dots x_{k'}u \in E(\Psi_z)
				 \textrm{ and }
			(x_2, \dots, x_{k'}) \textrm{ is $\xi$-leftconnectable in } \Psi_{zu} \bigr\}\,.
		\end{multline*}
		Assuming $\seq{x}\in S^{k'}$ the analogous set whose size decides 
		whether $\seq{x}$ is $\frac\xi2$-leftconnectable
		in $\Psi_z[S]$ either contains $U\cap S$ as a subset, 
		or it does not. Accordingly, if $\seq{x}$ fails to be $\frac\xi2$-leftconnectable 
		in $\Psi_z[S]$, then either $|U\cap S|\le \frac\xi2 |V(\Psi_z[S])|$ or the event $\gA$
		that for some $u\in S\cap U$ the $(k'-1)$-tuple $(x_2, \dots, x_{k'})$ 
		fails to be $\frac\xi2$-leftconnectable in $\Psi_{zu}[S]$ occurs. For this reason,
		it suffices to prove
		\begin{align}
			\PP\bigl(|U\cap S|\le \tfrac\xi2 |S| \coprn \seq{x}\in S^{k'} 
			\text{ and } z'\subseteq S\bigr) &\le \exp\bigl(-\Omega(M)\bigr)
			\label{eq:1428} \\
			\text{ and } \quad
			\PP\bigl(\gA \coprn \seq{x}\in S^{k'}
			\text{ and } z'\subseteq S\bigr) &\le \exp\bigl(-\Omega(M)\bigr)\,.
			\label{eq:1429}
		\end{align}
		Now~\eqref{eq:1428} follows in the usual way from Lemma~\ref{lem:probappcor}~\ref{it:1041a}.
		To prove~\eqref{eq:1429} we observe that the induction hypothesis yields 
		\begin{multline*}
			\PP\bigl((x_2, \dots, x_{k'}) \text{ fails to be $\tfrac\xi2$-leftconnectable in } 
				\Psi_{zu}[S]\coprn (x_2, \dots, x_{k'}) \in S^{k'-1}\,, \\ 
			\text{ and } (z'\cup\{x_1\})\subseteq S \bigr)
			\le
			\exp\bigl(-\Omega(M)\bigr)
		\end{multline*}
		for every $u\in U$, whence
		\begin{align*}
			\PP(\gA \coprn \seq{x}\in S^{k'} \text{ and } z'\subseteq S) 
			&\le
			\sum_{u\in U}\PP(u\in S)\exp\bigl(-\Omega(M)\bigr)\\
			&\le 
			M^2\exp\bigl(-\Omega(M)\bigr)  
			\le
			\exp\bigl(-\Omega(M)\bigr)\,. \qedhere
		\end{align*}
	\end{proof}
	
	By applying the case $k'=k-1$ of Claim~\ref{clm:2234} to all $\xi$-leftconnectable 
	$(k-1)$-tuples in $\Psi$ we obtain
	\begin{multline*}
		\PP\bigl(\text{Some $\seq{x}\in S^{k-1}$ that is $\xi$-leftconnectable in $\Psi$}\\
			\text{fails to be $\tfrac\xi2$-leftconnectable in $\Psi[S]$}\bigr)
		\le
			\exp\bigl(-\Omega(M)\bigr)\,.
	\end{multline*}
	By symmetry the same holds for rightconnectability as well and, therefore,
	\[
		\PP\bigl(\text{Some  $\xi$-bridge $\seq{x}\in S^k$ fails to be a
			$\tfrac\xi2$-bridge in $\Psi[S]$}\bigr)
		\le
			\exp\bigl(-\Omega(M)\bigr)\,.
	\]
	In other words, we have thereby proved~\eqref{eq:endziel} and, hence, Lemma~\ref{l:cr}.
\end{proof}

The next lemma shows how to ascend from $(k-1)$-uniform coverings to $k$-uniform coverings.

\begin{lemma}\label{lem:indstep}
	For every~$k\ge 4$ the covering principle $\pik_{k-1}$ implies $\herz_k$.
\end{lemma}

\begin{proof}
	Let~$\alpha, \beta, \thetas>0$, and an odd integer~$\ell\geq 3$ be given. 
	Without loss of generality we may assume that $\thetas\ll \alpha, \beta, k^{-1}, \ell^{-1}$. 
	Pick a sufficiently small constant 
	\begin{equation}\label{eq:2155}
		\zetass\ll \thetas\,.
	\end{equation}
	The statement~$\pik_{k-1}$ applied 
	to~$\frac{\alpha}{2}$,~$\frac{\beta}{2}$,~$\ell$,~$\frac\zetass2$ 
	here in place of~$\alpha$,~$\beta$,~$\ell$,~$\xi$ there delivers an infinite arithmetic 
	progression~$P\subseteq (k-1)\NN$. Choose~$M\gg\zeta_{\star\star}^{-1}$ 
	such that~$\frac{k-1}{k}(M+1)\in P$ and notice that $M\equiv -1\pmod{k}$ is clear. 
	
	Now let~$\Psi$ be a~$(\alpha,\beta,\ell,\frac{4\alpha}{17^k})$-constellation
	on $n$ vertices, where $n$ is sufficiently large. 	
	We are to prove that all but at most $\thetas^2|V(\Psi)|$ vertices of $\Psi$ can be covered
	by vertex-disjoint $M$-vertex paths starting end ending with $\zetass$-connectable $(k-1)$-tuples.
	Let 
	\begin{align*}
		\ccP
		=
		\bigl\{P\subseteq H(\Psi)\colon P \text{ is a }k&\text{-uniform }M\text{-vertex path}\\& 
		\text{whose first and last }(k-1)\text{-tuple is }\zeta_{\star\star}\text{-connectable}\bigr\} 
	\end{align*}
	be the collection of all paths that might occur in such a covering, and 
	let~$\ccC\subseteq \ccP$ be a maximal subcollection of vertex-disjoint paths from~$\ccP$. 
	Further, let
	\[
		U=V(\Psi)\setminus\bigcup_{P\in \ccC}V(P)
	\]
	be the set of uncovered vertices. We may assume that 
	\begin{equation}\label{eq:1431}
		\vert U\vert > \vartheta_\star^2\vert V(\Psi)\vert\,,
	\end{equation}
	since otherwise nothing is left to show. Now roughly speaking the strategy is to find a
	set $S\subseteq V(\Psi)$ of size $M^2$ meeting at most $M$ paths from $\ccC$ such that 
	for `many' vertices~$u\in U$ we can apply $\pik_{k-1}$ to the $(k-1)$-uniform constellation
	$\Psi_u[S]$, thus getting at least $M+1$ vertex-disjoint paths with $\frac{k-1}{k}(M+1)$
	vertices. These paths will agree for many vertices~$u\in U$ and can then be 
	augmented to $k$-uniform paths engendering a contradiction to the maximality of~$\ccC$. 
	In the intended application of $\pik_{k-1}$ we are allowed to specify a set of 
	bridges $\mathfrak B$ that we potentially would like to see at the ends of the paths we obtain. 
	Since we ultimately aim at generating paths in $\ccP$ and, hence, paths starting and ending 
	with $\zetass$-connectable $(k-1)$-tuples, it seems advisable to let $\mathfrak B$ be
	the set of $\frac\zetass2$-bridges in~$\Psi_u[S]$ that are $\zetass$-connectable in $\Psi$. 
	This choice of $\mathfrak B$ is only permissible if~$|\mathfrak B|$ is sufficiently large 
	(i.e., at least $\frac\zetass2 |S|^{k-1}$). Our way of ensuring this in sufficiently many cases 
	exploits that for fixed $u\in U$ and a random choice of $S\subseteq V(\Psi)$ 
	Lemma~\ref{l:cr} tells us that the $\zetass$-bridges in~$\Psi_u$ are likely to 
	be $\frac\zetass2$-bridges in~$\Psi_u[S]$.
	Thus it suffices to focus on vertices $u\in U$ which are not in the set	
	\[
		\Ubad
		=
		\bigl\{u\in U\colon \text{at most }\tfrac{1}{20}n^{k-1}\text{ of the }
			\zeta_{\star\star}\text{-bridges in }\Psi_u\text{ are }
			\zeta_{\star\star}\text{-connectable in }\Psi\bigr\}\,.
	\]
	The next claim states that this set is indeed small.
	
	\begin{claim}\label{cl:Ubadsmall}
		We have $\vert \Ubad\vert \le 40\zeta_{\star\star}n$.
	\end{claim}
	
	\begin{proof}
		Set 
		\begin{multline*}
			\Pi=\bigl\{(x_1,\dots ,x_{k-1},u)\in V(\Psi)^{k-1}\times \Ubad\colon 
			(x_1,\dots ,x_{k-1})\text{ is a }\zeta_{\star\star}\text{-bridge in }\Psi_u\\
			\text{ but not } \zeta_{\star\star}\text{-connectable in }\Psi\bigr\}\,.
		\end{multline*}
		For every~$u\in \Ubad$ Corollary~\ref{c:219} tells us that the number 
		of~$\zeta_{\star\star}$-bridges $(x_1,\dots ,x_{k-1})$ in~$\Psi_u$
		is at least~$\frac19(n-1)^{k-1}>\frac{1}{10}n^{k-1}$ and by the definition of~$U_{\text{bad}}$ 
		at least~$\frac{1}{20}n^{k-1}$ among them fail to be~$\zeta_{\star\star}$-connectable in~$\Psi$. 
		This proves that 
		\[
			\vert \Pi\vert
			\ge 
			\frac{1}{20}n^{k-1}\vert \Ubad\vert\,.
		\]

		On the other hand, an upper bound on~$|\Pi|$ can be obtained as follows. 
		Let~$\Pi_{\text{left}}$ be the set of $k$-tuples in~$\Pi$ for which~$(x_1,\dots ,x_{k-1})$ 
		fails to be~$\zeta_{\star\star}$-leftconnectable and define $\Pi_{\text{right}}$ 
		similarly with respect to rightconnectability. As a~$(k-1)$-tuple that is 
		not~$\zeta_{\star\star}$-leftconnectable in~$\Psi$ can only be a~$\zeta_{\star\star}$-bridge 
		in~$\Psi_u$ for less than~$\zeta_{\star\star}n$ vertices~$u$, we have 
		$\vert \Pi_{\text{left}}\vert\le \zeta_{\star\star}n^k$. The same upper 
		bound can be proved for $|\Pi_{\text{right}}|$ 
		and because of~$\Pi=\Pi_{\text{left}}\cup\Pi_{\text{right}}$
		this yields~$\vert \Pi\vert\le 2\zeta_{\star\star}n^k$. 
		Combining the two bounds on~$\vert \Pi\vert$ we obtain 
		indeed~$\vert U_{\text{bad}}\vert\leq 40\zeta_{\star\star} n$.
	\end{proof}
	
	Because of our choice of $\zetass$ in~\eqref{eq:2155} this 
	yields $|\Ubad|\le \frac12 \thetas^2n$, which combined with~\eqref{eq:1431} implies 
	\begin{equation}\label{eq:1427}
		|U\setminus\Ubad|\ge \frac12\thetas^2n\,.
	\end{equation} 

	Next we will partition the vertex set into blocks some of which will later be selected 
	randomly for hosting the augmentation of $\ccC$. 
	Form a partition 
	\begin{equation}\label{eq:1901}
		V(\Psi)
		=
		B_1\dcup\dots\dcup B_{\nu}\dcup B'\,,
	\end{equation}
	with~$\vert B_1\vert =\dots =\vert B_{\nu}\vert =M >\vert B'\vert$, 
	where the first $|\ccC|$ classes $B_1,\dots , B_{\vert\ccC\vert}$ are the vertex sets 
	of the paths in the collection~$\ccC$, and $B_{|\ccC|+1},\dots , B_{\nu}$ are arbitrary 
	disjoint $M$-sets making~\eqref{eq:1901} true.
	The sets~$B_1, \ldots, B_\nu$ are called \textit{blocks}. 
	A \textit{society} is a set of~$M$ blocks. We point out that 
	\begin{equation}\label{eq:1924}
		\text{if $\ccS$ is a society and $S=\bigcup \ccS$, then $|S|=M^2$.}
	\end{equation}

	\begin{dfn}\label{dfn:usefulsociety}
		A society~$\ccS$ with~$S=\bigcup \ccS$ is called \textit{useful} 
		for a vertex~$u\in U$ if
		\begin{enumerate}[label=\nlabel]
			\item\label{it:usoc1} $u\not\in S$,
			\item\label{it:usoc2} $\Psi_u[S]$ is 
				a~$(k-1)$-uniform
				$(\frac{\alpha}{2},\frac{\beta}{2},\ell,\frac{\alpha/2}{17^{k-1}})$-constellation.
			\item\label{it:usoc3} The number of~$(k-1)$-tuples in~$S^{k-1}$ that 
				are~$\frac\zetass2$-bridges in~$\Psi_u[S]$ 
				and~\hbox{$\zeta_{\star\star}$-connectable} in~$\Psi$
				is at least~$\frac\zetass2|S|^{k-1}$.
		\end{enumerate}
	\end{dfn}

	The next claim explains the naming of useful societies:~$\Psi_u[S]$ 
	contains~$M+1$ ``suitable'' paths.
	
	\begin{claim}\label{cl:uSocietyuseful}
		If a society~$\ccS$ is useful for~$u\in U$ and~$S=\bigcup\ccS$, 
		then there is a collection~$\ccW$ of mutually disjoint $(k-1)$-uniform 
		paths in~$\Psi_u[S]$ with the following properties.
		\begin{enumerate}[label=\rmlabel]
			\item Every path in $\ccW$ has $\frac{k-1}{k}(M+1)$ vertices. 
			\item Every path in $\ccW$ starts and ends with a $(k-1)$-tuple that 
				is~$\zeta_{\star\star}$-connectable in~$\Psi$.
			\item $|\ccW|\ge M+1$. 
		\end{enumerate}
	\end{claim}
	
	\begin{proof}
		By Definition~\ref{dfn:usefulsociety}~\ref{it:usoc3} and~\eqref{eq:1924} the set
		\[
			\Xi=\bigl\{\seq{e}\in S^{k-1}\colon 
			\seq{e} \text{ is }\zeta_{\star\star}\text{-connectable in }\Psi
			\text{ and a }\tfrac\zetass2\text{-bridge in }\Psi_u[S]\bigr\}
		\]
		satisfies~$\vert \Xi\vert\geq \frac\zetass2(M^2)^{k-1}$. 
		Now we apply $\pik_{k-1}$ to~$\Psi_u[S]$,~$\Xi$,~$\frac\zetass2$, 
		and~${\frac{k-1}{k}(M+1)}$ here in place of~$\Psi$,~$\mathfrak{B}$,~$\xi$, and~$M$ 
		there -- which is permissible due to the selection of parameters in the beginning of 
		the proof of Lemma~\ref{lem:indstep}. 
		
		This application of $\pik_{k-1}$ yields a collection~$\ccW$ of mutually 
		disjoint~$(k-1)$-uniform paths in~$\Psi_u[S]$ 
		that covers all but at most~$\frac\zetass2\vert S\vert+\frac{k-1}{k}(M+1)$ vertices 
		of~$S$ such that each path starts and ends with a bridge from~$\Xi$. 
		Since each bridge in~$\Xi$ is a~$\zeta_{\star\star}$-connectable tuple in~$\Psi$, 
		it remains to check that~$\vert \ccW\vert\geq M+1$.
		Because of~$M\gg\zeta_{\star\star}^{-1}\gg k$ we have indeed 
		\[
			\vert \ccW\vert
			\ge
			\frac{(1-\zetass/2)M^2-\frac{k-1}{k}(M+1)}{\frac{k-1}{k}(M+1)}
			\ge
			\frac{(1-\zeta_{\star\star})M(M+1)}{(1-\zeta_{\star\star})M}
			= 
			M+1\,. \qedhere
		\]
	\end{proof}
		
	Lemma~\ref{l:cr} implies that some society is useful for many vertices. 
	
	\begin{claim}\label{cl:onesocusefulformany}
		There exists a society~$\ccS$ that is useful 
		for~$\frac{2}{3}\vert U\setminus U_{\text{bad}}\vert$ vertices in~$U\setminus U_{\text{bad}}$.
	\end{claim}

	\begin{proof}
		By double counting it suffices to establish that for every vertex $u\in U\setminus\Ubad$
		at least~$\frac23$ of all societies are useful. Fix an arbitrary such vertex $u$ and 
		suppose first that $u\not\in B'$. Without loss of generality we may assume 
		that $u\in B_\nu$. We plan to apply Lemma~\ref{l:cr} with
		$(k-1, \frac{\alpha}{4\cdot 17^{k-1}}, \zetass)$ here in place of 
		$(k, \mu, \xi)$ there to the $(k-1)$-uniform 
		constellation $\Psi_u$, the partition 
		\[
			V(\Psi_u)=B_1\dcup\ldots\dcup B_{\nu-1}\dcup (B_\nu\cup B'\setminus\{u\})\,,
		\]
		and the set 
		\[
			\gB_u=\bigl\{\seq{x}\in V(\Psi_u)^{k-1}\colon \seq{x} 
			\text{ is $\zetass$-connectable in $\Psi$ and a $\zetass$-bridge in $\Psi_u$}\bigr\}\,.
		\]

		Notice that Fact~\ref{f:linkfullc} tell us that $\Psi_u$ 
		is indeed an $(\alpha, \beta, \ell, \frac{\alpha}{4\cdot 17^{k-1}})$-constellation. Moreover, 
		$u\not\in\Ubad$ implies $|\gB_u|\ge\frac1{20}n^{k-1}>\zetass|V(\Psi_u)|^{k-1}$.
		So all assumptions of 
		Lemma~\ref{l:cr} hold and we conclude that at 
		least $\frac34\binom{\nu-1}M>\frac23\binom{\nu}M$ societies are useful for $u$.
		The case $u\in B'$ is similar.    
	\end{proof}
	
	For the remainder of this proof we fix a society~$\ccS$ that is useful 
	for at least~$\frac{2}{3}\vert U\setminus U_{\text{bad}}\vert$ vertices 
	in~$U\setminus U_{\text{bad}}$ and set~$S=\bigcup\ccS$. 
	Claim~\ref{cl:uSocietyuseful} informs us that for every $u\in U$, for which~$\ccS$ is useful, 
	there is a collection~$\ccW_u$ of~$M+1$ mutually vertex disjoint~$(k-1)$-uniform paths 
	in~$\Psi_u[S]$ consisting of~$\frac{k-1}{k}(M+1)$ vertices each, which start and end 
	with~$\zetass$-connectable $(k-1)$-tuples.
 
	Since there are at most~$(M^2)!$ possibilities to order the vertices in~$S$, 
	there has to exist a subset~$U'\subseteq U\setminus U_{\text{bad}}$ 
	such that~$\ccW_u=\ccW$ is the same for every~$u\in U'$ and 
	\[
		\vert U'\vert
		\ge
		\frac{\frac23\vert U\setminus U_{\text{bad}}\vert}{(M^2)!}
		\overset{\eqref{eq:1427}}{\ge}
		\frac{\vartheta_\star^2 n}{3(M^2)!}
		\ge 
		\frac{(M-(k-1))(M+1)}{k}\,.
	\]
	Now, for every path in~$\ccW$ put~$\frac{M-(k-1)}{k}$ distinct vertices from~$U'$ aside
	and insert them at every~$k$-th position into the path from~$\ccW$ (see Figure~\ref{fig:aug}).
	
		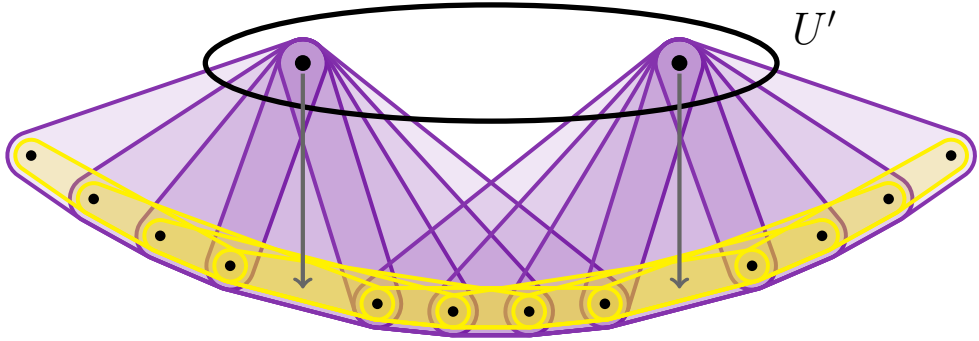
\begin{figure}[h!]
	\begin{tikzpicture}[scale=1]
		
	\def\an{2.9};
	\def\s{270-13*\an};
	\def\ra{10cm};
	
	\foreach \i in {0,...,13}{
		\coordinate (a\i) at (\s+2*\i*\an:\ra);
	}
	
	\coordinate (u1) at ($(a4) + (0,3)$);
	\coordinate (u2) at ($(a9) + (0,3)$);
	
	\def \co{violet!80!blue}

		\foreach \ii/\ij/\ik/\im/\in in {
			a3/a2/a1/a0/u1,
			a5/a3/a2/a1/u1,
			a6/a5/a3/a2/u1,
			a7/a6/a5/a3/u1,
			a8/a7/a6/a5/u1,
			a8/a7/a6/a5/u2,
			a10/a8/a7/a6/u2,
			a11/a10/a8/a7/u2,
			a12/a11/a10/a8/u2,
			a13/a12/a11/a10/u2}	
		{\pedge{(\ii)}{(\ij)}{(\ik)}{(\im)}{(\in)}{9pt}{1.5pt}{\co!80!white}{\co,opacity=0.1};}

		\foreach \i/\j/\k/\m in {
			a3/a2/a1/a0,
			a5/a3/a2/a1,
			a6/a5/a3/a2,
			a7/a6/a5/a3,
			a8/a7/a6/a5,
			a10/a8/a7/a6,
			a11/a10/a8/a7,
			a12/a11/a10/a8,
			a13/a12/a11/a10}	
				\redge{(\i)}{(\j)}{(\k)}{(\m)}{6pt}{1.5pt}{yellow!100!black}{yellow,opacity=0.2};

	\node at (4.3,-6.2) {\Large $U'$};
	
	\draw[black, line width=2pt] ($(u1)!.5!(u2)$) ellipse (3.8cm and 22pt);	
	
	\draw [ultra thick, black!60!white, shorten <= 4pt, ->] (u1) -- (a4); 
	\draw [ultra thick, black!60!white, shorten <= 4pt, ->] (u2) -- (a9); 				
				
	\foreach \i in {1,2}{
		\fill (u\i) circle (3pt);
	}			
	
	\foreach \i in {0,1,2,3,5,6,7,8,10,11, 12, 13}{
		\fill (a\i) circle (2pt);
	}

	\end{tikzpicture}
	\caption {Augmenting a yellow $\frac 45 (M+1)$-vertex path to a lila $M$-vertex path.}
	\label{fig:aug}
\end{figure}
 
	Since the starting and ending~$(k-1)$-tuples of every path in~$\ccW$ 
	are~$\zeta_{\star\star}$-connectable in~$\Psi$ and the insertion of the 
	additional vertices increases their length to~$\frac{k-1}{k}(M+1)+ \frac{M-(k-1)}{k}=M$, 
	the resulting~$M+1$ paths are elements 
	of~$\ccP$. Hence, the collection~$\ccC$ can be augmented by removing the at most~$M$ paths 
	whose blocks lie in~$\ccS$ and adding the~$M+1$ newly constructed paths instead. 
	As this contradicts the maximality of~$\ccC$, the assumption~\eqref{eq:1431} must have been
	false. This concludes the proof of Lemma~\ref{lem:indstep}.	
\end{proof}

Finally, we arrive at the main result of this section.

\begin{prop}\label{prop:1851}
	For every~$k\geq 3$ the statement~$\heartsuit_k$ holds.
\end{prop}

\begin{proof}
	We argue by induction on $k$, the base case being provided by Fact~\ref{f:herz3}. 
	The Lemmata~\ref{lem:herzzupik} and~\ref{lem:indstep} show 
	that~$\heartsuit_{k-1}\Rightarrow\spadesuit_{k-1}\Rightarrow\heartsuit_k$, 
	which is the induction step.
\end{proof}

\section{The proof of Theorem~\ref{t:main}}
\label{sec:main-pf}

The results in the foregoing sections routinely imply Theorem~\ref{t:main}, but for the 
sake of completeness we provide the details.

\begin{proof}[Proof of Theorem~\ref{t:main}]
	Given $k\ge 3$ and $\alpha>0$ we choose some auxiliary constants fitting into the hierarchy
	\begin{align}\label{eq:hierarchy}
		\alpha, k^{-1}\gg \mu\gg \beta ,\ell^{-1}\gg \zeta_{\star}\gg\vartheta_{\star}
		\gg\zeta_{\star\star}\gg\vartheta_{\star\star}\gg M^{-1}\gg n_0^{-1}\,,
	\end{align}
	where $\ell\ge 3$ is an odd integer and $M\equiv -1\pmod{k}$.
	
	Now let~$H=(V,E)$ be a~$k$-uniform hypergraph on~$n\ge n_0$ vertices satisfying the 
	minimum~$(k-2)$-degree condition~$\delta_{k-2}(H)\geq(\frac{5}{9}+\alpha)\frac{n^2}{2}$. 
	By Fact~$\ref{l:223}$ and~$\alpha\gg\mu\gg\beta, \ell^{-1}$ there exists 
	an~$(\alpha,\beta,\ell,\mu)$-constellation~$\Psi$ with underlying hypergraph $H$.
	
	\noindent {\bf Stage A.}
	We set aside a reservoir set~$\cR$ of size~$\vert\mathcal{R}\vert\leq\vartheta_{\star}^2n$ 
	provided by Proposition~\ref{prop:reservoir}. Let us recall that by 
	Corollary~\ref{lem:use-reservoir} and $\thetass\ll \thetas, k^{-1}, \ell^{-1}$ 
	\begin{enumerate}[label=\nlabel]
		\item\label{it:10a} 
				for every set~$\mathcal{R}'\subseteq\mathcal{R}$ of 
				at most $\thetass^2 n$ ``forbidden'' vertices, 
				every~$\zetass$-leftconnectable $(k-1)$-tuple~$\seq{a}$, 
				every~$\zetass$-rightconnectable $(k-1)$-tuple~$\seq{b}$
				that is disjoint to $\seq{a}$, and every $i\in [0, k)$, 
				there is an $\seq{a}$-$\seq{b}$-path through $\cR\setminus \cR'$ 
				with~$f(k,i,\ell)$ inner vertices.
	\end{enumerate}
	
	\noindent {\bf Stage B.}
	Next, we choose an absorbing path avoiding $\cR$. 
	More precisely, Proposition~\ref{prop:apl} yields a path~$P_A\subseteq H-\cR$ 
	with the properties that
	\begin{enumerate}[label=\nlabel, resume]
		\item $\vert V(P_A)\vert\leq\vartheta_{\star}n$,
		\item the starting and ending~$(k-1)$-tuple of~$P_A$ are~$\zetass$-connectable,
		\item\label{it:10d} and for every subset~$Z\subseteq V\setminus V(P_A)$ 
		with~$\vert Z\vert\leq 2\vartheta_{\star}^2n$ and~$\vert Z\vert\equiv 0\pmod{k}$, 
		there is a path~$Q\subseteq H$ with~$V(Q)=V(P_A)\cup Z$ having the same 
		end-$(k-1)$-tuples as~$P_A$.
	\end{enumerate}

	\noindent {\bf Stage C.}
	We proceed by covering almost all vertices belonging neither to $\cR$ nor to $P_A$
	by long paths. To this end we set~$X=\mathcal{R}\cup V(P_A)$ and consider the 
	constellation~$\Psi'=\Psi-X$. 
	Since~$\vert X\vert\le \vartheta_{\star}^2n+\vartheta_{\star}n\leq 2\vartheta_{\star}n$,
	Lemma~\ref{lem:1816} tells us that $\Psi'$ is an $\left(\frac{\alpha}{2}, \frac{\beta}{2}, \ell,
	2\mu \right)$-constellation. So the covering principle $\herz_k$ defined in Definition~\ref{d:herz}
	and proved in Proposition~\ref{prop:1851} applies to $\Psi'$, $2\zetass$ here in place of 
	$\Psi$, $\zetass$ there. In other words, in $\Psi'$ there exists a collection $\ccC$ of 
	mutually disjoint $M$-vertex paths whose end-tuples are $(2\zetass)$-connectable in $\Psi'$
	such that 
	\[
		\Big|V(\Psi')\setminus\bigcup_{P\in\ccC} V(P)\Big|
		\le 
		\thetas^2 n\,.
	\]
	Due to Fact~\ref{f:41}, the end-tuples of the paths in $\ccC$ are $\zetass$-connectable in $\Psi$.
	
	\noindent {\bf Stage D.}
	Now we want to connect the paths in $\ccC$ and $P_A$, thus obtaining one long path $T$
	with $\zetass$-connectable end-tuples. This is to be done by means of $|\ccC|$ connections 
	through the reservoir, iteratively using~\ref{it:10a} with $i=0$. Altogether these connections 
	require 
	\[
		|\ccC|f(k, 0, \ell) 
		\le 
		\frac{4^k\ell k n}{M}
		\le
		\thetass^2 n
	\]
	vertices from the reservoir. So $|\ccC|$ successive applications of~\ref{it:10a} 
	indeed allow us to construct this long path $T$ 
	(see Figure~\ref{fig:71}).
	
	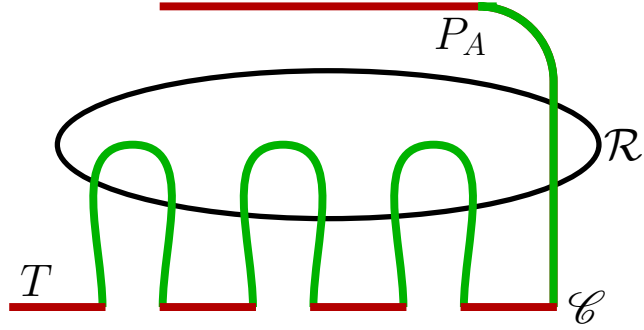
\begin{figure}[h!]
	\begin{tikzpicture}[scale=.8]
	
	\coordinate (x) at (0,0);
	
		\draw[black, line width=2pt] (5.3,2.7) ellipse (4.5cm and 35pt);	
		
	\foreach \i in {0,1,2,3}{
		\coordinate (x\i) at ($(x)+(\i*2.5,0)$);
		\coordinate (y\i) at ($(x\i)+(1.6,0)$);
		}
	
	\draw [line width= 3pt, color = green!70!black] ($(y0)-(.05,0)$) to [out = 90, in = 180]
	 +(.5,2.7)to [out = 0, in = 90] ($(x1)+(.05,0)$);
	
	\draw [line width= 3pt, color = green!70!black] ($(y1)-(.05,0)$) to [out = 90, in = 180]
	+(.5,2.7)to [out = 0, in = 90] ($(x2)+(.055,0)$);
	
	\draw [line width= 3pt, color = green!70!black] ($(y2)-(0.05,0)$) to [out = 90, in = 180]
	+(.5,2.7)to [out = 0, in = 90] ($(x3)+(.055,0)$);
	
		\draw [line width= 3pt, color = red!70!black, rounded corners =10mm] 
		($(y3)-(0.055,0)$) -- +(0,5) -- ($(x1)+(0,5)$);
		
		\draw [line width= 3pt, color = green!70!black, rounded corners =10mm] 
		($(y3)-(0.055,0)$) -- +(0,5) -- ($(y3)+(-1,5)$);
		
		\foreach \i in {0,1,2,3}{
			\draw [line width=3pt, color=red!70!black] (x\i) -- (y\i);
		}

	\node [above right ] at (x0) {\Large $T$};
	\node [right ] at (y3) {\Large $\ccC$};
	\node at (10.2,2.7){\Large $\cR$};
	\node at (7.5,4.5) {\Large $P_A$};

	\end{tikzpicture}
	\caption {The situation after Stage D.}
	\label{fig:71}
\end{figure}

	\noindent {\bf Stage E.}
	Moreover, we can still use~\ref{it:10a} 
	one more time in order to connect the end-tuples of $T$, thus creating one long cycle $C$. 
	For this last connection we use $f(k, i, \ell)$ inner vertices, where $i\in [0, k)$ is 
	determined by the congruence~$i \equiv n-|V(T)|\pmod{k}$. The current situation is depicted 
	in Figure~\ref{fig:72}.
	
	\begin{figure}[h]
	\begin{tikzpicture}[scale=1]
	
	\coordinate (x) at (0,0);
		
	\foreach \i in {0,1,2,3}{

		\coordinate (x\i) at ($(x)+(\i*2.5,0)$);
		\coordinate (y\i) at ($(x\i)+(1.6,0)$);
		}
		
			\draw[black, line width=2pt] (5.3,2.7) ellipse (4.5cm and 35pt);

	\coordinate (z1) at (6.5,3.5);
	\coordinate (z2) at (5.5,5);
	\coordinate (z3) at (4,5);
	
	\node at (6.2,5.2) {\Large $Z$};

\qedge{(z2)}{(z1)}{(z3)}{6pt}{1.5pt}{blue!70!black}{blue!50!white,opacity=0.1};
\foreach \i in {1,...,3} \fill (z\i) circle (2pt);

			\draw [line width= 3pt, color = green!70!black, rounded corners =10mm] 
			($(x1)+(0,4.5)$)--($(y0)+(0,4.5)$) --  +(2,-1.5)  -- ($(x0)+(0,3)$) -- ($(x0)+(0.055,0)$);
			
			\node [green!50!black] at (4,3.2){$f(k,i,\ell)$};
	
	\foreach \i in {0,1,2,3}{
		\draw [line width=3pt, color=red!70!black] (x\i) -- (y\i);
	}

	\draw [line width= 3pt, color = red!70!black] ($(y0)-(.05,0)$) to [out = 90, in = 180]
	 +(.5,2.7)to [out = 0, in = 90] ($(x1)+(.05,0)$);
	
	\draw [line width= 3pt, color = red!70!black] ($(y1)-(.05,0)$) to [out = 90, in = 180]
	+(.5,2.7)to [out = 0, in = 90] ($(x2)+(.055,0)$);
	
	\draw [line width= 3pt, color = red!70!black] ($(y2)-(0.05,0)$) to [out = 90, in = 180]
	+(.5,2.7)to [out = 0, in = 90] ($(x3)+(.055,0)$);

	\draw [line width= 3pt, color = red!70!black, rounded corners =10mm] 
		($(y3)-(0.055,0)$)-- +(0,4.5) -- ($(x1)+(0,4.5)$);% -- ++(1,-2) -- ($(x0)+(0,3)$) -- ($(x0)+(0.05,0)$);

	\node at ($(x0)+(.5,.5)$) {\Large $C$};
	\node [right] at (y3) {\Large $\ccC$};
	\node at (10.2,2.7){\Large $\cR$};
	\node at (7.8,4.9) {\Large $P_A$};

	\end{tikzpicture}
	\caption {The situation after Stage E. The dots in $Z$ represent sets of $k$ vertices each.}
	\label{fig:72}
\end{figure}

	Our choice of $i$ guarantees that the set $Z=V(\Psi)\setminus V(C)$ satisfies 
	\[
		|Z|\equiv n-|V(T)|-f(k, i, \ell)\equiv 0\pmod{k}\,.
	\]
	Furthermore, $Z$ has at most the size 	
	\[
		|Z| 
		\le 
		|\cR| + \Big|V(\Psi')\setminus\bigcup_{P\in\ccC} V(P)\Big|
		\le
		2\thetas^2 n\,.
	\]

	\noindent {\bf Stage F.}
	Taken together, the last two displayed formulae and~\ref{it:10d} show that~$Z$ can be 
	absorbed by $P_A$, i.e., that there exists 
	a path $Q$ with $V(Q)=V(P_A)\cup Z$ having the same end-tuples as~$P_A$. Upon replacing 
	the subpath~$P_A$ of~$C$ by~$Q$ we obtain the desired Hamiltonian cycle in~$H$ 
	(see Figure~\ref{fig:73}).  
\end{proof}

\begin{figure}[h]
	\begin{tikzpicture}[scale=1]
	
	\coordinate (x) at (0,0);
	\foreach \i in {0,1,2,3}{

		\coordinate (x\i) at ($(x)+(\i*2.5,0)$);
		\coordinate (y\i) at ($(x\i)+(1.6,0)$);
		}
		
			\draw[black, line width=2pt] (5.3,2.7) ellipse (4.5cm and 35pt);

	\coordinate (z1) at (6.5,3.5);
	\coordinate (z2) at (5.5,5.5);
	\coordinate (z3) at (4,5.5);

	\node at (3.3,5.4) {\Large $Q$};

			\draw [line width= 3pt, color = red!70!black, rounded corners =10mm] 
			($(x1)+(0,5)$)--($(y0)+(0,5)$) --  +(2,-2)  -- ($(x0)+(0,3)$) -- ($(x0)+(0.055,0)$);
			
	%		\node [green!50!black] at (4,3.2){$f(k,i,\ell)$};
	
	\foreach \i in {0,1,2,3}{
		\draw [line width=3pt, color=red!70!black] (x\i) -- (y\i);
	}

	\draw [line width= 3pt, color = red!70!black] ($(y0)-(.05,0)$) to [out = 90, in = 180]
	 +(.5,2.7)to [out = 0, in = 90] ($(x1)+(.05,0)$);
	
	\draw [line width= 3pt, color = red!70!black] ($(y1)-(.05,0)$) to [out = 90, in = 180]
	+(.5,2.7)to [out = 0, in = 90] ($(x2)+(.055,0)$);
	
	\draw [line width= 3pt, color = red!70!black] ($(y2)-(0.05,0)$) to [out = 90, in = 180]
	+(.5,2.7)to [out = 0, in = 90] ($(x3)+(.055,0)$);

	\draw [line width= 3pt, color = red!70!black, rounded corners =10mm] 
		($(y3)-(0.055,0)$)-- +(0,5) -- ($(x3)+(0,5)$);% -- ++(1,-2) -- ($(x0)+(0,3)$) -- ($(x0)+(0.05,0)$);

		\draw [line width= 3pt, color = red!70!black] 
		($(x1)+(0,5)$) -- (3.6,5) to [out = 0, in  = 180] (z3) to [out =0, in= 180] (4.4,5) 
		--(5.1,5) to [out = 0, in  = 180] (z2) to [out =0, in= 180] (5.9,5)
		--(6.1,5) to [out = 0, in  = 180] (z1) to [out =0, in= 180] (6.9,5) -- ($(x3)+(0,5)$);
		
		\draw [line width= 3pt, color = green!70!black] 
		(3.6,5) to [out = 0, in  = 180] (z3) to [out =0, in= 180] (4.4,5) ;
		
			\draw [line width= 3pt, color = green!70!black] 
			(5.1,5) to [out = 0, in  = 180] (z2) to [out =0, in= 180] (5.9,5);
			
				\draw [line width= 3pt, color = green!70!black] 
				(6.1,5) to [out = 0, in  = 180] (z1) to [out =0, in= 180] (6.9,5) ;

%	\node [above right ] at (x0) {\Large $C$};
%	\node [right ] at (y3) {\Large $\ccC$};
	\node at (10.2,2.7){\Large $\cR$};
%	\node at (8,4.5) {\Large $P_A$};

	\foreach \i in {1,...,3} \fill (z\i) circle (3pt);

	\end{tikzpicture}
	\caption {The situation after Stage F.}
	\label{fig:73}
\end{figure}


\begin{bibdiv}
\begin{biblist}
\bib{AR}{article}{
   author={Alon, Noga},
   author={Ruzsa, Imre Z.},
   title={Non-averaging subsets and non-vanishing transversals},
   journal={J. Combin. Theory Ser. A},
   volume={86},
   date={1999},
   number={1},
   pages={1--13},
   issn={0097-3165},
   review={\MR{1682960}},
   doi={10.1006/jcta.1998.2926},
}

\bib{Mathias}{article}{
	author={Ara\'{u}jo, Pedro},
	author={Piga, Sim\'{o}n}, 
	author={Schacht, Mathias}, 
	title={Localised codegree conditions for tight Hamilton cycles in $3$-uniform hypergraphs}, 
	eprint={2005.11942},
	note={Submitted},
}

\bib{BR}{article}{
	author={Blakley, G. R.},
	author={Roy, Prabir},
	title={A H\"older type inequality for symmetric matrices with nonnegative
					entries},
	journal={Proc. Amer. Math. Soc.},
	volume={16},
	date={1965},
	pages={1244--1245},
	issn={0002-9939},
	review={\MR{0184950}},
}

\bib{Chv}{article}{
   author={Chv\'{a}tal, V.},
   title={On Hamilton's ideals},
   journal={J. Combinatorial Theory Ser. B},
   volume={12},
   date={1972},
   pages={163--168},
   issn={0095-8956},
   review={\MR{294155}},
   doi={10.1016/0095-8956(72)90020-2},
}

\bib{Dirac}{article}{
   author={Dirac, G. A.},
   title={Some theorems on abstract graphs},
   journal={Proc. London Math. Soc. (3)},
   volume={2},
   date={1952},
   pages={69--81},
   issn={0024-6115},
   review={\MR{47308}},
   doi={10.1112/plms/s3-2.1.69},
}

\bib{E}{article}{
   author={Erd\H{o}s, P.},
   title={On extremal problems of graphs and generalized graphs},
   journal={Israel J. Math.},
   volume={2},
   date={1964},
   pages={183--190},
   issn={0021-2172},
   review={\MR{183654}},
   doi={10.1007/BF02759942},
}
	
\bib{ES}{article}{
   author={Erd\H{o}s, Paul},
   author={Simonovits, Mikl\'{o}s},
   title={Supersaturated graphs and hypergraphs},
   journal={Combinatorica},
   volume={3},
   date={1983},
   number={2},
   pages={181--192},
   issn={0209-9683},
   review={\MR{726456}},
   doi={10.1007/BF02579292},
}
	
\bib{Fiedler}{book}{
   author={Fiedler, Miroslav},
   title={Theory of graphs and its applications},
   series={Proceedings of the Symposium held in Smolenice in June 1963},
   publisher={Publishing House of the Czechoslovak Academy of Sciences,
   Prague},
   date={1964},
   pages={234},
   review={\MR{0172259}},
}
	
		
\bib{Fitch}{article}{
   author={Fitch, Matthew},
   title={Rational exponents for hypergraph Tur\'{a}n problems},
   journal={J. Comb.},
   volume={10},
   date={2019},
   number={1},
   pages={61--86},
   issn={2156-3527},
   review={\MR{3890916}},
   doi={10.4310/joc.2019.v10.n1.a3},
}

\bib{HS}{article}{
   author={Hajnal, A.},
   author={Szemer\'{e}di, E.},
   title={Proof of a conjecture of P. Erd\H{o}s},
   conference={
      title={Combinatorial theory and its applications, II},
      address={Proc. Colloq., Balatonf\"{u}red},
      date={1969},
   },
   book={
      publisher={North-Holland, Amsterdam},
   },
   date={1970},
   pages={601--623},
   review={\MR{0297607}},
}
	
\bib{HZ16}{article}{
   author={Han, Jie},
   author={Zhao, Yi},
   title={Forbidding Hamilton cycles in uniform hypergraphs},
   journal={J. Combin. Theory Ser. A},
   volume={143},
   date={2016},
   pages={107--115},
   issn={0097-3165},
   review={\MR{3519818}},
   doi={10.1016/j.jcta.2016.05.005},
}

\bib{KaKi}{article}{
   author={Katona, Gyula Y.},
   author={Kierstead, H. A.},
   title={Hamiltonian chains in hypergraphs},
   journal={J. Graph Theory},
   volume={30},
   date={1999},
   number={3},
   pages={205--212},
   issn={0364-9024},
   review={\MR{1671170}},
   doi={10.1002/(SICI)1097-0118(199903)30:3$<$205::AID-JGT5$>$3.3.CO;2-F},

}

\bib{Lang}{article}{
	author={Lang, Richard},
	author={Sanhueza-Matamala, Nicolás}, 
	title={Minimum degree conditions for tight Hamilton cycles}, 
	eprint={2005.05291},
	note={Submitted},
}

\bib{Lee}{article}{
	author={Lee, Joonkyung}, 
	title={On some graph densities in locally dense graphs}, 
	eprint={1707.02916},
	note={Random Structures \& Algorithms. To Appear},
}

\bib{Y}{article}{
	author={Polcyn, Joanna},
   author={Reiher, Chr.},
   author={R\"{o}dl, Vojt\v{e}ch},
   author={Ruci\'{n}ski, Andrzej},
   author={Schacht, Mathias},
   author={Sch\"ulke, Bjarne},
   title={Minimum pair-degree condition for tight Hamiltonian cycles in $4$-uniform hypergraphs},
   eprint={2005.03391},
	note={Acta Mathematica Hungarica. To Appear},
}

\bib{Posa}{article}{
   author={P\'{o}sa, L.},
   title={A theorem concerning Hamilton lines},
   language={English, with Russian summary},
   journal={Magyar Tud. Akad. Mat. Kutat\'{o} Int. K\"{o}zl.},
   volume={7},
   date={1962},
   pages={225--226},
   issn={0541-9514},
   review={\MR{184876}},
}

\bib{R}{article}{
   author={Reiher, Chr.},
   author={R\"{o}dl, Vojt\v{e}ch},
   author={Ruci\'{n}ski, Andrzej},
   author={Schacht, Mathias},
   author={Szemer\'{e}di, Endre},
   title={Minimum vertex degree condition for tight Hamiltonian cycles in
   3-uniform hypergraphs},
   journal={Proc. Lond. Math. Soc. (3)},
   volume={119},
   date={2019},
   number={2},
   pages={409--439},
   issn={0024-6115},
   review={\MR{3959049}},
   doi={10.1112/plms.12235},
}

\bib{RRS}{article}{
   author={R\"{o}dl, Vojt\v{e}ch},
   author={Ruci\'{n}ski, Andrzej},
   author={Szemer\'{e}di, Endre},
   title={An approximate Dirac-type theorem for $k$-uniform hypergraphs},
   journal={Combinatorica},
   volume={28},
   date={2008},
   number={2},
   pages={229--260},
   issn={0209-9683},
   review={\MR{2399020}},
   doi={10.1007/s00493-008-2295-z},
}

\bib{Schuelke}{article}{
	title={A pair-degree condition for Hamiltonian cycles in $3 $-uniform hypergraphs},
	author={Sch{\"u}lke, Bjarne},
	eprint={1910.02691},
	note={Submitted},
}

\bib{Sid}{article}{
   author={Sidorenko, Alexander},
   title={A correlation inequality for bipartite graphs},
   journal={Graphs Combin.},
   volume={9},
   date={1993},
   number={2},
   pages={201--204},
   issn={0911-0119},
   review={\MR{1225933}},
   doi={10.1007/BF02988307},
}
	
\bib{Sim}{article}{
   author={Simonovits, Mikl{\'o}s},
   title={Extremal graph problems, degenerate extremal problems, and
   supersaturated graphs},
   conference={
      title={Progress in graph theory},
      address={Waterloo, Ont.},
      date={1982},
   },
   book={
      publisher={Academic Press, Toronto, ON},
   },
   date={1984},
   pages={419--437},
   review={\MR{776819}},
}

\bib{ST}{article}{
   author={Staden, Katherine},
   author={Treglown, Andrew},
   title={On degree sequences forcing the square of a Hamilton cycle},
   journal={SIAM J. Discrete Math.},
   volume={31},
   date={2017},
   number={1},
   pages={383--437},
   issn={0895-4801},
   review={\MR{3615461}},
   doi={10.1137/15M1033101},
}
	
\bib{Sz}{article}{
   author={Szemer\'{e}di, Endre},
   title={Is laziness paying off? (``Absorbing'' method)},
   conference={
      title={Colloquium De Giorgi 2010--2012},
   },
   book={
      series={Colloquia},
      volume={4},
      publisher={Ed. Norm., Pisa},
   },
   date={2013},
   pages={17--34},
   review={\MR{3089025}},
   %doi={10.1007/978-88-7642-457-1_3},
}

\bib{Tr}{article}{
   author={Treglown, Andrew},
   title={A degree sequence Hajnal-Szemer\'{e}di theorem},
   journal={J. Combin. Theory Ser. B},
   volume={118},
   date={2016},
   pages={13--43},
   issn={0095-8956},
   review={\MR{3471843}},
   doi={10.1016/j.jctb.2016.01.007},
}

\end{biblist}
\end{bibdiv}
\end{document}